\documentclass{amsart}
\usepackage{amsmath,amssymb,amsthm,graphicx,setspace,verbatim,fullpage,bbm,nicefrac, hyperref, subcaption, caption}
\usepackage[usenames]{xcolor}
\usepackage{amsfonts}
\newtheorem{theorem}{Theorem}
\newtheorem{conjecture}{Conjecture}
\newtheorem{lemma}{Lemma}
\newtheorem{question}{Question}
\newtheorem{claim}{Claim}

\newtheorem{fact}{Fact}
\theoremstyle{remark}

\theoremstyle{definition}
\newtheorem{case}{Case}
\newtheorem{definition}{Definition}
\newtheorem{observation}{Observation}
\newtheorem{subcase}{Subcase}
\numberwithin{subcase}{case}
\newtheorem{sscase}{Subcase}

\numberwithin{sscase}{subcase}

\newcommand{\ba}{\backslash}
\newcommand{\R}{\mathbb{R}}

\makeatletter
\def\Ddots{\mathinner{\mkern1mu\raise\p@
\vbox{\kern7\p@\hbox{.}}\mkern2mu
\raise4\p@\hbox{.}\mkern2mu\raise7\p@\hbox{.}\mkern1mu}}
\makeatother

\begin{document}

\title{A Proof of the Cycle Double Cover Conjecture}
\author{Mary Radcliffe}

\begin{abstract}
Given a bridgeless graph $G$, the Cycle Double Cover Conjecture posits that there is a list of cycles of $G$, such that every edge appears in exactly two cycles. This conjecture was originally posed independently in 1973 by Szekeres and 1979 by Seymour. We here present a proof of this conjecture by analyzing certain kinds of cycles in the line graph of $G$. Further, in the case that $G$ is 3-regular, we prove the stronger conjecture that given a bridgeless graph $G$ and a cycle $C$ in $G$, then there exists a cycle double cover of $G$ containing $C$.
\end{abstract}

\maketitle

\section{Introduction}\label{S:intro}

The Cycle Double Cover Conjecture (CDCC) was originally posed independently by Szekeres \cite{szekeres1973polyhedral} in 1973 and Seymour \cite{seymour1979sums} in 1979. The conjecture is as follows:

\begin{conjecture}[CDCC]
If $G$ is a bridgeless graph, then there is a list of cycles $\mathcal{C}$ in $G$ such that every edge appears in exactly two cycles in $\mathcal{C}$.
\end{conjecture}

Such a list of cycles is typically referred to as a cycle double cover. Much effort has been spent on resolving this conjecture, and an excellent history of approaches to the problem can be found in several survey papers \cite{chan2009survey, jaeger1985survey}. In \cite{jaeger1985survey}, it is shown that it is sufficient to prove that every 3-regular graph has a cycle double cover, and it is this theorem that we prove in this work. Specifically, we show

\begin{theorem}
If $G$ is a cubic, bridgeless graph, then there is a list of cycles $\mathcal{C}$ in $G$ such that every edge appears in exactly two cycles in $\mathcal{C}$.
\end{theorem}

In \cite{cai1992cycle}, an approach to the CDCC is considered in which, rather than find cycle double covers in the graph $G$, one can instead find cycle double covers in the line graph $L(G)$. We here adapt this technique, and show that it is sufficient to produce a cycle decomposition in $L(G)$, rather than a double cover, where the cycles in the decomposition satisfy some particular properties. Using this adapted technique, we provide a full proof that every 3-regular, bridgeless graph has a cycle double cover.

Specifically, rather than consider the line graph alone, we color the edges of the line graph $L(G)$ with the vertices $V(G)$, so that for each $v\in V(G)$ there exists a monochromatic triangle in $L(G)$ of color $v$. We then note that cycles in $G$ correspond directly to rainbow cycles in $L(G)$, and hence a rainbow cycle decomposition in $L(G)$ corresponds to a cycle double cover in $G$. We note here that this is only true in the case of a 3-regular graph $G$, as for a 3-regular graph, any rainbow cycle decomposition in $L(G)$ will use each vertex of $L(G)$ exactly twice; that is, such a decomposition will use each edge of $G$ exactly twice.

The proof is obtained by induction on a larger class of edge-colored graphs, of which the line graph of any bridgeless cubic graph is a member. This class of graphs is formally defined in Section \ref{S:setup} below. Fundamentally, the only barrier to producing a cycle decomposition in this larger class of graphs will be a vertex that behaves like a bridge: it is a cut vertex, and has all its edges of one color into one block, and all edges of another color into another block. It is in forbidding this type of vertex that the work of the proof is found.

As we will note in Section \ref{S:conclusions}, our proof technique in fact resolves a stronger conjecture, due to Goddyn, in the case of 3-regular graphs. Specifically, Goddyn's conjecture is as follows.

\begin{conjecture}\label{C:Goddyn}
If $G$ is a bridgeless graph, and $C$ is a cycle in $G$, then there exists a cycle double cover of $G$ containing $C$.
\end{conjecture}

Our technique will immediately yield the following partial resolution to Conjecture \ref{C:Goddyn}.

\begin{theorem}
If $G$ is a cubic bridgeless graph, and $C$ is a cycle in $G$, then there exists a cycle double cover of $G$ containing $C$.
\end{theorem}

This paper is organized as follows. First, in Section \ref{S:defs}, we introduce the basic notations, concepts, and language we shall require. In Section \ref{S:outline}, we outline the proof of the CDCC, and reduce to an equivalent condition on a particular class of edge-colored graphs (Theorem \ref{T:mainthm}). In Section \ref{S:proof}, we prove Theorem \ref{T:mainthm}. Finally, in Section \ref{S:conclusions}, we conclude with some generalizations and conjectures that might be addressed using a similar technique.

\section{Definitions and Notations}\label{S:defs}

We here outline the basic tools and notations we shall require for this proof. Any language or notation not defined in this section, as well as any basic facts and observations made, can be found in \cite{chartrand2010graphs}.

As our primary object of study here will be cycles, we make a note on the notation used to describe cycles. Primarily, we shall write a path as $P=u_1, u_2, \dots, u_j$, and a  cycle as $C=(v_1, v_2, \dots, v_k, v_1)$, where the $v_i$ are the vertices involved. At times, we shall use the notation $C=(v_1, v_2, \dots, v_i, u_1, P, u_j, \dots, v_k, v_1)$ to indicate the cycle $C=(v_1, v_2, \dots, v_i, u_1, u_2, \dots, u_j, \dots, v_k, v_1)$. That is, inserting $P$ into the cycle implies that we take all internal vertices of $P$ as members of the cycle. In these cases, we include the endpoints $u_1$ and $u_j$ in the cycle definition to indicate the direction along which the path is followed.

On some occasions, we may refer to a cycle by its edges, in the form $C=(e_1, e_2, \dots, e_n)$. We shall routinely abuse notation and write $e\in C$ to indicate that the edge $e$ appears on the cycle $C$, even if the cycle is presented in vertex notation. We shall also occasionally write $C\subset E(G)$ to describe the cycle, as $C$ can be uniquely defined by the set of edges that appear in $C$.

Given a graph $G$ and a subgraph $H\subset G$, we write $G\ba H$ to indicate the subgraph of $G$ defined by $V(G\ba H)=V(G)$ and $E(G\ba H)=E(G)\ba E(H)$. Given a vertex $v$, we write $G\ba \{v\}$ as the subgraph of $G$ obtained by deleting $v$ from $V(G)$, and deleting all edges from $E(G)$ that include the vertex $v$.

In an edge-colored graph $G$, given a color $c$, we define the {\it color class} of $G$ corresponding to $c$ to be the set of edges in $G$ having color $c$. In this way, the color classes partition the edges of $G$. We will at times refer to the color class for a given color as a subgraph of $G$, by taking the subgraph having exactly these edges. A subgraph of $G$ is called {\it rainbow} if no two edges in the subgraph belong to the same color class. A subgraph of $G$ is called {\it monochromatic} if every edge in the subgraph belongs to the same color class. Throughout, if $G$ is a colored graph, and $H$ is a subgraph of $G$, we will assume that $H$ is also colored by the induced coloring from $G$.

Given a graph $G$, we define the colored line graph of $G$ to be the edge-colored graph $L=L(G)$, having
\begin{itemize}
\item $V(L) = E(G)$
\item $e\sim_{L} f$ if and only if $e$ and $f$ share a vertex in $G$
\item The edges of $L$ are colored by $V(G)$, where $c(ef)=v$ whenever $e$ and $f$ share the vertex $v$ in $G$.
\end{itemize}

Note that in any line graph colored in this way, the color classes form cliques in $L$, and each vertex of $L$ is a member of exactly two such cliques. In particular, since $G$ is 3-regular, the color classes here will be copies of $K_3$. Moreover, each vertex in $L$ will be a member of exactly two distinct color classes, so each vertex in $L$ is a member of exactly two monochromatic triangles, and has degree 4. Given any edge-colored graph, we shall say that $v$ has {\it color degree} $k$ if the number of colors assigned to the incident edges of $v$ is exactly $k$. Hence, in a line graph, every vertex has color degree 2. As for our purposes, the line graph of $G$ will always be colored, we shall omit the adjective ``colored'' and simply write ``line graph'' to mean the edge-colored line graph. We first make the following simple observation regarding triangles in $G$.

\begin{observation}\label{monos}
If $L$ is the colored line graph of a cubic graph $G$, and $G$ has no triangles, then the only triangles in $L$ are the color classes.
\end{observation}

The following lemma is a key ingredient for the proof of the CDCC.

\begin{lemma}\label{ltog}
Let $C=(e_1, e_2, \dots, e_k)$. Then $C$ is a rainbow cycle in $L$ if and only if $C$ is a cycle in $G$.
\end{lemma}

\begin{proof}
For the forward direction, we need only show that in $G$, the cycle $C$ does not reuse any vertex. Let $v_i = c(e_ie_{i+1})$, where $i$ is taken modulo $k$. Then we may write $C$ in its edge form in $L$ as $C=(v_1, v_2, \dots, v_k)$. Moreover, as $C$ is rainbow, the set $\{v_1, \dots, v_k\}$ has $k$ distinct colors. But then it is clear that $(v_1, v_2, \dots, v_k)$ is a cycle in $G$, having edges $e_1, e_2, \dots, e_k$, and hence every rainbow cycle in $L$ is also a cycle in $G$.

The other direction is similar.

The other direction is similar.
\end{proof}

Combining Lemma \ref{ltog} with Observation \ref{monos}, we have that if $G$ is a cubic graph, then the only triangles in $L$ are either monochromatic or rainbow. This piece of structure will be fundamental to our main proof.

\subsection{Contractions, subdivisions, and cuts}

Throughout the main proof, we shall frequently make use of ideas related to topological structure in graphs. Indeed, the proof will hinge on the analysis of a certain kind of cut-vertex. We here outline the basic concepts and observations that will feature in the proof.

Let $G$ be a connected graph. We say that $S\subset V(G)$ is a {\it cut-set} of $G$ if the graph $G\ba S$ is disconnected. If $S$ consists of a single vertex $v$, we say that $v$ is a {\it cut vertex} of $G$. We recall the basic fact that a vertex $v$ is a cut-vertex of $G$ if and only if there exists a pair of vertices $u, w\in V(G)$, with $u, w\neq v$, such that every path from $u$ to $w$ includes the vertex $v$.

A graph $G$ is called {\it 2-connected} if $G$ has no cut vertices. Equivalently, $G$ is 2-connected if for every pair of vertices $u, w\in V(G)$, there exist at least two paths from $u$ to $w$ that share no vertices other than the endpoints (such paths are called {\it internally vertex disjoint}). 

A {\it block} in a graph $G$ is a maximal 2-connected subgraph of $G$. One can view a block as an induced 2-connected subgraph $H$, in which the addition of any vertex to $V(H)$ yields a graph that is not 2-connected. As a result, the edges of $G$ can be partitioned uniquely into subsets $E_1, E_2, \dots, E_k$, such that each subset $E_i$ induces a block, and the corresponding blocks are edge-disjoint. We note that any two blocks in $G$ can share at most one vertex, and this vertex must be a cut vertex. Let $G_1, G_2, \dots, G_k$ be the unique blocks of $G$. The {\it block graph} of $G$ is defined as the graph $B=B(G)$ having $V(B)=\{G_1, G_2, \dots, G_k\}$, and $G_i\sim_BG_j$ if and only if $V(G_i)\cap V(G_j)\neq\emptyset$. We recall the following basic fact.

\begin{fact} Given a connected graph $G$, the block graph $B(G)$ is a tree.\end{fact}

We shall frequently wish to focus our analysis on a single cut vertex, even though the graph may have more than one cut vertex. In this case, we use the following language. Given a cut vertex $v$, define the {\it pseudoblocks} of $G$ corresponding to $v$ to be subgraphs $G_1$ and $G_2$, having $V(G_1)\cap V(G_2)=\{v\}$, $E(G_1)\cup E(G_2)=E(G)$, and $E(G_1)\cap E(G_2)=\emptyset$. That is, we divide $G$ into two subgraphs, such that these subgraphs share the vertex $v$, but share no edges. We note that this can be done if and only if $v$ is a cut vertex; if we consider the block graph $B$, we may obtain the subgraphs $G_1$ and $G_2$ by taking $G_1$ as the union of a chosen block containing $v$ with all of its descendents, and $G_2$ as the union of all remaining blocks. As such, a pseudoblock decomposition is therefore not unique, but still satisfies the property that there will be no edges between the pseudoblocks.

We also routinely use the following basic fact about even graphs, a standard exercise in graph theory.
\begin{fact} If $G$ is an even graph, and $v$ has degree 2, then $v$ is not a cut vertex of $G$.
\end{fact}

Similarly, an edge $e\in E(G)$ is called a {\it bridge} of $G$ if $G\ba\{e\}$ is disconnected. As above, we have the following fact, another standard exercise.
\begin{fact}If $G$ is an even graph, then $G$ has no bridge.
\end{fact}

Let $G$ be a graph, and let $e=\{u, v\}\in E(G)$. The {\it contraction of $G$ along $e$} is the graph $H$ defined by $V(H)=V(G)\ba \{u, v\}\cup\{x\}$, where $x$ is a new vertex not found in $G$, and \[E(H)=E(G)\ba\{(y, z)\ |\ y=u\hbox{ or }v,\hbox{ or }z=u\hbox{ or }v\}\cup\{(x, y)\ | u\sim y\hbox{ or }v\sim y\}.\] That is to say, we create $H$ by removing the vertices $u, v$, and replacing them with a new vertex $x$, having as its neighbors all the neighbors of $u$ or $v$. We note that if $u$ and $v$ have a common neighbor, this can create multiple edges in $H$; this will not come up in our proof.

Given a graph $G$ and a subgraph $R$, the contraction of $G$ along $R$ is the graph $H$ obtained by contracting (in any order) all the edges in $R$. More generally, we refer to $H$ as a contraction of $G$. We may view $H$ as obtained by partitioning the vertices of $G$ into subsets $S_1, S_2, \dots, S_k$, and placing an edge between $S_i$ and $S_j$ if and only if there are vertices $u\in S_i$ and $v\in S_j$ such that $u\sim_Gv$. We shall write $[u]$ to denote the subset of this partition containing $u$. Then we have the following observation.

\begin{observation}\label{contraction}
Suppose that $G$ is a connected graph, and $H$ is obtained from $G$ by the contraction of a triangle in $G$, such that this triangle is not a block of $G$. Suppose that $x$ is a cut vertex of $G$. Then $[x]$ is a cut vertex of $H$.
\end{observation}

\begin{proof}
Let $v_1, v_2, v_3$ be the vertices of the triangle in $G$ that is contracted to form $H$. Let $G_1, G_2, \dots, G_k$ be the block decomposition of $H$. Then since the triangle $v_1, v_2, v_3$ is 2-connected, we have that there exists a block, say $G_1$, having $v_1, v_2, v_3\in V(G_1)$. Moreover, $G_1$ also contains at least one other vertex. Then $H$ has entirely the same structure as $G$, except that $G_1$ is replaced by the contraction of $G_1$ along the triangle $(v_1, v_2, v_3)$, and this contraction is not a single vertex. Hence, any cut vertex $x$ other than $v_1, v_2, v_3$ is still a cut vertex after the contraction. If one of $v_1, v_2, v_3$ is a cut vertex, say $v_1$, then for every vertex $x\neq v_1$ in $G_1$, and every vertex $y\neq v_1$ that is not in $G_1$, every path from $x$ to $y$ passes through $v_1$. Let $x\in V(G_1)$, with $x\neq v_1, v_2, v_3$. Then in $H$, every path from $x$ to $y$ passes through $[v_1]$, and hence $[v_1]$ is a cut vertex in $H$.
\end{proof}

Let $G$ be a graph, and let $e=\{u,v\}\in E(G)$. The {\it subdivision} of $G$ at $e$ is the graph $H$ defined by $V(H)=V(G)\cup\{x\}$, where $x$ is a new vertex not found in $G$, and $E(H) = E(G)\ba\{u, v\}\cup\{u, x\}\cup\{v, x\}$. That is, $H$ is obtained from $G$ by removing the edge $e$ and replacing it with a length 2 path $uxv$. We have the following observation.

\begin{observation}\label{subdivide}
If $H$ and $G$ are both even graphs, and $H$ is obtained from $G$ by subdividing one edge, then $v$ is a cut vertex of $G$ if and only if $v$ is a cut vertex of $H$.
\end{observation}

Note that this observation does not hold in general; if a graph is not even, then it could have a bridge, in which case subdividing the bridge produces a new cut vertex. However, in the case of an even graph, since no bridge may be present the observation holds.

\section{Proof outline and main ingredients}\label{S:outline}

In order to prove the main theorem, we first use Lemma \ref{ltog} to obtain the following equivalent condition.

\begin{lemma}\label{equiv}
A 3-regular graph $G$ has a cycle double cover if and only if its line graph $L$ has a decomposition of its edges into rainbow cycles. 
\end{lemma}

\begin{proof}
If $L$ has a decomposition into rainbow cycles, using Lemma \ref{ltog} gives the result immediately. As each edge of $L$ is used exactly once, and $L$ is 4-regular, we will thus have that each vertex of $L$ appears in precisely two cycles in the decomposition, i.e., each edge of $G$ appears in precisely two cycles in the decomposition.

For the other direction, suppose that $G$ has a cycle double cover $\mathcal{C}$. Consider a vertex $v$ with incident edges $vu$, $vw$, and $vx$. There will be precisely three cycles in $\mathcal{C}$ that include the vertex $v$, say $C_1, C_2, C_3$. Moreover, we have that each pairing $\{vu, vw\}$, $\{vw, vx\}$ and $\{vx, vu\}$ appears in exactly one of these cycles. Moreover, these pairings are precisely the edges of color $v$ in $L$, and thus, the corresponding rainbow cycles in $L$ use each edge with color $v$ exactly once, and no other cycle from $\mathcal{C}$ induces a cycle in $L$ including the color $v$. Hence, the cycles $\mathcal{C}$ are a rainbow cycle decomposition of the edges of $L$.
\end{proof}

Hence, it suffices to show that every line graph of a 3-regular graph has a rainbow cycle decomposition. In fact, we shall prove something slightly more general. We shall define a class of graphs as {\it good} if they satisfy a collection of characteristics that will always be satisfied by line graphs of 3-regular graphs. In this way, we can inductively find rainbow cycles in these graphs, remove them, and provided that the resulting structure is good, find more. The basic structure of the proof will be as follows. First, we define good graphs by isolating the characteristics of line graphs that are important to the cycle decomposition. We then observe some of the basic properties of such graphs, and develop terminology to discuss them. We then show by induction that every good graph $G$ has a rainbow cycle decomposition. In order to do so, we shall focus on a piece of local structure in $G$, and show that there is always a way to define a strictly smaller good (or almost-good) graph by deleting or contracting edges, or removing vertices and rewiring their edges. The proof will be by cases, depending on which particular local structures are present in $G$, and the structures of graphs obtained by the manipulation of the local structure in $G$. 

Fundamentally, the only obstruction we can have to finding a rainbow cycle in $G$ or a subgraph of $G$ is what we shall term a cut vertex of Type $X$: a cut vertex that has all of its edges of one color into one pseudoblock, and all of its edges of the other color into the other pseudoblock. Formally, we have the following definition.

\begin{definition} Let $G$ be an even edge-colored graph with maximum degree at most 4. We say a vertex $v$ is a we define a vertex $v$ to be a {\it cut vertex of type $X$} if the following conditions are met:
\begin{itemize}
\item $v$ is a cut vertex.
\item $v$ has degree 4.
\item There exist pseudoblocks $G_1$ and $G_2$ at $v$ such that $v$ has two edges into $G_1$ and two edges into $G_2$, and the edges in to each pseudoblock have the same color.
\end{itemize}
\end{definition}

A cut vertex of type $X$ is shown in Figure \ref{F:cvtypeX}.

\begin{figure}[htp]
\includegraphics{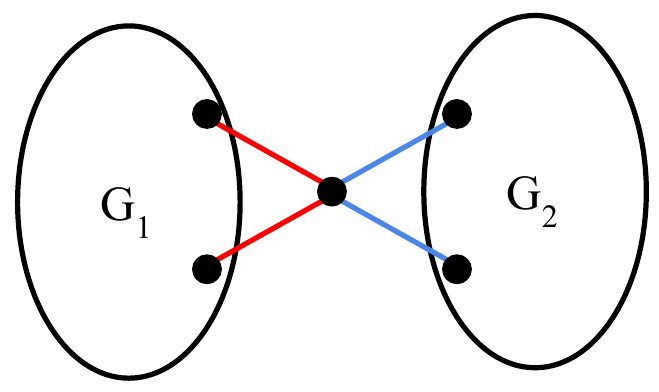}
\caption{A cut vertex of Type $X$. Note that there can be no rainbow cycles involving the central vertex, as necessarily a cycle must remain entirely in one of pseudoblocks $G_1$ or $G_2$.}
\label{F:cvtypeX}
\end{figure}

Notice that this is a clear obstacle to a rainbow cycle decomposition. If $v$ is a cut vertex of Type $X$, then the only cycles involving $v$ are those that remain entirely within the pseudoblocks $G_1$ and $G_2$, and hence must use two edges incident to $v$ in the same color class. Hence, $v$ and its incident edges cannot be members of any rainbow cycles. We note also that in the case that $v$ is a cut vertex of Type $X$ in an even graph $G$, we must have that there is a single component of $G_1\ba \{v\}$ such that both edges from $v$ into $G_1$ have their other endpoint at a vertex in this component; if this is not true, then these edges would be bridges, and as mentioned above, no even graph may have a bridge.

Essentially, as our proof will show, this is the only obstacle to a rainbow cycle decomposition in $G$. We note that if $G$ is itself the line graph of a cubic graph, a cut vertex of Type $X$ in $G$ corresponds to a bridge in in the original graph. Hence, we are isolating the ``bridge-like'' structure, and expressly forbidding it as we construct our decomposition. The difficulty of the proof lies not in finding a rainbow cycle in a good graph, but showing that one can always find a rainbow cycle such that upon its removal, there is not any cut vertex of Type $X$.

\subsection{Main ingredients}\label{S:setup}

In this section, we define our fundamental structures, and prove the key lemmas that will allow us to prove the main theorem. We first begin with a full definition of a good colored graph. Throughout, given a colored graph $G$, we shall use $c:E(G)\to \R$ to denote the coloring on the edges, even if such function has not explicitly been defined. 

\begin{definition}
Given an edge-colored graph $G$, we say $G$ is a {\it good} colored graph if the following conditions are met.
\begin{enumerate}
\item Every vertex of $G$ has even degree.\label{inheritb}
\item $G$ has maximum degree at most 4.
\item Every triangle in $G$ is either rainbow or monochromatic. \label{monotri}
\item Every nonisolated vertex of $G$ has color degree 2. \label{bender}
\item The subgraph induced by each color class has at most three vertices.\label{inherite}
\item $G$ has no cut-vertices of type $X$. \label{nox}
\end{enumerate}
\end{definition}

Note that these conditions force that every vertex of $G$ has at least two incident colors, and moreover, no color appears incident to $v$ more than two times (or else there would be more than three vertices incident to a given color.). This implies, together with property \eqref{inherite}, that the subgraph induced by each color class is either $K_2$, $P_2$, or $K_3$; that is, either a single edge, a path of length two, or a 3-clique. Further, we classify the vertices of a good colored graph into two types: Type I, a vertex of degree 2, having its two edges of different colors, and Type II, a vertex of degree 4, having two edges each of two different colors. We observe the following.

\begin{observation}
If $L$ is the line graph of a 3-regular graph, then $L$ is good.
\end{observation}

Moreover, as we shall be primarily seeking rainbow cycles in good graphs, we note that removing a rainbow cycle from a good graph will automatically preserve almost every property of goodness.

\begin{observation}\label{heredity}
If $G$ is a good colored graph, and $R$ is a rainbow cycle in $G$, then $G\ba R$ inherits properties \eqref{inheritb}-\eqref{inherite} from $G$.
\end{observation}

Hence, if we remove a rainbow cycle from a good colored graph, in order to check if the resulting colored graph is good, we need only verify property \eqref{nox}. 

Let us consider some basic properties of good colored graphs. First, we shall examine good colored graphs for which every vertex is of Type II. We note that in this case, we must have that the subgraph induced by each color class is a triangle, as every vertex incident to that color class has exactly two edges of that color.

\begin{lemma}\label{conn} If $G$ is a connected good colored graph consisting entirely of vertices of Type II, and $R$ is a rainbow cycle in $G$, then $G\ba R$ is connected.
\end{lemma}

\begin{proof}
Suppose not, so that $G\ba R$ has components $G_1, G_2, \dots, G_k$ for some $k\geq 2$. Then wolog there exists an edge $e=\{x, y\}\in R$ such that $e$ has one endpoint $x$ in $G_1$, and the other endpoint $y$ in $G_2$. Let $\alpha=c(e)$. Note that there are two other edges in $G$ of color $\alpha$, since every vertex in $G$ is of Type II. As $R$ is a rainbow cycle, both of these edges must appear in $G\ba R$. Note that one such edge must be incident to $x$, and one such edge must be incident to $y$, and hence we have one edge of color $\alpha$ in $G_1$ and one edge of color $\alpha$ in $G_2$. But these two edges then share no incident vertices, a violation of condition \eqref{inherite} of good colored graphs. Thus, a contradiction has been reached, and thus $G\ba R$ is connected.
\end{proof}

\begin{lemma}If $G$ is a good colored graph consisting entirely of vertices of Type II, then there exists a rainbow cycle $C$ in $G$. Moreover, for every rainbow cycle $C$ in $G$, $G\ba C$ is also good.\label{type2}
\end{lemma}

\begin{proof}
If $G$ has a rainbow triangle, then clearly $G$ has a rainbow cycle. Let us assume, then, that $G$ has no rainbow triangle. Choose any vertex $v\in V(G)$. Build a rainbow path $v, v_1, v_2, v_3, \dots, v_k$ beginning at $v$ by arbitrarily choosing $v_{i+1}$ from among all neighbors of $v_i$ that do not use any of the colors $c(vv_1), c(v_1v_2), \dots, c(v_{i-1}v_i)$, until such a choice is impossible. Then as $v_k$ is a Type II vertex, it must have two incident edges of color $\alpha$, such that $\alpha = c(v_jv_{j+1})$ for some $j<k-1$, and moreover these three edges of color $\alpha$ form a triangle. Therefore, the edges $v_kv_j$ and $v_kv_{j+1}$ both present, and both with color $\alpha$, and no other edges in the rainbow path $v, v_1, \dots, v_k$ use color $\alpha$ (see Figure \ref{Lemma1a}). Hence, we have the rainbow cycle $C=(v_k,v_{j+1},v_{j+2},\dots, v_{k-1},v_k)$.

\begin{figure}[htp]
\includegraphics[scale=1]{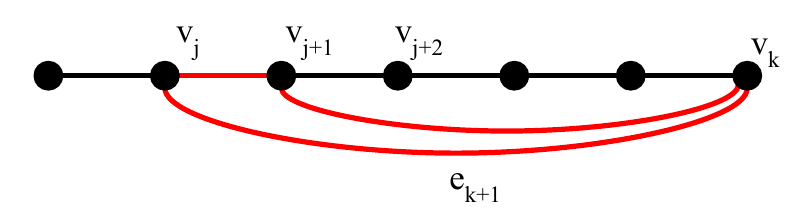}
\caption{The rainbow path formed by starting at some $v$ and proceeding arbitrarily in Lemma \ref{type2}. The color $\alpha$ is represented by the red edges, and we note that no black edges here can take color $\alpha$, and no two black edges may have the same color.}\label{Lemma1a}
\end{figure}

Hence, $G$ contains a rainbow cycle.

Now, let us take $C$ to be any rainbow cycle in $G$. Notice that removing this rainbow cycle preserves properties \eqref{inheritb}-\eqref{inherite} for a good colored graph, and we need only verify property \ref{nox}. To that end, we shall suppose to the contrary that $G\ba C$ has a cut-vertex of Type $X$, say $v$. Let $G_1$ and $G_2$ be the pseudblocks of $G$ at $v$. Note that as $G$ had no cut-vertices of Type $X$, there must have been an additional path between $G_1$ and $G_2$ along $C$ in $G$. Moreover, by Lemma \ref{conn}, as $G\ba C$ is connected, this path must be a single edge between $G_1$ and $G_2$, as there are no vertices of degree 2 in $G$ and no other components (see Figure \ref{Lemma1b}).

\begin{figure}[htp]
\includegraphics[scale=.75]{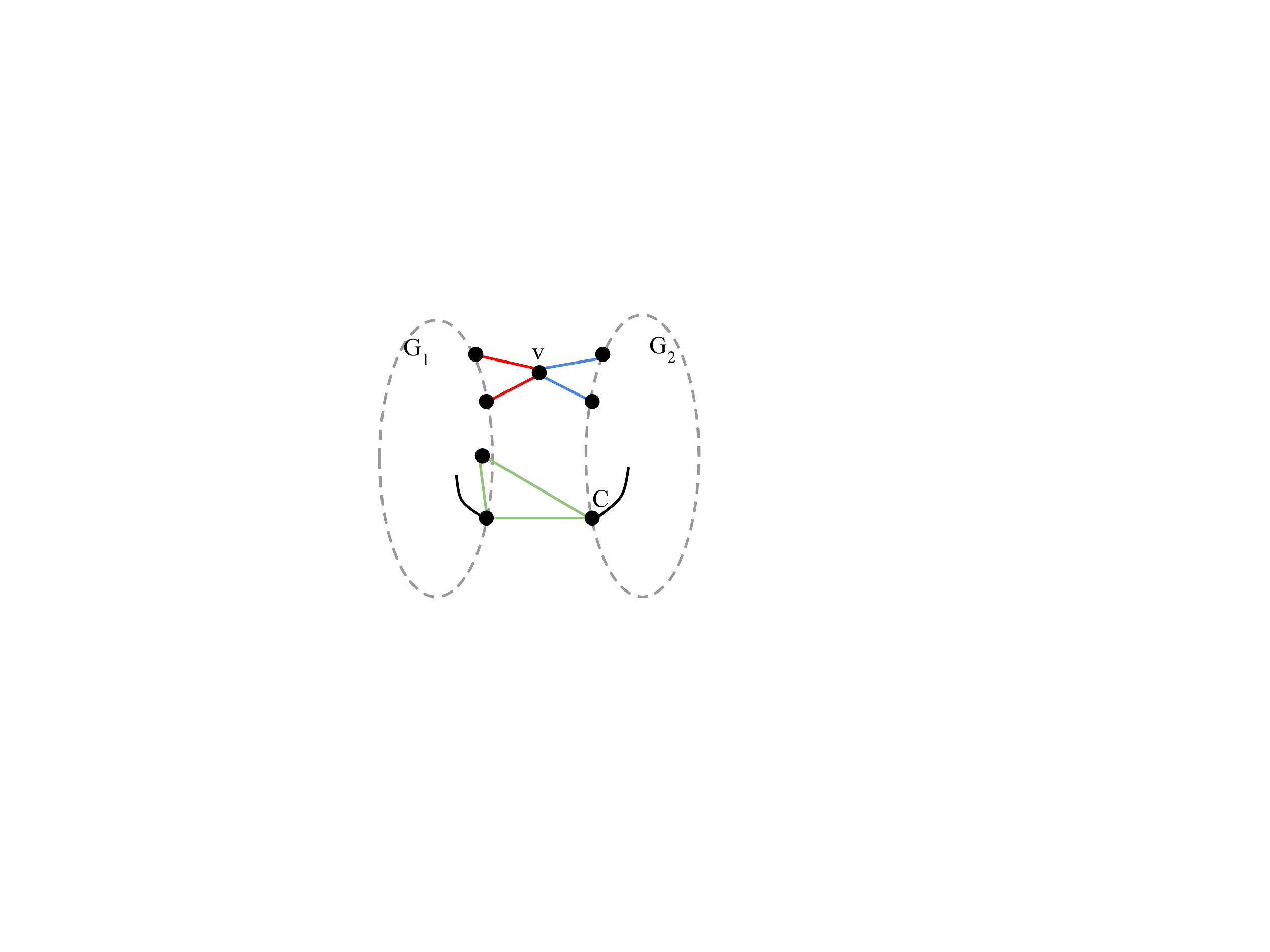}
\caption{An illustration of a potential cut vertex of Type $X$ in $G\ba C$ obtained in Lemma \ref{type2}. We note here that we do not necessarily assume that the vertices from $R$ incident to the edge colored $\alpha$ are all disjoint from the neighbors of $v$; but as noted above, they cannot both be neighbors of $v$.}\label{Lemma1b}
\end{figure}

Let $\alpha$ be the color of this edge. By the same argument as in Lemma \ref{conn}, we must have that these two vertices have a common neighbor in $G$ via edges of color $\alpha$; in Figure \ref{Lemma1b}, this neighbor is shown wolog in $G_1$. But neither of these two edges of color $\alpha$ appear in $C$, hence, vertex $v$ is not a cut vertex of Type $X$, a contradiction.

Therefore, the removal of $C$ from $G$ yields a good colored graph, as desired.

\end{proof}

The following lemma, although somewhat trivial, will in fact be quite useful in the main proof.

\begin{lemma}\label{setcycles}
Let $G$ be a good colored graph, and let $C_1, C_2, \dots, C_k$ be a collection of edge-disjoint rainbow cycles in $G$ such that $G\ba\{C_1, C_2, \dots, C_k\}$ is a good colored graph. Then $G\ba C_1$ is a good colored graph also.
\end{lemma}

\begin{proof}
Note that we need only prove that $G\ba C_1$ has no cut vertices of Type $X$, by Observation \ref{heredity}. To that end, let us suppose that $v$ is a cut vertex in $G\ba C_1$ of Type $X$. Let $G_1$ and $G_2$ be the pseudoblocks of $G\ba C_1$ at $v$.

Notice that no cycle in $\{C_2, C_3, \dots, C_k\}$ can include the vertex $v$, since as noted above, no cut vertex of Type $X$ can be present in any rainbow cycle. Hence, in $G\ba\{C_1, \dots, C_k\}$, we must have that $v$ is still a cut vertex of Type $X$, as we may take pseudoblocks by removing the cycles $C_2, \dots, C_k$ from $G_1$ and $G_2$. Therefore, $G\ba C_1$ contains no cut vertices of Type $X$.

\end{proof}

Now, in order to proceed with the main proof, we shall first require a slightly modified version of the condition of goodness in a colored graph. This will be necessary, as in some of our cases, we shall have local structure that does not lend itself easily to a modification that forces goodness. As a result, we shall bend condition \eqref{bender} slightly to broaden the class of graphs we consider.
\begin{definition}
Given a connected edge-colored graph $G$, we say that $G$ is an {\it almost-good} colored graph if the following conditions are met.
\begin{enumerate}
\item Every vertex of $G$ has even degree.
\item $G$ has maximum degree 4.
\item Every triangle in $G$ is either rainbow or monochromatic.
\item[(4a)] There is exactly one nonisolated vertex $v$ in $V(G)$ with color degree 1, and moreover this vertex has degree 2. Every other nonisolated vertex has color degree 2.\label{atmostone}
\item[(5)] The subgraph induced by each color class has at most three vertices. \label{colorgraphs}
\item[(6)] $G$ has no cut-vertices of type $X$. 
\end{enumerate}
\end{definition}

Hence, the difference between a good and an almost-good graph is that in an almost-good graph, we permit exactly one vertex to violate the condition that every vertex has color degree exactly 2. Moreover, any vertex violating this condition must have degree 2. We shall refer to this violating vertex as the bad vertex of $G$; the other vertices will still be called Type I and Type II, as with good colored graphs. We note that Observation \ref{heredity} will extend naturally to almost good graphs.

\begin{lemma}\label{rainbowtri}
Let $G$ be a good or almost-good colored graph, and let $C=(v_1, v_2, v_3, v_1)$ be a rainbow triangle in $G$. Then $G\ba C$ is a good or almost-good colored graph, respectively.
\end{lemma}

\begin{proof}
By Observation \ref{heredity}, it is necessary only to check that the removal of the rainbow triangle $C$ does not result in any cut vertices of Type $X$. Clearly, if $G$ is almost good, we cannot have any of the $v_i$ as the bad vertex. Moreover, if every vertex of the triangle is Type I, this is immediate, and hence we may assume that there exists at least one vertex of the triangle that is Type II, say $v_1$.

Let us suppose that $z$ is a cut vertex of Type $X$ in $G\ba C$. Let $G_1$ and $G_2$ be the pseudoblocks of $G\ba C$ at $z$. Note that as $z$ was not a cut vertex of Type $X$ in $G$, we must have that at least one of the vertices of $C$ is in $G_1$, and at least one in $G_2$.

Wolog, suppose that $v_1\in V(G_1)$ and $v_2\in V(G_2)$. Let $\alpha = c(v_1v_2)$. Then as $v_1$ is of Type II, there must exist a vertex $w\neq v_2, v_3$ such that $v_1\sim_G w$ and $c(v_1w)=\alpha$. By the definition of $G_1$ and $G_2$, we must also have $w\in V(G_1)$ (see Figure \ref{Lemma4}). But then $v_2\not\sim_G w$, and hence $v_2$ is of Type $I$ in $G$. Hence, $v_2$ is isolated in $G\ba C$, and thus $v_2$ can be reassigned to $V(G_1)$ to obtain a new psuedoblock decomposition of $G\ba C$ at $z$. Similarly, if $v_3\in V(G_2)$, then $v_3$ is also of Type I in $G$ and can be reassigned to $V(G_1)$. But then $z$ is a cut vertex of Type $X$ in $G$, also, a contradiction.

\begin{figure}[htp]
\includegraphics[scale=.8]{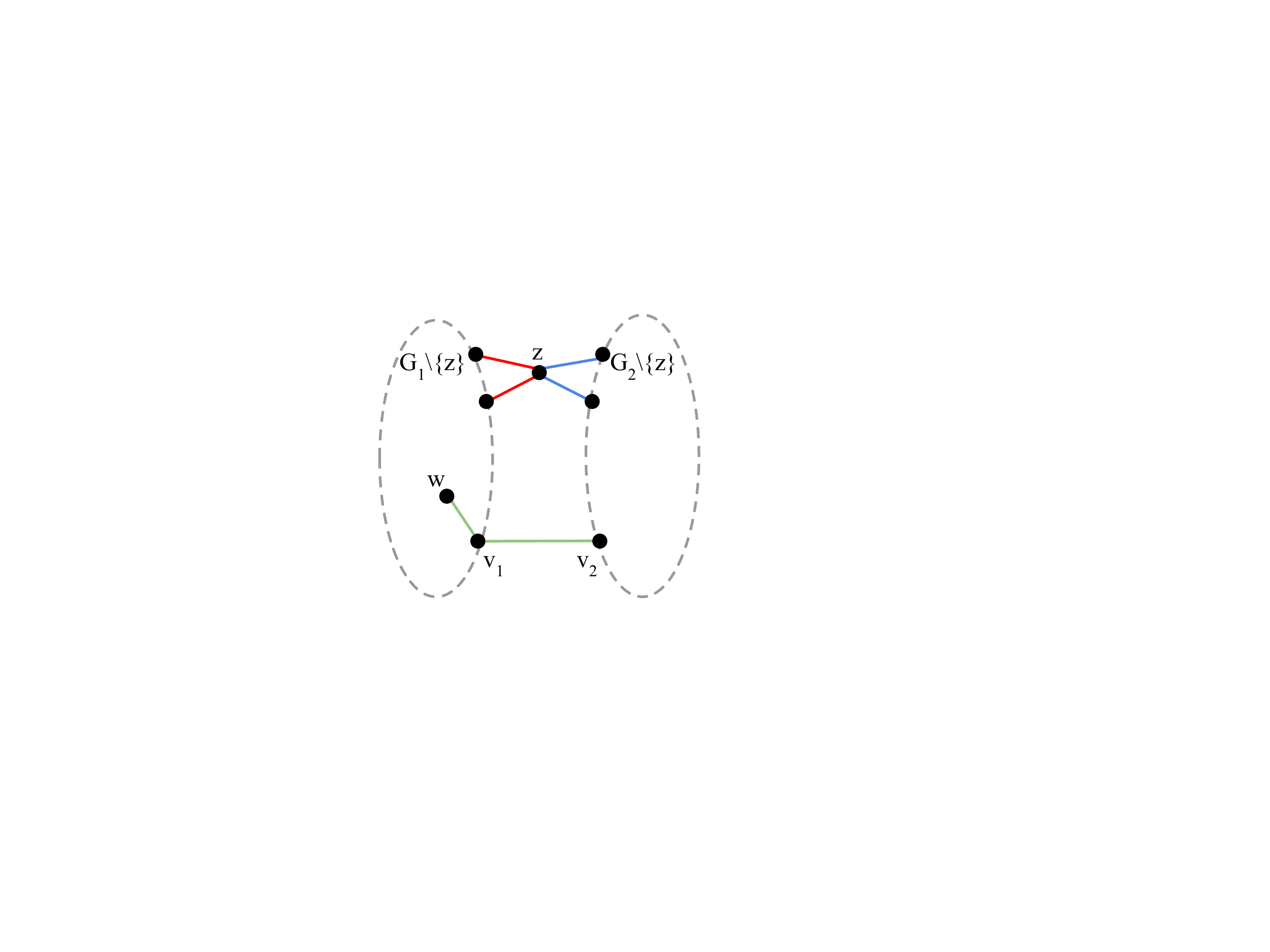}
\caption{Structure of $G$ in Lemma \ref{rainbowtri} in the case that $z$ is a cut vertex of Type $X$.}\label{Lemma4}
\end{figure}

Therefore, the removal of a rainbow triangle cannot produce any cut vertices of Type $X$, and hence $G\ba C$ is good or almost-good, respectively.

\end{proof}

Our primary goal is to show that given a good colored graph, we can always find a rainbow cycle decomposition. In order to do so, we actually prove a stronger result that holds for both good and almost good colored graphs. This stronger version is required, as the proof will be by induction, and the inductive step will require reduction to a smaller graph. Such a reduction may, in some cases, yield an almost-good graph, instead of a good graph, and hence we shall include almost-good graphs in our key thorem.

In order to state the main theorem, we must first consider what a rainbow-like cycle decomposition should look like in an almost-good graph, since certainly there can be no actual rainbow cycle decomposition. However, since only one vertex fails to have color degree 2, we could (and will) decompose the edges into cycles so that only one cycle fails to be rainbow. Moreover, we shall ensure that this cycle is as close to rainbow as possible. Specifically, we shall use the following definition.

\begin{definition}
Let $G$ be an edge colored graph. We call a cycle $C=(v_0, v_1, v_2, \dots, v_k, v_0)$ an {\it almost-rainbow} cycle if $c(v_0v_1)=c(v_kv_0)$, but the path $v_0, v_1, \dots, v_k$ is rainbow.
\end{definition}

That is to say, a cycle is almost-rainbow if it uses $k$ colors for $k+1$ edges, and the repeated color appears on two consecutive edges.

\begin{theorem}\label{T:mainthm}
Let $G$ be a graph that is either good or almost-good. Then there is a decomposition of the edges of $G$ into cycles $\{C_1, C_2, \dots, C_r\}$ such that one of the following is true:
\begin{enumerate}
\item If $G$ is good, then $C_i$ is rainbow for all $i$. \label{firsttype}
\item If $G$ is almost good, then $C_i$ is rainbow for all $i\geq 2$, and $C_1$ is almost-rainbow. \label{secondtype}
\end{enumerate}
\end{theorem}

We note that as line graphs of cubic graphs are always good, combining Theorem \ref{T:mainthm} with Lemma \ref{equiv} immediately yields a proof of the CDCC. Hence, in order to resolve the CDCC, it remains only to prove Theorem \ref{T:mainthm}.

\section{Proof of Theorem \ref{T:mainthm}}\label{S:proof}

Our proof shall be done in cases. Before we begin, we first define a singular path; the presence of such a path will be one of the cases on which the proof relies.

\begin{definition}
Let $G$ be a good or almost-good graph. We call a path $v_0, v_1, v_2, \dots, v_k$ in $G$ a {\it singular path of length $k$} if the vertices $v_1, v_2, \dots, v_{k-1}$ are all of Type I. 
\end{definition}

We note that the length of a singular path is the number of edges involved, not the number of vertices. Moreover, any cycle that includes one edge of the singular path must include all edges of the singular path.

\begin{proof}[Proof of Theorem \ref{T:mainthm}]
Let $G$ be a good or almost-good colored graph with order $n$ and size $m$. We shall work by induction, first on $n$ and then on $m$. We shall assume throughout that $G$ is connected, as if not, we may simply choose one connected component of $G$ to work with.

Note that the minimum case will be when $n=4$, and $G$ is a $C_4$ having either three or four colors, such that if a color is repeated, the repetition appears at a single vertex. This is clearly a cycle satisfying one of the above conditions.

Indeed, this base case extends to any $n$ with the minimal number of edges $n$, as in these cases, the graph is a single cycle, which automatically satisfy condition \eqref{firsttype} if $G$ is good, and \eqref{secondtype} if $G$ is almost-good.

Now, suppose that $G$ is either a good or almost-good colored graph, and suppose further that any good or almost-good colored graph on either fewer vertices or edges has a cycle decomposition satisfying one of the two conditions. Note that by induction, it is sufficient to show that $G$ contains either a rainbow cycle $C$, or, if $G$ is almost good, an almost-rainbow cycle $C$ using the bad vertex, such that $G\ba C$ is also either good or almost-good (appropriately). We may then remove this cycle and use induction to decompose the remainder of the graph into cycles that satisfy the conditions. Note further that by Lemma \ref{rainbowtri}, we may assume that $G$ has no rainbow triangles, as if it does, we may immediately remove one such, and apply induction to the remainder of the graph. We shall consider two cases, according to whether $G$ is good or almost-good.

\begin{case}$G$ is almost-good.\label{C:almost}
\end{case}

 Let $v$ be the bad vertex of $G$, let $\alpha$ be its incident color, and let $x_1, x_2$ be its neighbors. Note that by property \eqref{colorgraphs}, there may be only one other edge in $G$ with color $\alpha$, namely $x_1x_2$. Moreover, by the restriction against triangles having exactly two edges with the same color, if $x_1x_2\in E$, it must take color $\alpha$. We shall split into two subcases according to whether this edge is present.

\begin{subcase} $x_1x_2\in E$. \label{C:almost_tri}
\end{subcase}

Note that there are no edges of color $\alpha$ anywhere else in the graph $G$. Moreover, $x_1$ and $x_2$ are both Type II vertices. Let $x_1$ have neighbors $y_1, w_1$, with $c(x_1y_1)=c(x_1w_1)=\beta$, and let $x_2$ have neighbors $y_2, w_2$, with $c(x_2y_2)=c(x_2w_2)=\gamma$. Notice that we must have $\{y_1, w_1\}$ disjoint from $\{y_2, w_2\}$, as otherwise we have a rainbow triangle in $G$. Create a new graph $G'$ by contracting $G$ along the triangle $(x_1, v, x_2, x_1)$; note that as $\{y_1, w_1\}$ is disjoint from $\{y_2, w_2\}$, this cannot create any multiple edges. Let $x=[x_1]$ in the contraction; by abuse of notation, we shall refer to any other vertex by its label in $G$. This contraction is shown in Figure \ref{F:Case1.1}.

\begin{figure}[htp]
\includegraphics[scale=.7]{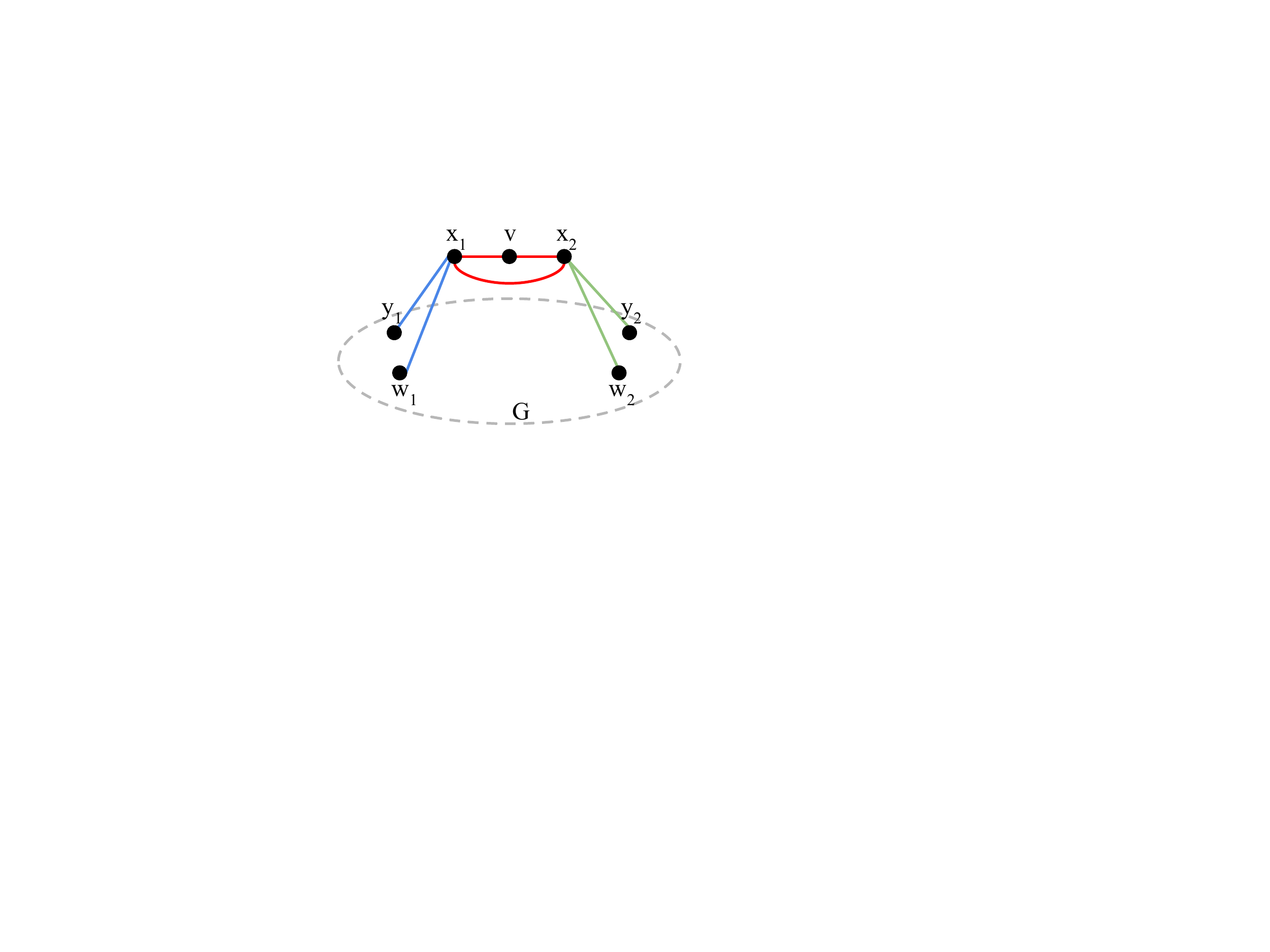}
\hspace{.3in}
\includegraphics[scale=.7]{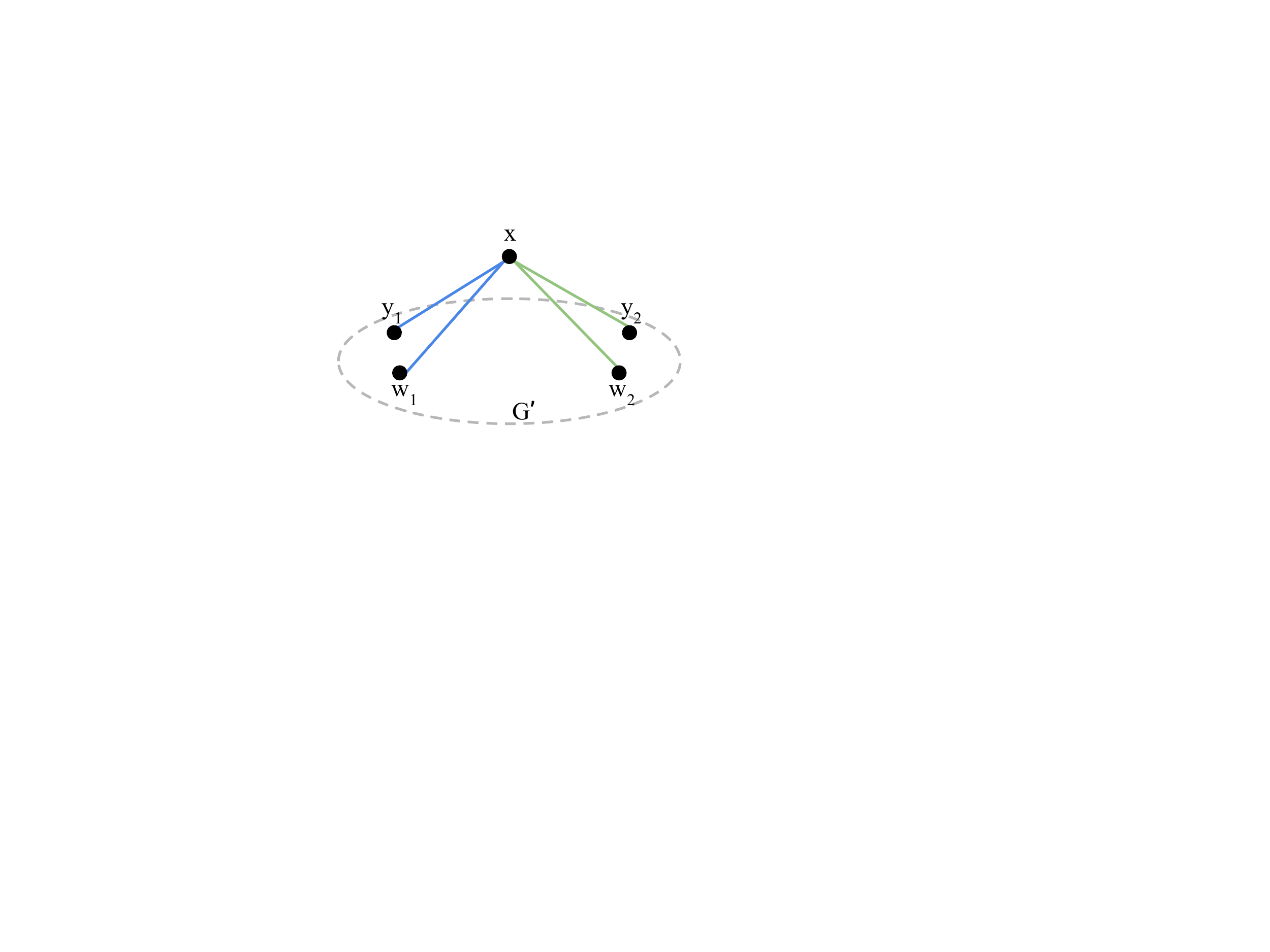}
\caption{An illustration of the construction of $G'$ from $G$ in Subcase \ref{C:almost_tri}, by contracting the triangle $x_1vx_2$. Here, we use red for color $\alpha$, blue for color $\beta$, and green for color $\gamma$.}
\label{F:Case1.1}
\end{figure}

Note that $G'$ has $n-2$ vertices and $m-3$ edges. Also, since $x_1$ and $x_2$ are both not cut vertices in $G$, clearly $x$ is not a cut vertex in $G'$, and hence $G'$ has no cut vertex of Type $X$. Moreover, provided we have not formed a nonmonochromatic, nonrainbow triangle, the resulting graph is good. Note that we can only form a nonmonochromatic triangle in the event that there is an edge between the sets $\{y_1, w_1\}$ and $\{y_2, w_2\}$. Moreover, if such an edge exists, it cannot take color $\alpha, \beta, $ or $\gamma$, as there would be too many vertices having incident edges in these color classes. Hence, we cannot form a triangle having exactly two edges of the same color. Thus, $G'$ is good, and by the inductive hypothesis, there exists a decomposition $\mathcal{C}$ of the edges of $G'$ into rainbow cycles. Let $C'\in \mathcal{C}$ be such that $C'$ includes the vertex $x$. Note that $G'\ba C'$ is a good colored graph, by Lemma \ref{setcycles}.

Form a cycle $C$ in $G$ by replacing the vertex $x$ in $C'$ with the path $x_1vx_2$ or $x_2vx_1$, appropriately. Note then that $G\ba C$ can be obtained from $G'\ba C'$ by subdividing the edge $xy_1$ (wolog), and recoloring appropriately. Hence, since $G'\ba C'$ is good by induction, we have that $G\ba C$ is good by Observation \ref{subdivide}. This construction is illustrated in Figure \ref{F:Case1.1C}.

\begin{figure}[htp]
\includegraphics[scale=.7]{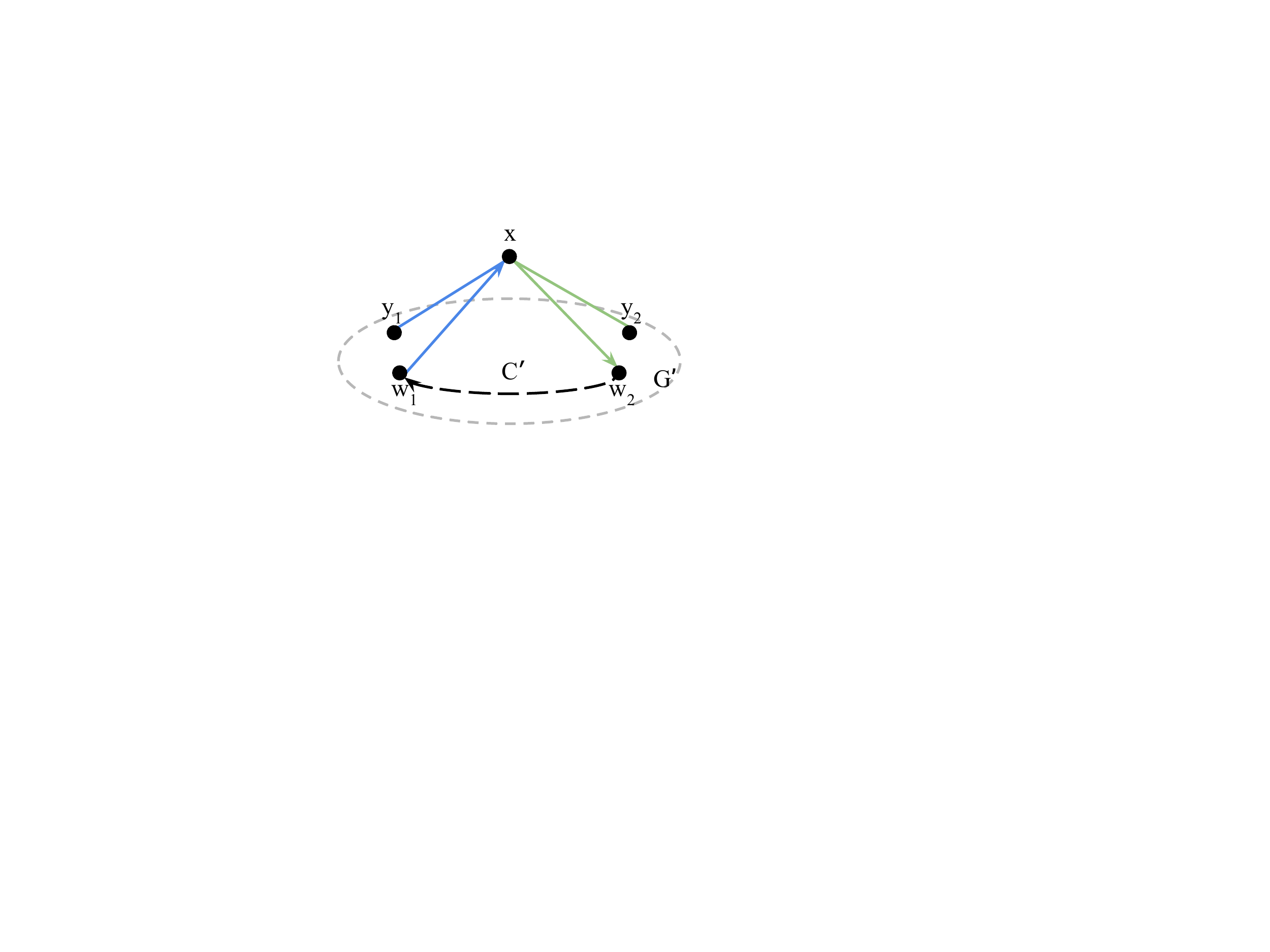}
\hspace{.3in}
\includegraphics[scale=.7]{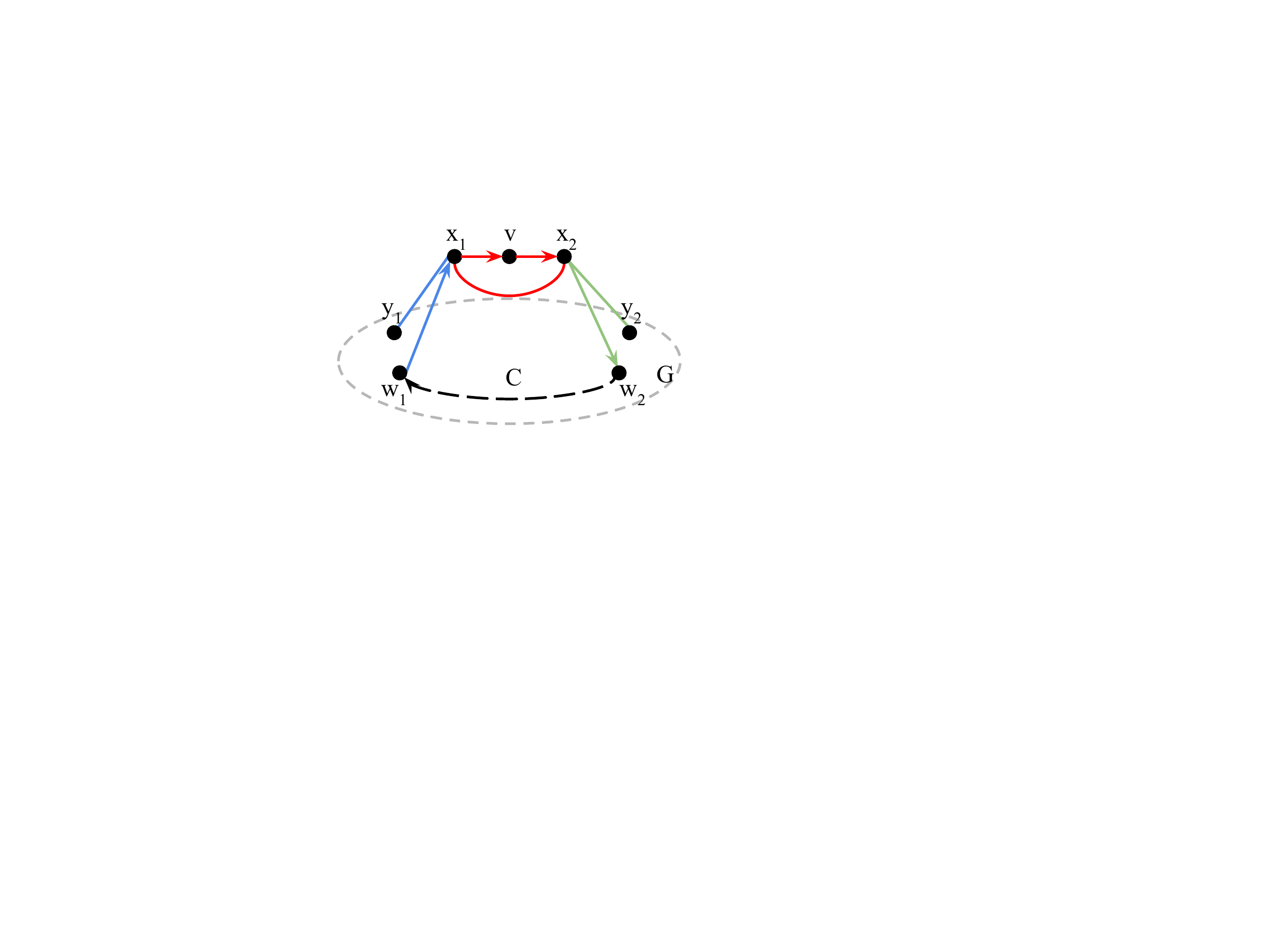}
\caption{The construction of $C$ from $C'$ in Subcase \ref{C:almost_tri}, in the case that the cycle $C'$ uses the vertex $x$. Here, we indicate the cycles $C$ and $C'$ using arrows on the associated edges. Between two vertices, we draw a dashed line to indicate a path (rather than an edge); here we see a dashed line indication the portion of the cycles $C$ and $C'$ that does not involve the vertices pictured.
}\label{F:Case1.1C}
\end{figure}

Therefore, $G\ba C$ is almost-good if $C$ does not use the path $x_1vx_2$, or good if it does, as desired.

\begin{subcase}$x_1x_2\notin E$\label{C:almost_notri}
\end{subcase}

In this case, we must have that both $x_1$ and $x_2$ are vertices of Type I.  Let $y_1$ be the neighbor of $x_1$ other than $v$, and $y_2$ the neighbor of $x_2$ other than $v$. Let $\beta = c(y_1x_1)$ and $\gamma = c(y_2x_2)$. Then we have a singular path of length $4$ given by $y_1, x_1, v, x_2, y_2$. Let $G'$ be the graph obtained from $G$ by contracting the edge $vx_1$, and (by abuse of notation) labeling the resulting node $x_1$. This is illustrated in Figure \ref{F:Case1.2}. 

The resulting graph $G'$ is clearly good, and moreover, $G'$ has $n-1$ vertices. Hence, by induction, there exists a decomposition $\mathcal{C}$ of the edges of $G'$ into rainbow cycles. Let $C'\in \mathcal{C}$ with the edge $x_1x_2$ appearing on $C'$, and note that $G'\ba C'$ is good. Form a cycle $C$ in $G$ by subdividing this edge in $C'$ by $v$. Note that this is an almost-rainbow cycle, and $G\ba C=G'\ba C'$, a good colored graph.

\begin{figure}[htp]
\includegraphics[scale=.8]{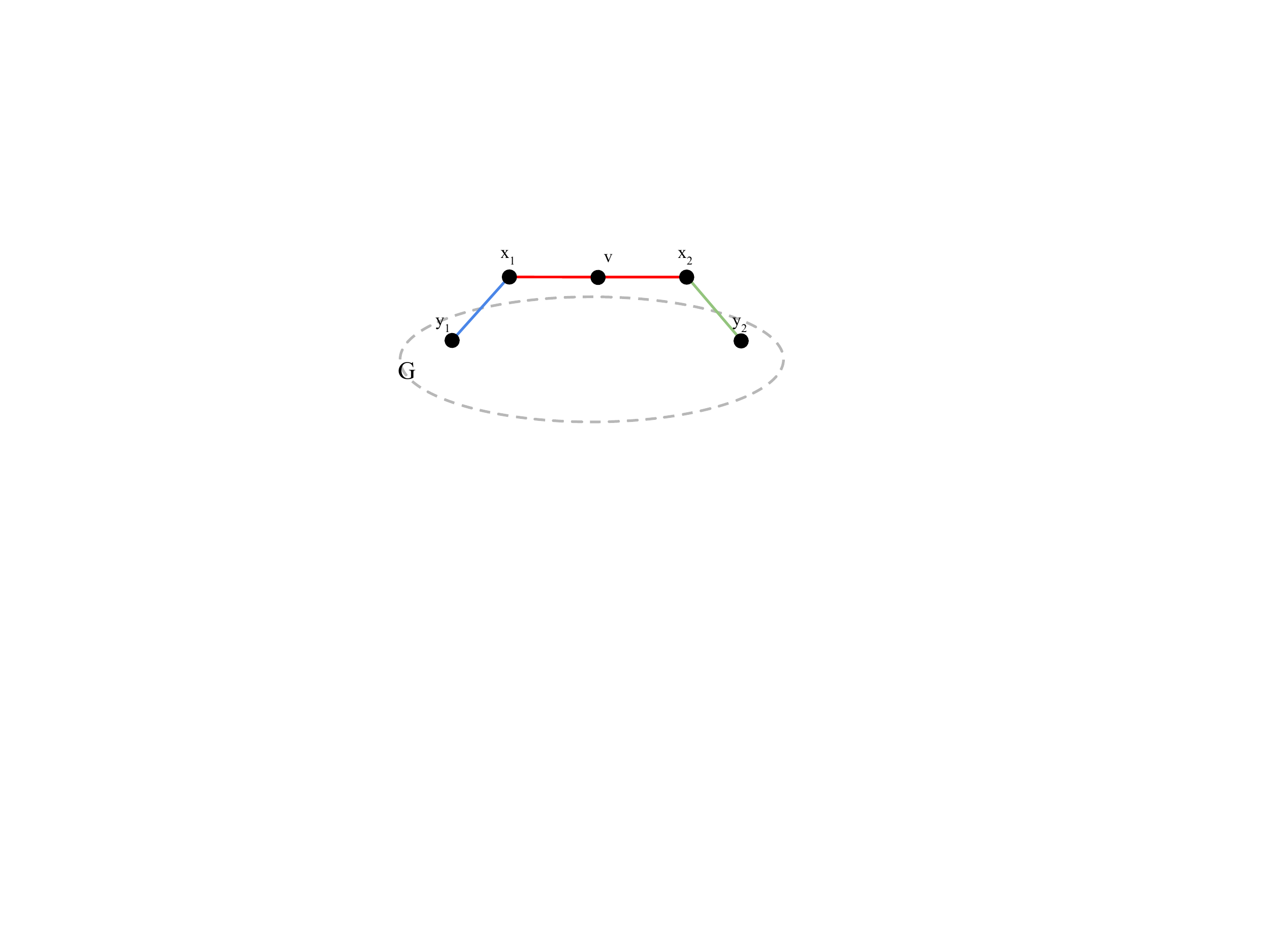}
\hspace{.3in}
\includegraphics[scale=.8]{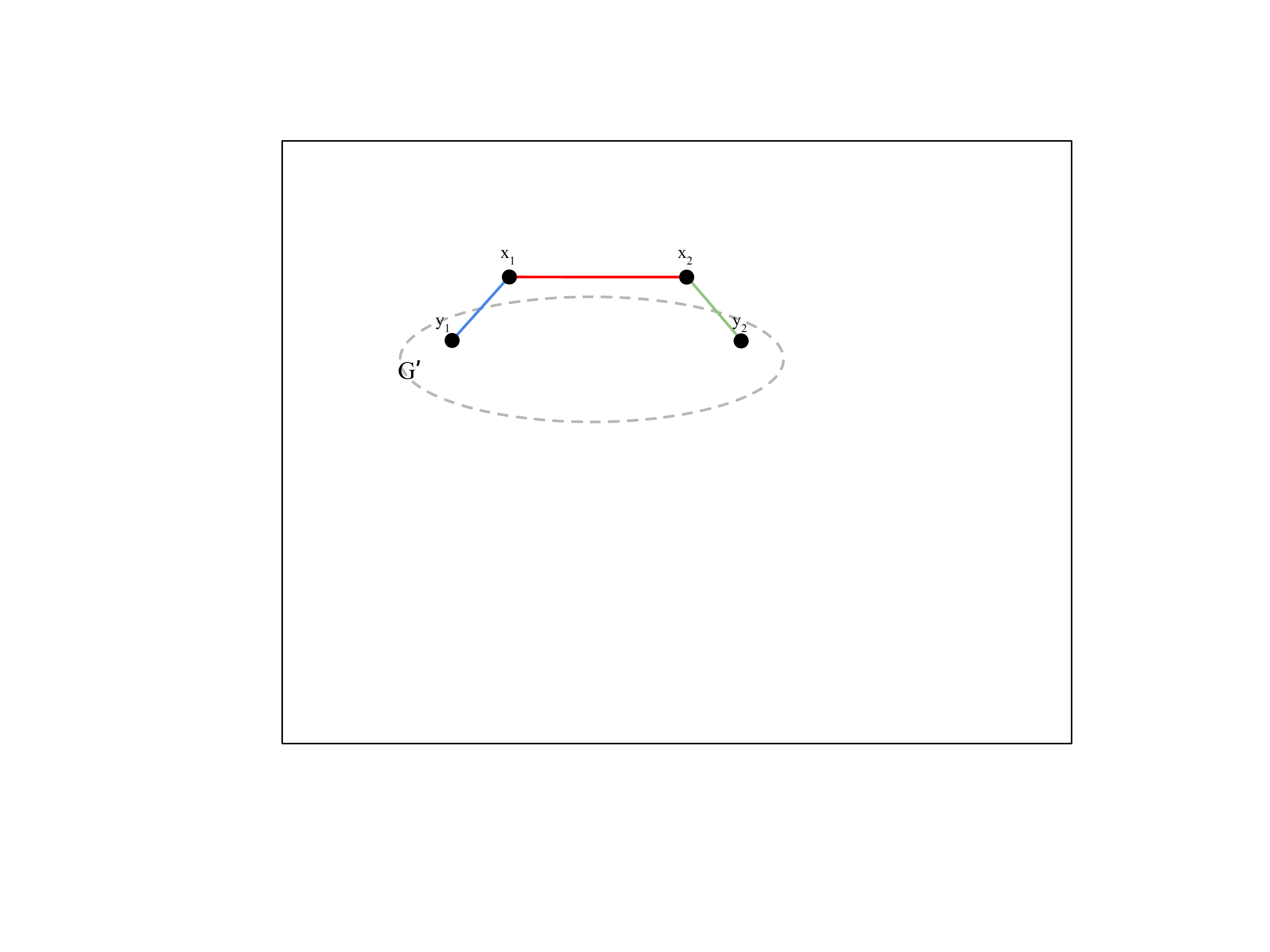}
\caption{The graph $G$ and corresponding contraction to $G'$ for Subcase \ref{C:almost_notri}. As above, red edges indicate color $\alpha$, blue indicate color $\beta$, and green indicate color $\gamma$.}\label{F:Case1.2}
\end{figure}

\begin{case}$G$ is good.\end{case}

We note that by Lemma \ref{type2}, if $G$ consists entirely of vertices of Type II, then we are done. Hence, we may suppose that $G$ has at least one vertex of Type I. We consider two cases, according to the length of the longest singular path in $G$.

\begin{subcase}The longest singular path in $G$ has length at least 3.\label{C:good_singular}
\end{subcase}
Let $P=v_0, v_1, v_2, v_3$ be a singular path in $G$, so that $v_1$ and $v_2$ are both Type I vertices. Let $\alpha = c(v_1v_2)$. Note that no other edge in $G$ may have color $\alpha$, as if so, it would be incident to one of $v_1$ or $v_2$, and thus one of these vertices would be of Type II. 

Form a new graph $G'$ from $G$ by contracting the edge $v_1v_2$; label the new vertex formed (by abuse of notation) as $v_1$, and color the new edge $v_1v_3$ with $c(v_2v_3)$. This contraction is illustrated in Figure \ref{F:Case2.1}.

Clearly, $G'$ is a good colored graph, and moreover, $G'$ has order $n-1$. Hence, by the inductive hypothesis, there exists a decomposition $\mathcal{C}$ of the edges of $G'$ into rainbow cycles. Let $C'\in\mathcal{C}$ be a rainbow cycle that includes the edge $v_1v_3$, and note that $G'\ba C'$ is good. Create a rainbow cycle $C$ in $G$ by subdividing this edge with $v_2$, and recoloring as in $G$. Note that $G\ba C = G'\ba C' \cup\{v_2\}$, where $v_2$ is an isolated vertex, so $G\ba C$ is good. Moreover, no edge in $C'$ can be colored $\alpha$, and hence $C$ is also rainbow.

\begin{figure}[htp]
\includegraphics[scale=.8]{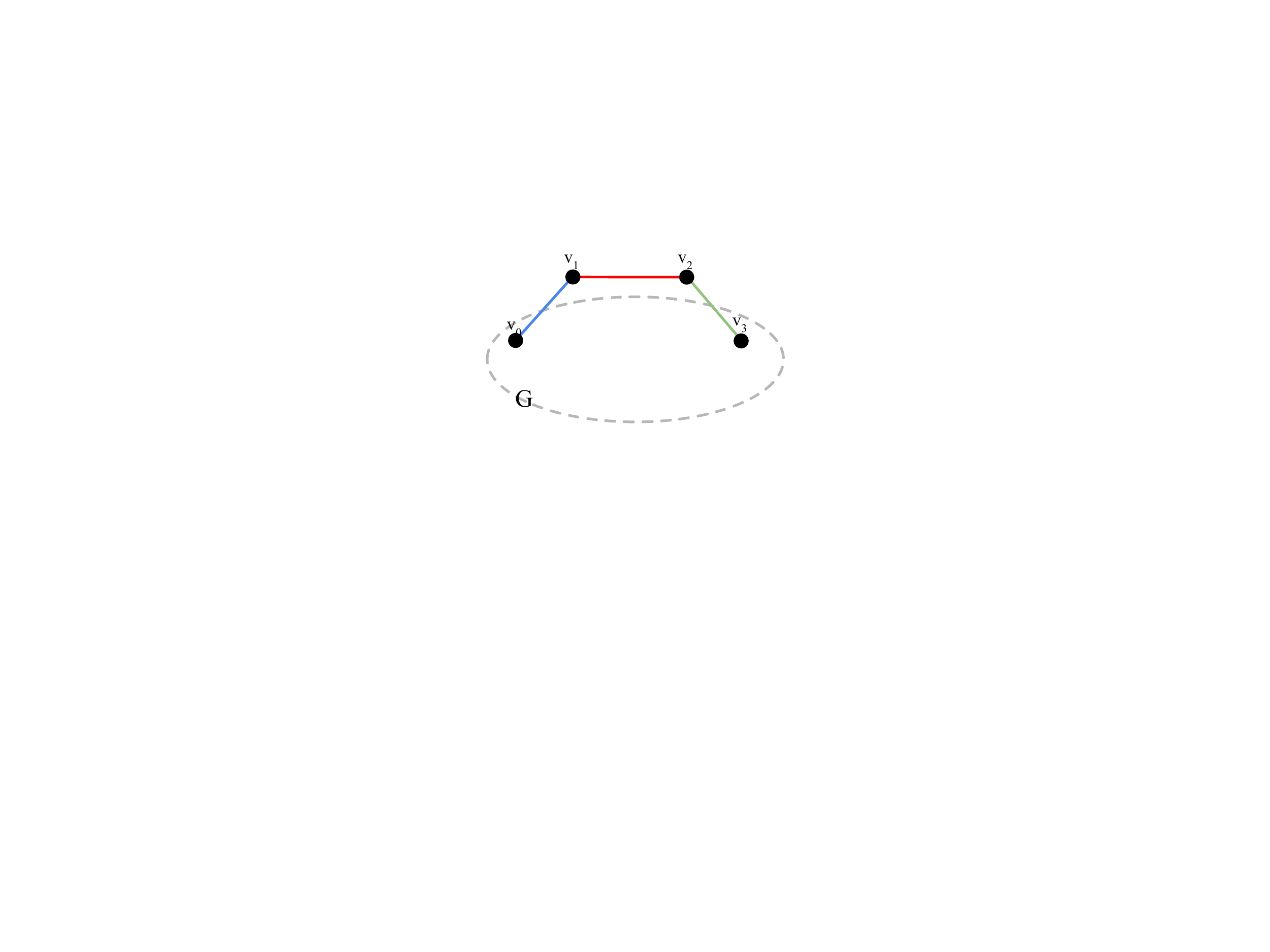}
\hspace{.3in}
\includegraphics[scale=.8]{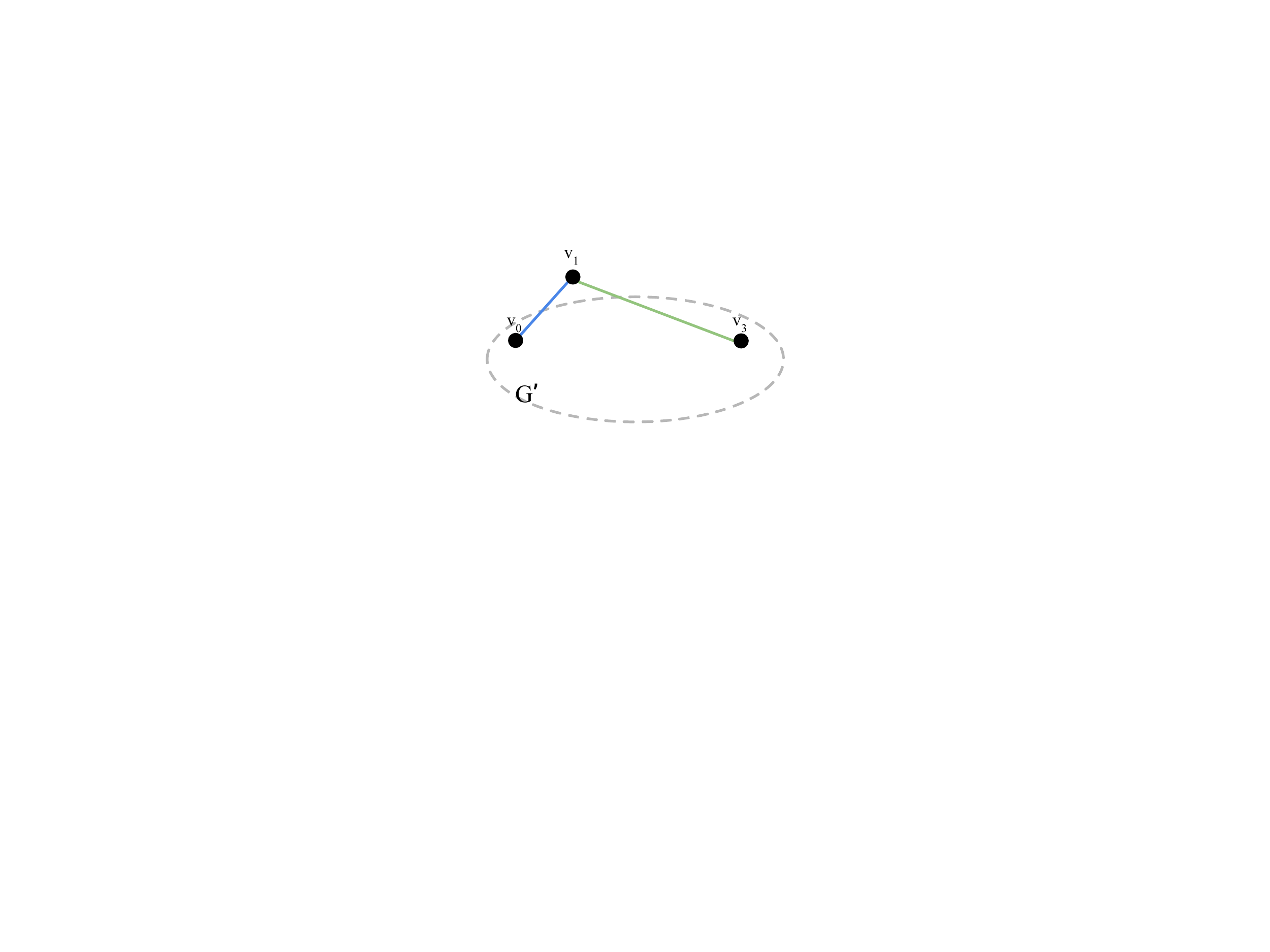}
\caption{The graph $G$ and corresponding transformation to $G'$ for Subcase \ref{C:good_singular}. As above, red edges indicate color $\alpha$.}\label{F:Case2.1}
\end{figure}

\begin{subcase}The longest singular path in $G$ has length $2$.\label{C:nonsingular}
\end{subcase}

As we know that $G$ contains at least one Type I vertex, let $v$ be a Type I vertex in $G$, with neighbors $x_1$ and $x_2$, and $\alpha = c(x_1v)$, $\beta = c(vx_2)$. Note that as $G$ contains no singular path of length 3, we must have that $x_1$ and $x_2$ are both of Type II.

Let $y_1$, $w_1$, and $z_1$ be the neighbors of $x_1$ other than $v$, such that $c(x_1y_1)=\alpha$, and $c(x_1w_1)=c(x_1z_1)=\gamma$. Likewise, let $y_2, w_2,$ and $z_2$ be the neighbors of $x_2$ other than $v$, such that $c(x_2y_2)=\beta$, and $c(x_2w_2)=c(x_2z_2)=\delta$. This basic structure is shown in Figure \ref{F:2.2setup}.

\begin{figure}[htp]
\includegraphics[scale=.8]{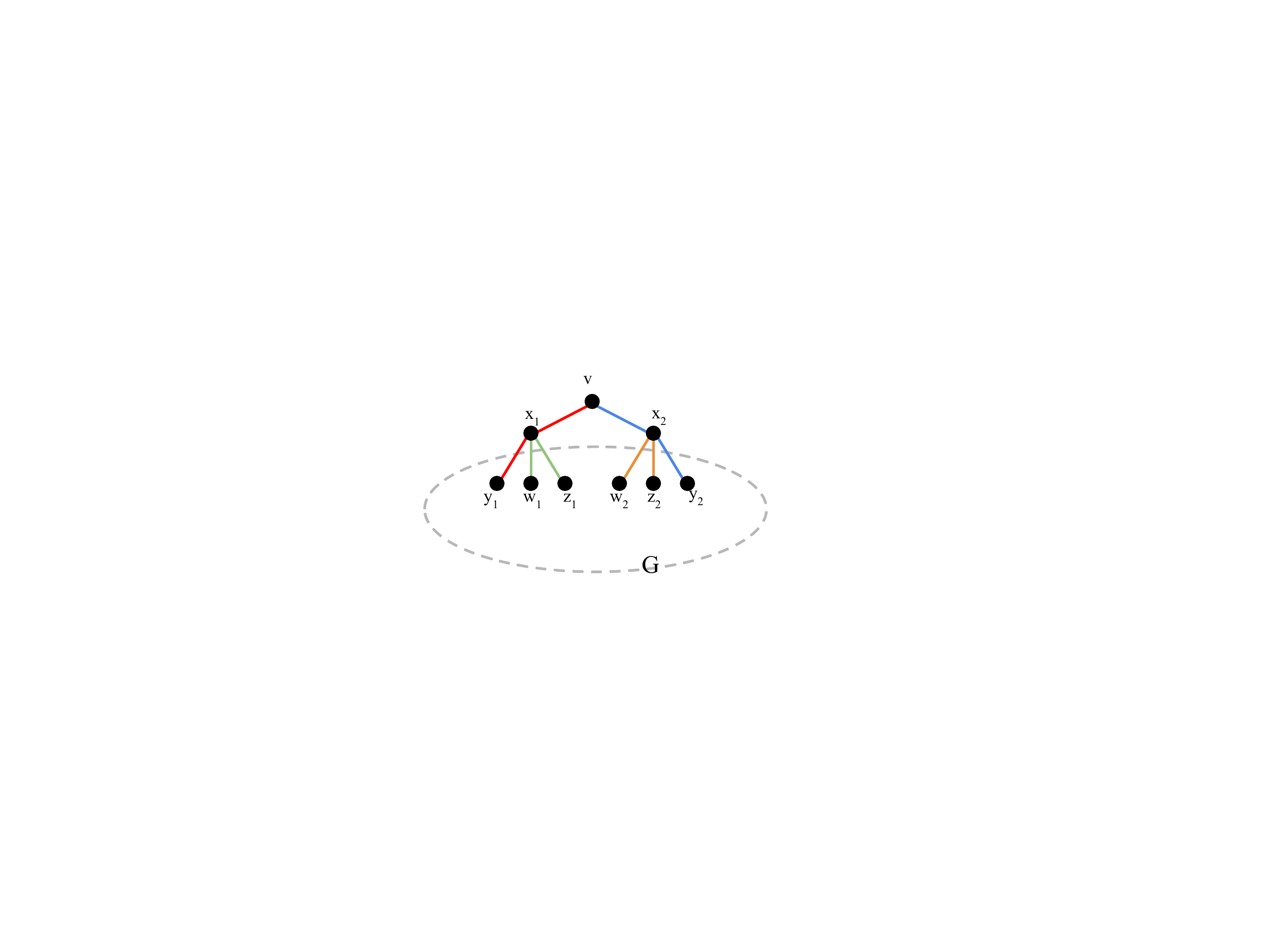}
\caption{We here illustrate the basic structure for all Subcases under Subcase \ref{C:nonsingular}. Throughout this subcase, we use color red for $\alpha$, blue for $\beta$, green for $\gamma$, and orange for $\delta$. We note that this particular drawing depicts vertices $\{y_1, w_1, z_1\}$ as disjoint from $\{y_2, w_2, z_2\}$; although that may not be the case, as in Subcase \ref{C:overlap}, it is sufficient for this illustration of the fundamental structure. Note that both $y_1$ and $y_2$ are vertices of Type I.}\label{F:2.2setup}
\end{figure}

Moreover, as the edges $y_1v$ and $y_2v$ are not present in $G$, we must have that $y_1$ and $y_2$ are both vertices of Type I. We shall consider several cases, depending on whether $\{y_1, w_1, z_1\}$ is disjoint from $\{y_2, w_2, z_2\}$. We first consider the case that the two sets are disjoint. 

\begin{sscase}The sets $\{y_1, w_1, z_1\}$ and $\{y_2, w_2, z_2\}$ are disjoint.
\end{sscase}

Form a new graph $G'$ from $G$ as follows.

\begin{itemize}
\item Remove edges $y_1x_1$, $x_1v$, $vx_2$, and $x_2y_2$ from $G$.
\item Add edges $y_1v$ and $vy_2$, colored $\alpha$ and $\beta$, respectively.
\item Merge vertices $x_1$ and $x_2$ into a new vertex, $x$, having neighbors $w_1, z_1, w_2, z_2$.
\end{itemize}
This construction is illustrated in Figure \ref{F:2.2.1Gprime}. Note that $G'$ has $n-1$ vertices. Moreover, $G'$ immediately satisfies all but properties \eqref{monotri} and \eqref{nox} of a good graph.

We first claim that $G'$ cannot contain any nonmonochromatic, nonrainbow triangles. Indeed, there are two ways to produce a triangle in $G'$ that was not already present in $G$. Either we have the edge $y_1y_2$, or we have at least one edge between $\{w_1, z_1\}$ and $\{w_2, z_2\}$.

Let us first consider the second case. Suppose, wolog, that $w_1w_2\in E(G)$; and let $\zeta = c(w_1w_2)$. Note that $\zeta \neq \alpha$, as it is not incident to any of $x_1, y_1 v$, and likewise, $\zeta\neq \beta$. On the other hand, as only vertices $x_1, w_1, z_1$ may be incident to edges of color $\gamma$, and the sets $\{y_1, w_1, z_1\}$, $\{y_2, w_2, z_2\}$ are disjoint, we also have $\zeta \neq \gamma$. Likewise, $\zeta \neq\delta$, and hence the triangle $(w_1,w_2,x)$ is rainbow.

On the other hand, let us suppose that we have the edge $y_1y_2$ in $G$. Recalling that both $y_1$ and $y_2$ are Type I vertices in $G$, we have that the edge $y_1y_2$ cannot take color $\alpha$ or $\beta$. Hence, the triangle $vy_1y_2$ in $G'$ must be rainbow.

\begin{figure}[htp]
\includegraphics[scale=.8]{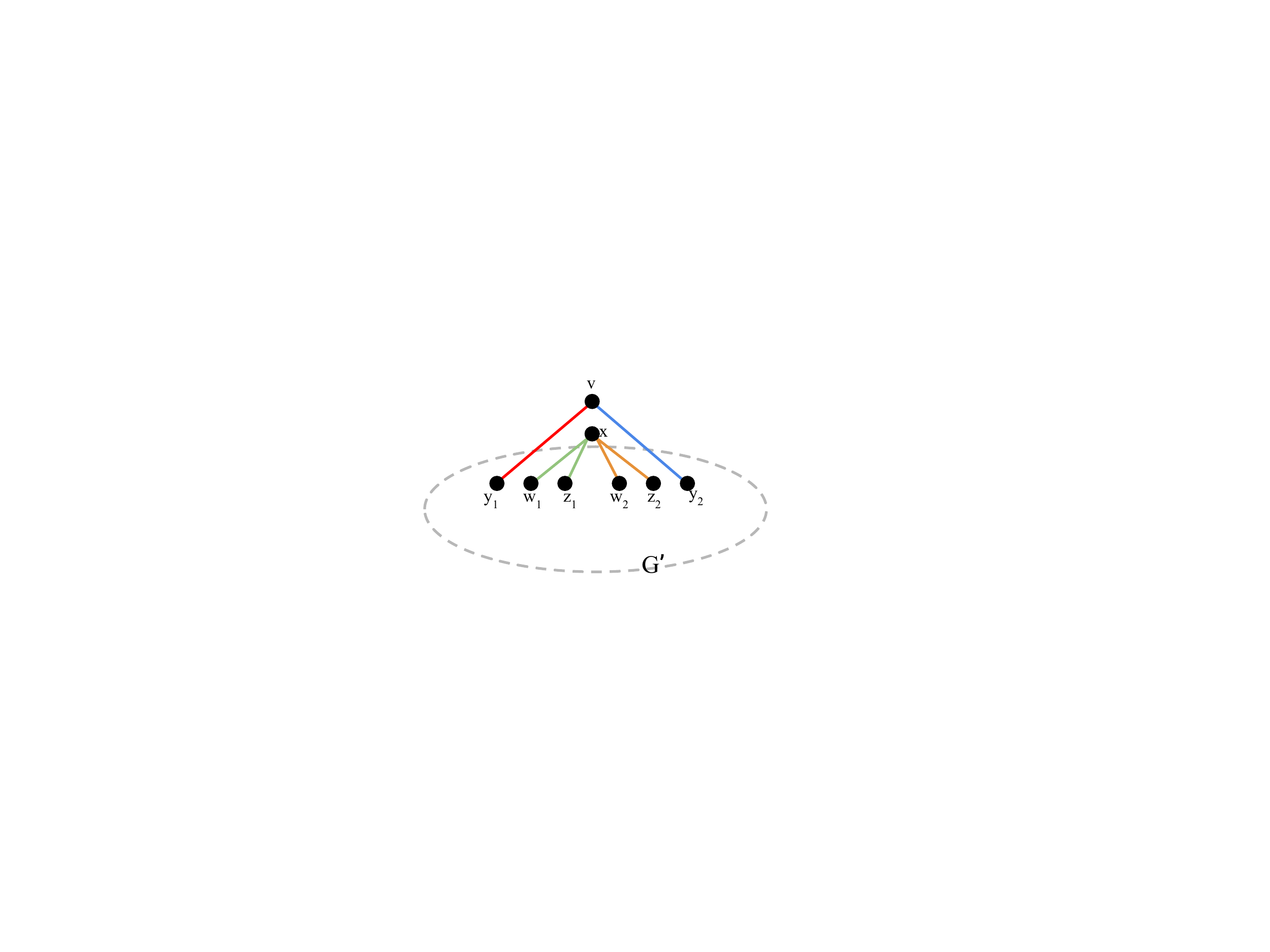}
\caption{The transformation to $G'$ in the subcase \ref{C:nonsingular}.1.}\label{F:2.2.1Gprime}
\end{figure}

Therefore, any new triangle created in $G'$ must be rainbow. Hence, there are two possibilities: either $G'$ is good, or $G'$ contains a cut vertex of Type $X$. We consider each of these as subcases.

\quad

\noindent\textbf{Subcase 2.2.1(a).} $G'$ is a good colored graph.

\quad

Since $G'$ is good, there exists a decomposition of the edges of $G'$ into rainbow cycles, $\mathcal{C}=\{C_1, C_2, \dots, C_k\}.$ 

Suppose that we have a rainbow cycle $C\in\{C_1, \dots, C_k\}$ such that $C$ uses neither $v$ nor $x$. Note then that $C$ uses none of the edges shown in Figure \ref{F:2.2.1Gprime}. Moreover, $C$ is also a rainbow cycle in $G$, using none of the edges shown in Figure \ref{F:2.2setup}. Moreover, by Lemma \ref{setcycles}, $G'\ba C$ is a good colored graph. Moreover, if $G\ba C$ has a cut vertex of Type $X$, say $t$, then take $G_1, G_2$ to be pseudoblocks of $G\ba C$ at $t$. Note that if $t$ is not one of $x_1, x_2, w_i$ or $z_i$, then as the subgraph induced on $\{v, x_1, y_1, w_1, z_1, x_2, y_2, w_2, z_2\}$ is connected in $G$, we must have that all of these vertices are in the same pseudoblock, say $G_1$. But then $G_2$ is an induced subgraph of $G'\ba C$, and hence $t$ is also a cut vertex of Type $X$ in $G'\ba C$, a contradiction. Hence, $t$ must be one of $x_1, x_2, w_i$ or $z_i$.

Suppose $x_1$ or $x_2$ is a cut vertex of Type $X$ in $G\ba C$; wolog, say it is $x_1$. Then wolog, we have $w_1, z_1\in V(G_1)$ and $y_1, v, x_2, w_2, z_2\in V(G_2)$, as shown in Figure \ref{F:2.2.1axcut}. But notice then that if $G_1'$ and $G_2'$ are subgraphs of $G'\ba C$, induced on $V(G_1)\ba\{x_1\}\cup \{x\}$ and $V(G_2)\ba\{x_1\}\cup \{x\}$, then $G_1'$ and $G_2'$ cover all edges of $G'\ba C$, and hence $x$ is a cut vertex of Type $X$ in $G'\ba C$, a contradiction. 

\begin{figure}[htp]
\includegraphics[scale=.8]{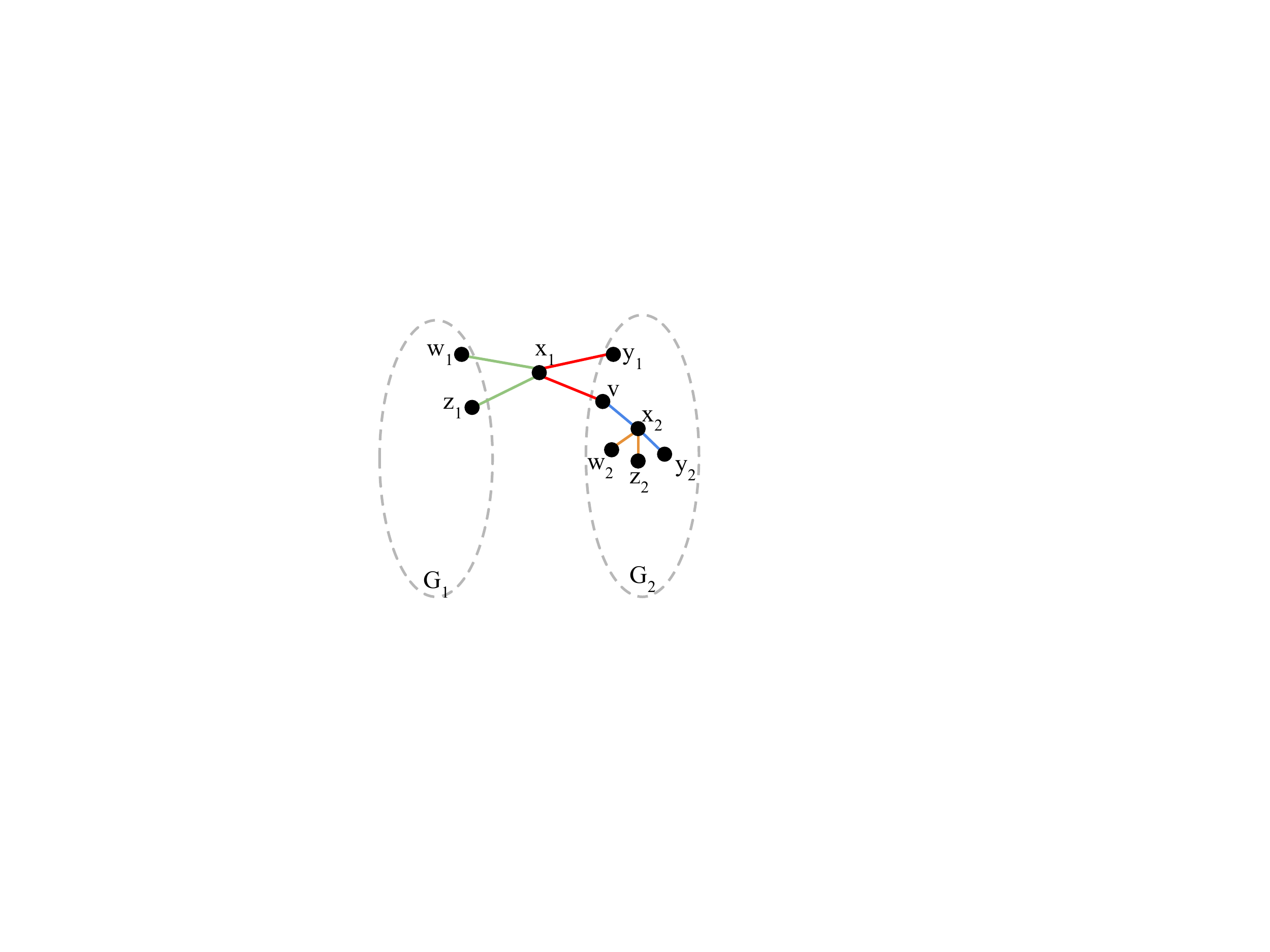}
\caption{The structure of $G\ba C$ in the case that $x_1$ is a cut vertex in Subcase 2.2.1(a).}\label{F:2.2.1axcut}
\end{figure}

The final possibility, then, is that one of $\{w_1, w_2, z_1, z_2\}$ is a cut vertex of Type $X$, suppose wolog that it is $w_1$. Note then that there must be an edge $w_1z_1\in E(G)$, with color $\gamma$, and moreover, vertices $z_1, x_1$ are in the same pseudoblock of $G\ba C$, say $G_1$. But then as the subgraph induced on $x_1, v, x_2, w_2, z_2$ is connected, we must have that all of these vertices are in the same pseudoblock of $G\ba C$, and hence as in the case that $x_1$ was a cut vertex of Type $X$, we must have $w_1$ is a cut vertex of Type $X$ in $G'\ba C$, a contradiction.

Hence, if there is a cycle $C$ in $\{C_1, C_2, \dots C_k\}$ that does not use the vertices $v$ or $x$, then that same cycle can be lifted to $G$, and moreover, its removal yields a good colored graph $G\ba C$.

On the other hand, if there are no rainbow cycles in $\{C_1, C_2, \dots, C_k\}$ that do not use the vertices $v$ or $x$, then we have two cases. Either $k=2$, and there is one cycle using both $v$ and $x$, and another using only $x$, or $k=3$, and there are two cycles using $x$ but not $v$, and one using $v$. The possible arrangements of these cycles are shown in Figure \ref{F:2.2.1fewcycles}.

In the case that $k=3$, let $C_1$ be a cycle using both $v$ and $x$. Then there are two possibilities. Either we have a path $P_1$ from $y_1$ to (wolog) $w_1$ and a path $P_2$ from $y_2$ to (wolog) $w_2$ (see Figure \ref{F:2.2.1fewcyclesa}), or we have paths $P_1$ from $y_1$ to $w_2$ and $P_2$ from $y_2$ to $w_1$ (see Figure \ref{F:2.2.1fewcyclesb}). In either case, both of these paths do not use the colors $\alpha, \beta, \gamma$, or $\delta$, and the color sets of $P_1$ and $P_2$ are disjoint. In addition, we have a third path $P_3$ from $z_1$ to $z_2$, also not using the colors $\alpha, \beta, \gamma,$ or $\delta$.

\begin{figure}[htp]
\begin{subfigure}{0.3\textwidth}
\includegraphics[width=.9\textwidth]{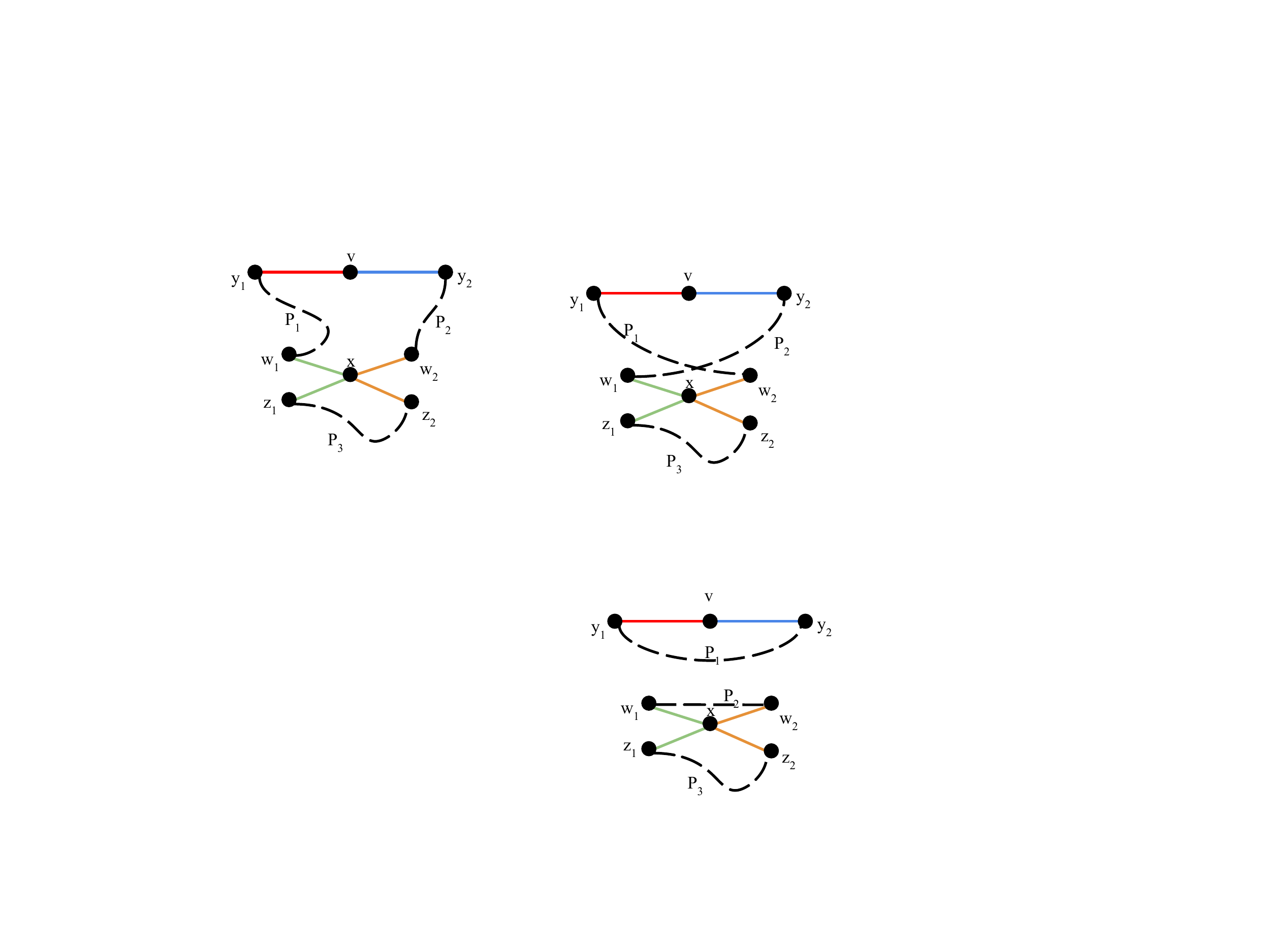}
\caption{The first option, representing one cycle containing $x, z_1, z_2$, and another containing $x, w_1, w_2, y_1, y_2, v$. We note here that $P_1$ and $P_2$ are disjoint paths containing no edges of color $\alpha, \beta, \gamma, \delta$, and $P_3$ is a path containing no edges of color $\alpha, \beta, \gamma, \delta$, but need not be disjoint from $P_1$ or $P_2$.}\label{F:2.2.1fewcyclesa}
\end{subfigure}\hspace{.3in}\begin{subfigure}{0.3\textwidth}
\includegraphics[width=.9\textwidth]{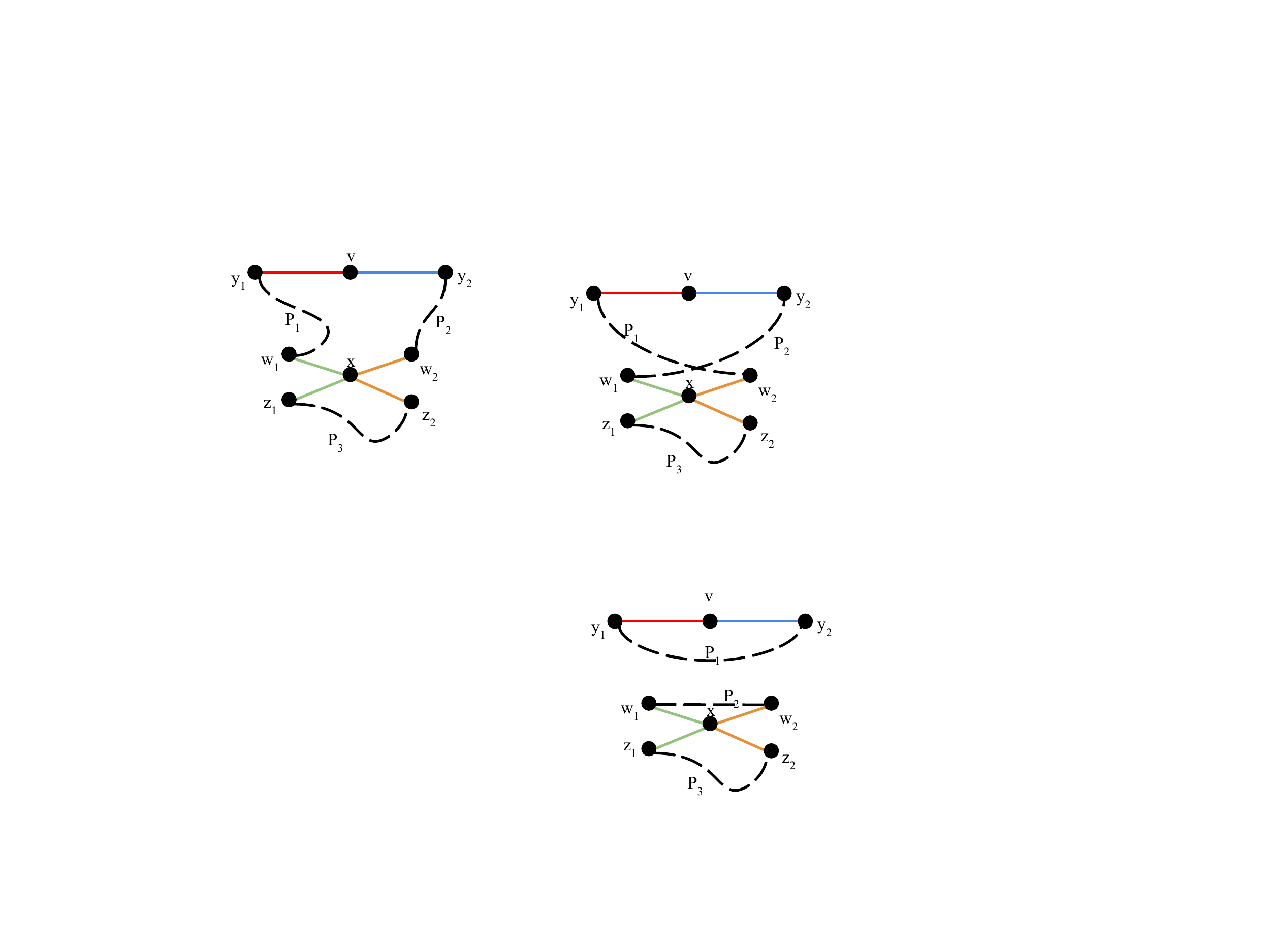}
\caption{The second option, representing one cycle containing $x, z_1, z_2$, and another containing $x, w_1, w_2, y_1, y_2, v$, where here $w_1$ and $w_2$ appear in the opposite orientation as in the first case. As in Figure \ref{F:2.2.1fewcyclesa}, $P_1$ and $P_2$ are disjoint paths containing no edges of color $\alpha, \beta, \gamma, \delta$, and $P_3$ is a path containing no edges of color $\alpha, \beta, \gamma, \delta$, but need not be disjoint from $P_1$ or $P_2$.}\label{F:2.2.1fewcyclesb}
\end{subfigure}\hspace{.3in}\begin{subfigure}{0.3\textwidth}
\includegraphics[width=.9\textwidth]{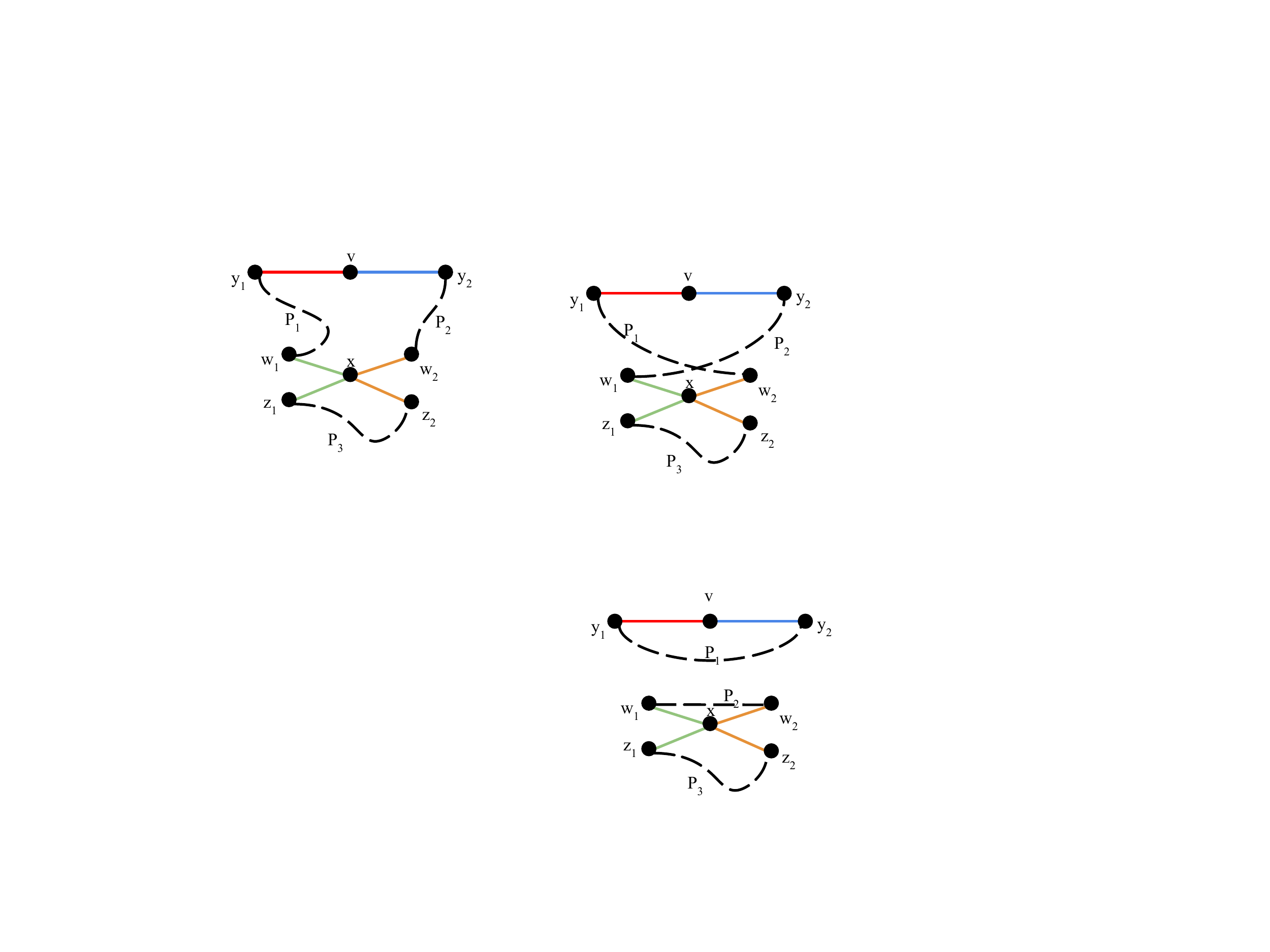}
\caption{The third option, representing three cycles, one of which contains $y_1, v y_2$, one of which contains $w_1, x, w_2$, and the third of which contains $z_1, x, z_2$. We note that none of the indicated paths may use colors $\alpha, \beta, \gamma,$ or $\delta$, but also none of the indicated paths need be disjoint from any of the others.}\label{F:2.2.1fewcyclesc}
\end{subfigure}

\caption{The possible options for Subcase 2.2.1(a) in the event that there are no rainbow cycles in $\{C_1, C_2, \dots, C_k\}$. In these diagrams, we note that any dashed lines represent paths, rather than edges. Moreover, we do not claim that any of these paths are in fact disjoint; any specific disjoint paths are mentioned for that case.}
\label{F:2.2.1fewcycles}
\end{figure}

For the first option, we take the decomposition in $G$ to be three cycles: $(x_1, w_1, P_1, y_1, x_1)$, $(x_2, w_2, P_2, y_2, x_2)$, and $(z_1, x_1, v, x_2, z_2, P_3, z_1)$ (see Figure \ref{F:2.2.1fewcyclesGa}). All of these are rainbow cycles, and cover all edges of $G$. In the second option, we take the decomposition in $G$ to be two cycles: $(x_1, w_1, P_2, y_2, x_2, w_2, P_1, y_1, x_1)$ and $(x_1, v, x_2, z_2, P_3, z_1, x_1)$ (see Figure \ref{F:2.2.1fewcyclesGb}). All of these are rainbow cycles, since $P_1$ and $P_2$ may not repeat colors, and cover all edges of $G$. 

In the case that $k=2$, we have that $G'$ takes the following structure: One rainbow cycle of the form $(y_1, v, y_2, P_1, y_1)$, and two rainbow cycles involving vertex $x$; wolog, these take the form $(w_1, P_2, w_2, x, w_1)$ and $(z_1, P_3, z_2, x, z_1)$ (see Figure \ref{F:2.2.1fewcyclesc}). We note that the path $P_1$ does not use colors $\alpha$ or $\beta$, and the paths $P_2, P_3$ do not use colors $\alpha, \beta, \gamma, $ or $\delta$, as the only edges colored $\alpha$ or $\beta$ are incident to $y_1$ and $y_2$, respectively. We further note that it may not be the case that any of these cycles are disjoint; that is, we may have shared colors in any of these paths.

We then consider the following cycle in $G$: $C=(x_1, w_1, P_2, z_1, x_2, v, x_1)$ (see Figure \ref{F:2.2.1fewcyclesGc}). Clearly this cycle is rainbow, since $P_2$ cannot include any edges of color $\alpha, \beta, \gamma, $ or $\delta$. Moreover, we claim that the graph $G\ba C$ is 2-connected (disregarding any isolated vertices), and hence has no cut vertices of Type $X$. Indeed, suppose that $a, b\in V(G\ba C)$ are nonisolated. If $a, b$ both appear on the cycle $C_1=(x_1, y_1, P_1, y_2, x_2, v, x_1)$, then clearly there are two vertex disjoint paths between them. Likewise, if $a, b$ both appear on the cycle $C_2=(x_1, z_1, P_3, z_2, x_2, v, x_1)$, then again there are two vertex disjoint paths between them. 

Hence, we may suppose that $a$ is on $C_1$, and $b$ is on $C_2$, but neither vertex is on both cycles. We form two vertex disjoint paths between $a$ and $b$ as follows. First, let $c$ be the first vertex of $C_1$ appearing before $a$ in the presentation $(y_1, x_1,v,x_2,y_2,P_1,y_1)$ with $c$ also a member of $C_2$. Note that as $a$ is not a member of $C_2$, $a\neq x_1$, and hence this is well defined, as $x_1$ is a member of $C_2$ and we thus will always choose a vertex between $x_1$ and $a$. Likewise, let $d$ be the first vertex of $C_1$ appearing after $a$ in the presentation $(y_1,x_1,v,x_2,y_2,P_1,y_1)$ with $d$ also a member of $C_2$. As above, $a\neq x_2$, and hence this vertex is well defined. Moreover, as $x_1$ and $x_2$ could both satisfy these conditions, we therefore will have that $c\neq d$. As $c$ and $d$ are both distinct vertices on $C_2$, we can construct two paths $P_c$ and $P_d$ from $c$ to $b$ and $d$ to $b$, respectively, along $C_2$, such that $P_c$ and $P_d$ are internally vertex disjoint. Hence, by concatenating $P_c$ and $P_d$ with the paths along $C_1$ from $a$ to $c$ and $a$ to $d$, we obtain two internally vertex disjoint paths between $a$ and $b$.

Therefore, in the case that $k=2$, we have found a cycle $C$ such that $G\ba C$ is good, as desired.

\begin{figure}[htp]
\begin{subfigure}{0.3\textwidth}
\includegraphics[width=.9\textwidth]{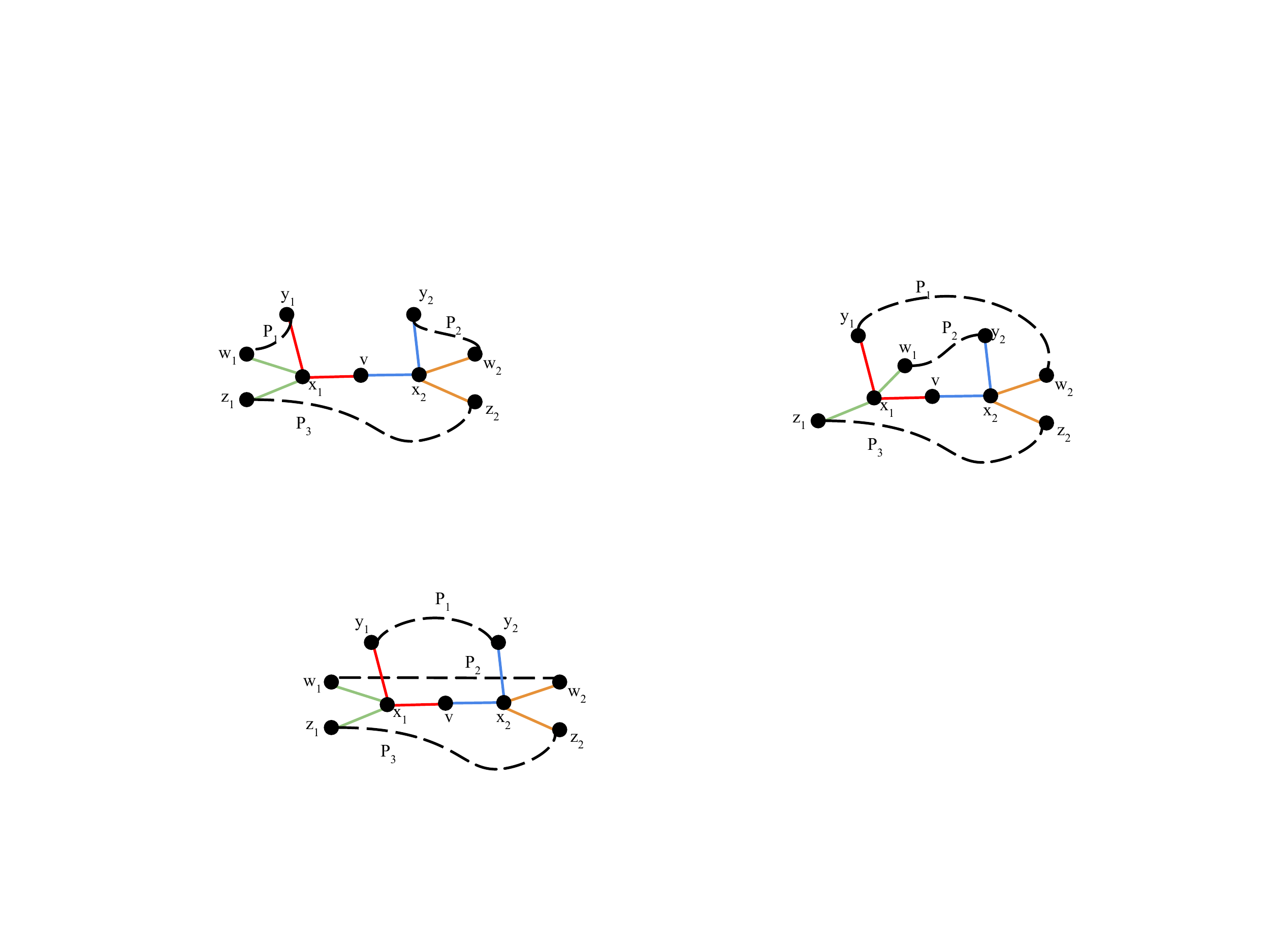}
\caption{}\label{F:2.2.1fewcyclesGa}
\end{subfigure}\hspace{.3in}\begin{subfigure}{0.3\textwidth}
\includegraphics[width=.9\textwidth]{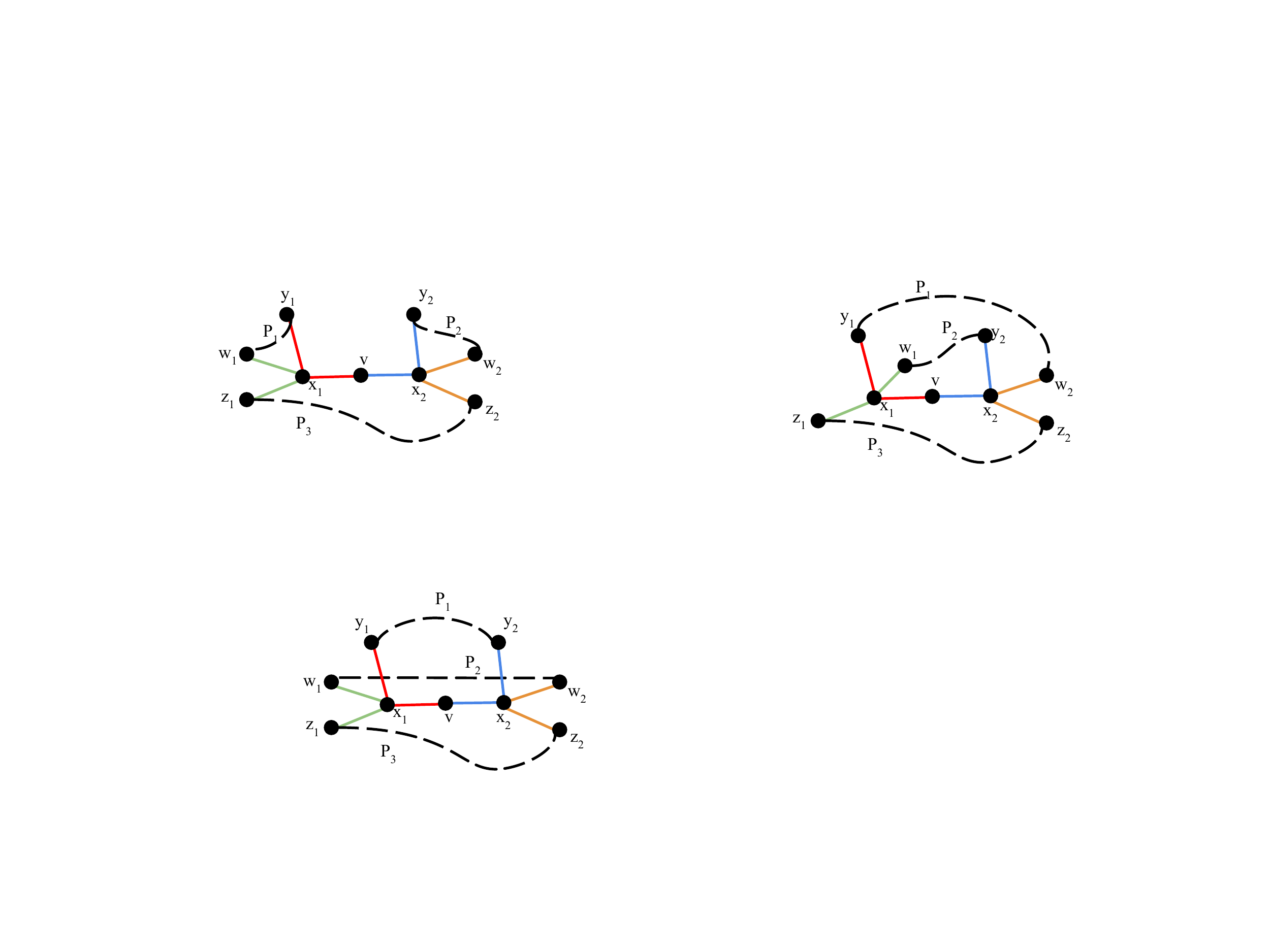}
\caption{}\label{F:2.2.1fewcyclesGb}
\end{subfigure}\hspace{.3in}\begin{subfigure}{0.3\textwidth}
\includegraphics[width=.9\textwidth]{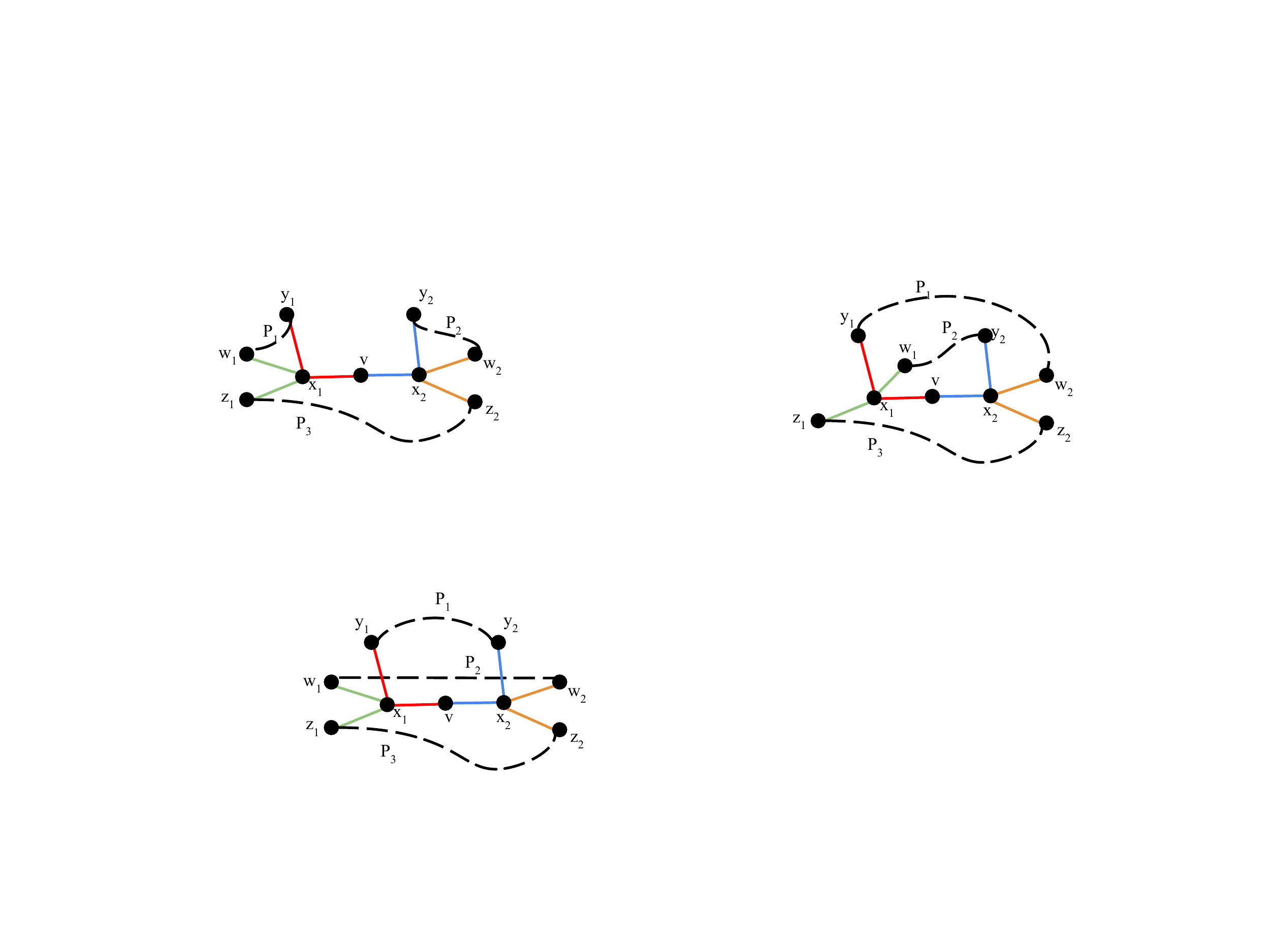}
\caption{}\label{F:2.2.1fewcyclesGc}
\end{subfigure}

\caption{The transformations from $G'$ to $G$ corresponding to the cycle structures shown in Figure \ref{F:2.2.1fewcycles}. We note as above that any dashed lines represent paths, rather than edges, and that we make no assumptions about the disjointness of these paths that are not mentioned in the captions in Figure \ref{F:2.2.1fewcycles}.}
\label{F:2.2.1fewcyclesG}
\end{figure}

Therefore, if $G'$ is good, then there exists a rainbow cycle $C$ in $G$ such that $G\ba C$ is good. Let us now turn to the subcase in which $G'$ is not good.

\quad

\noindent\textbf{Subcase 2.2.1(b).}  $G'$ has a cut vertex of Type $X$.

\quad

Let $t$ be a cut vertex of Type $X$ in $G'$. First, suppose $t=x$. Let $G_1$ and $G_2$ be the pseudoblocks of $G$ at $x$. Wolog, suppose $w_1, w_2\in V(G_1)$, and $z_1, z_2\in V(G_2)$, and $y_1, v, y_2\in V(G_1)$. Let $G_2'$ be the induced subgraph of $G$ having $V(G_2')=V(G_2)\ba\{x\}\cup\{x_2\}$, and let $G_1'$ be the induced subraph of $G$, having $V(G_1')=V(G_1)\ba\{x\}\cup\{x_1\}$. Notice then that all edges of $G$ are present in either $G_1'$ or $G_2'$, and hence we have that $x_2$ is a cut vertex of type $X$ in $G$. As this is impossible, since $G$ is good by hypothesis, $x$ cannot be a cut vertex of type $X$. This structure is illustrated in Figure \ref{F:Case2.2.1cxcut}.
 
Hence, we must have that $t\neq x$. Let us form a decomposition of $G$ into induced subgraphs $G_1, G_2, \dots, G_s$ as follows.

First, let $B_1, B_2, \dots, B_r$ be the blocks of $G$, and let $B$ be the block graph of $G$. Define an equivalence relation $R$ on $\{B_i\}$ as follows: for any $i, j$, if $B_i$ and $B_j$ are in the same component of $B$, let $P$ be the path $B_i=B_{i_0}, B_{i_1}, \dots, B_{i_r}, B_j=B_{i_{r+1}}$ from $B_i$ to $B_j$ in $B$. If, for all $0\leq k\leq r$, we have that the unique vertex in $V(B_{i_k})\cap V(B_{i_{k+1}})$ is not a cut vertex of Type $X$, then we take $B_i \sim_R B_j$. 

Let $G_1, G_2, \dots, G_s$ be the graphs obtained by the unions of each equivalence class under $R$. Then although these graphs may not be 2-connected, they are induced subgraphs, and the intersection between any pair $G_i, G_j$ is either empty or is a cut vertex of Type $X$. Moreover, every cut vertex of Type $X$ in $G$ will be found as the intersection of two such graphs. We shall refer to these graphs as $X$-blocks, and the decomposition as the $X$-block decomposition of $G'$; note that unlike pseudoblocks, this decomposition is unique. In a natural way, then, we may define a forest $T$ with $V(T)=\{G_1, \dots, G_s\}$ and $E(T) = \{ij\ | \ V(G_i)\cap V(G_j)\hbox{ is a cut vertex of Type $X$}\}$. Note that $T$ can also be viewed as a contraction of $B$, where we contract any edge corresponding to a cut vertex that is not of Type $X$. Wolog, let $x\in V(G_1)$, and note that as $x$ is not a cut vertex of type $X$ in $G'$, we must have $w_1, z_1, w_2, z_2\in V(G_1)$ also. Moreover, since no cut vertices of Type $X$ exist in $G$, we must have that $v$ is in a distinct $X$-block of $G'$, say $G_s$. We note that $v$ cannot be a cut vertex, since it is of degree 2 and $G'$ is even, and so we must also have $y_1, y_2\in V(G_s)$.

 \begin{figure}[htp]
\includegraphics[height=.2\textheight]{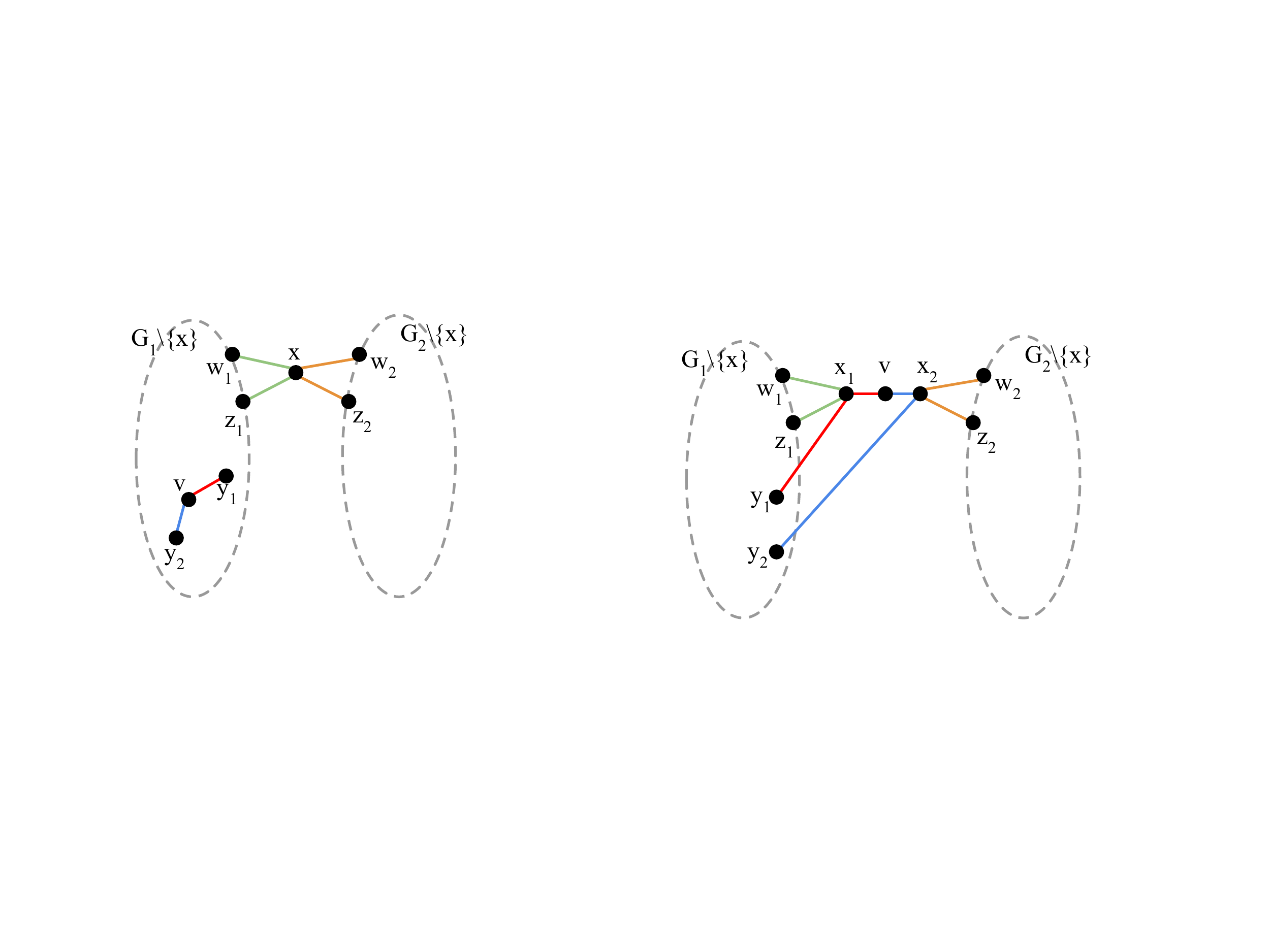}
\hspace{.3in}
\includegraphics[height=.18\textheight]{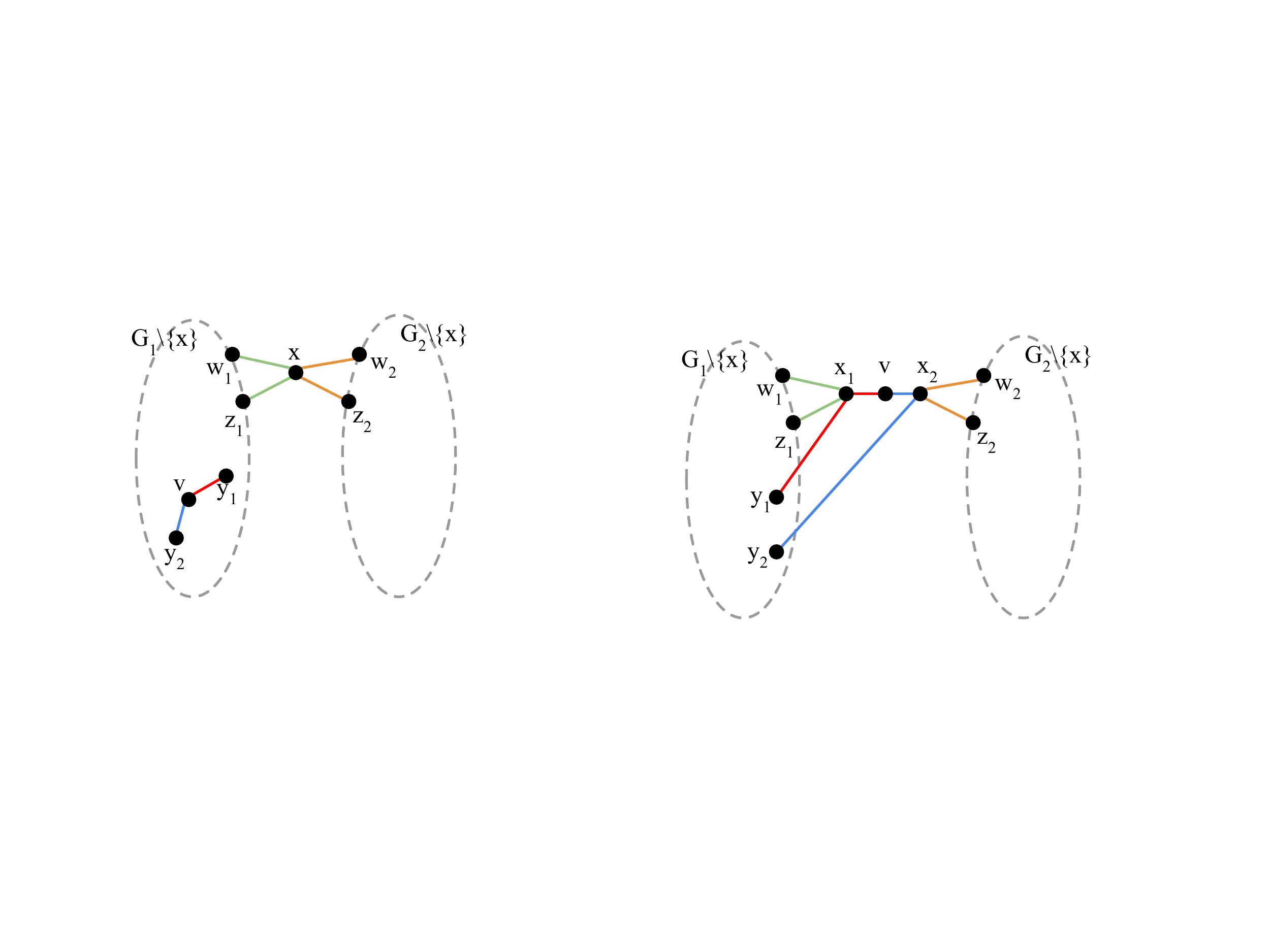}
\caption{The structure in $G'$ (left) and $G$ (right) in the case that $x$ is a cut vertex in $G'$ in Subcase 2.2.1(b).}\label{F:Case2.2.1cxcut}
 \end{figure}

\begin{claim}
$T$ is a path, and its endpoints are $G_1$ and $G_s$.
\end{claim}

\begin{proof}[Proof of Claim] 

Notice that if we consider the $X$-block decomposition of $G$, it can be obtained from the $X$-block decomposition of $G'$ by merging together $G_1, G_s$, and (if $G_1$ and $G_s$ are in the same component of $T$) any vertices on the path between $G_1$ and $G_s$, as in $G$, these two $X$-blocks will be joined by the edges $x_1v$ and $x_2v$.

Moreover, since $G$ is either good or almost good, the $X$-block decomposition of $G$ is a single vertex. Hence, if $T$ is disconnected, then $T$ consists of exactly two $X$-blocks, namely $G_1$ and $G_2=G_s$. But then there is no cut vertex of Type $X$ in either of these $X$-blocks, a contradiction. Thus, $T$ is connected, and as merging $G_1$ with $G_s$ and all other vertices on the path between them, we have that every vertex in $T$ appears in the path between $G_1$ and $G_s$; that is, $T$ must consist of a single path, with endpoints $G_1$ and $G_s$. 
\end{proof}

Hence, we have a canonical labeling of the vertices $G_1, G_2, \dots, G_s$, by traveling along the path. Note moreover that both $G_1$ and $G_s$ are almost good colored graphs, as each contains no cut vertices of Type $X$, but will have a bad vertex at the intersection $V(G_1)\cap V(G_2)$ or $V(G_{s-1})\cap V(G_s)$, respectively.

Let $t=V(G_1)\cap V(G_2)$, so that $t$ is the bad vertex of $G_1$. Applying the induction hypothesis on $G_1$, there exists a decomposition of the edges of $G_1$ into cycles $\mathcal{C}=\{C_1, C_2, \dots, C_k\}$, such that $C_1$ is almost rainbow, and the remaining cycles are rainbow. Moreover, by Lemma \ref{setcycles}, $G_1\ba\{C_i\}$ has no cut vertices of Type $X$ for any $i$. 

Note that as $x$ is of Type II in $G_1$, we must have two cycles in $\mathcal{C}$ that use the vertex $x$. Moreover, at most one of these two cycles may use the vertex $t$, since $t$ is present in exactly one cycle in $\mathcal{C}$. Wolog, suppose that $x\in V(C_2)$, so that $t\notin V(C_2)$, since $t$ appears in the cycle $C_1$. Moreover, we may assume wolog that the adjacent vertices to $x$ in $C_2$ are $w_1$ and $w_2$. Write $C_2 = (x, w_1, u_2, \dots, u_\ell, w_2, x)$. This structure is shown in Figure \ref{F:2.2.1cXblock}. We note that as $u_i\neq y_j$ for any $j$, and $C_2$ is rainbow, that the path $w_1, u_2, \dots, u_ell, w_2$ does not use any of the colors $\alpha, \beta, \gamma$, or $\delta$.

\begin{figure}[htp]
\includegraphics[scale=.7]{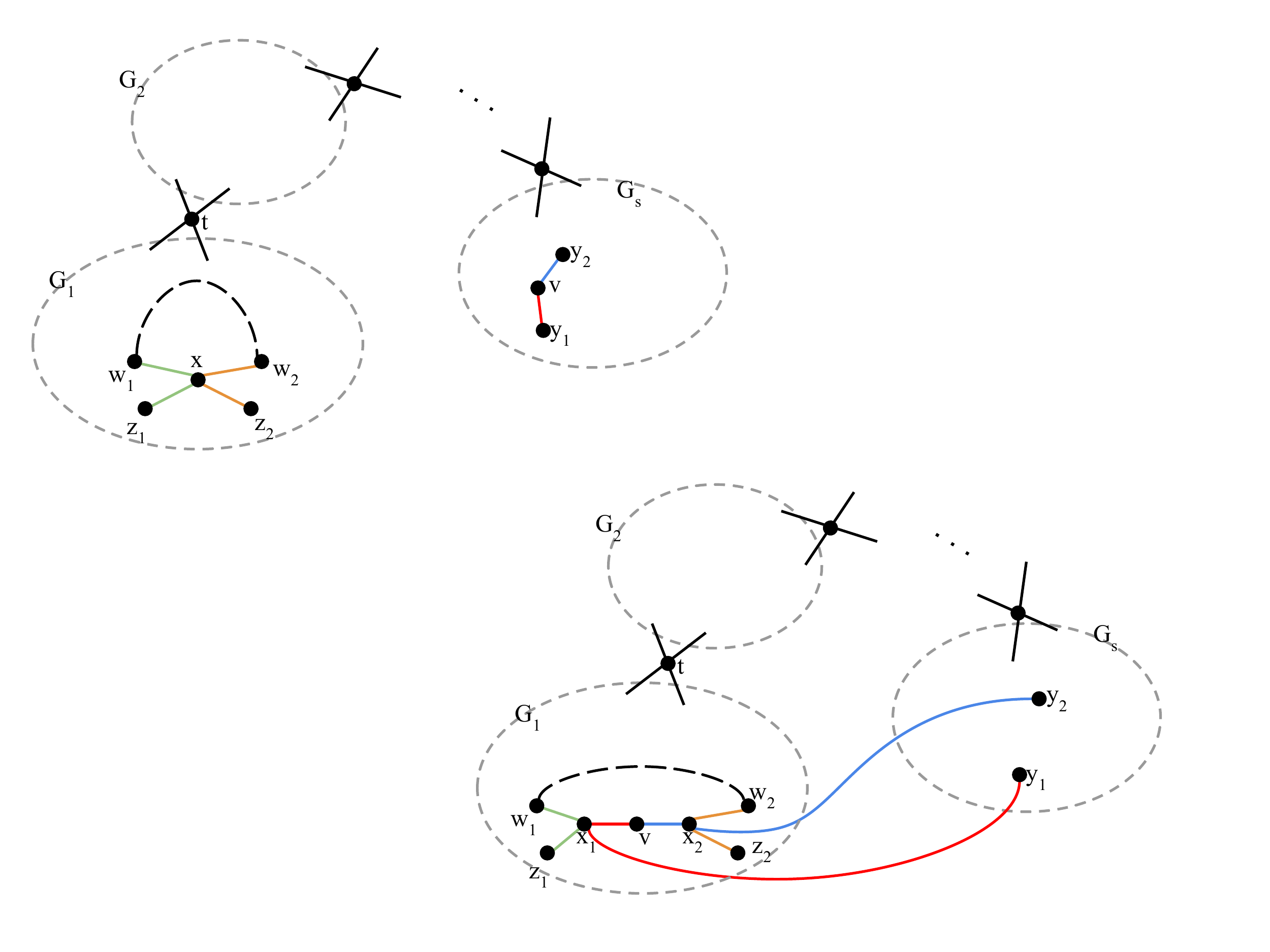}
\caption{The structure of the $X$-block decomposition in $G'$. Here, the dashed line between $w_1$ and $w_2$ indicates the remaining vertices of the cycle $C_2$, that is, the vertices $u_2, u_3, \dots, u_\ell$. Note that no edges of this path may be colored $\alpha, \beta, \gamma$, or $\delta$, due to the constraints on the number of vertices incident to each color set in $G$.}\label{F:2.2.1cXblock}
\end{figure} 

Let $C$ be the cycle in $G$ defined by $C=(v, x_1, w_1, u_2, \dots, u_\ell, w_2, x_2, v)$, see Figure \ref{F:2.2.1cXblockG}. As noted above, since the path from $w_1$ to $w_2$ along the $u_i$ does not use any colors among $\alpha, \beta, \gamma, \delta$, the cycle $C$ is rainbow. By Observation \ref{heredity}, to show that $G\ba C$ is good, we need only check the condition that $G\ba C$ contains no cut vertices of Type $X$. 

\begin{figure}[htp]
\includegraphics[scale=.7]{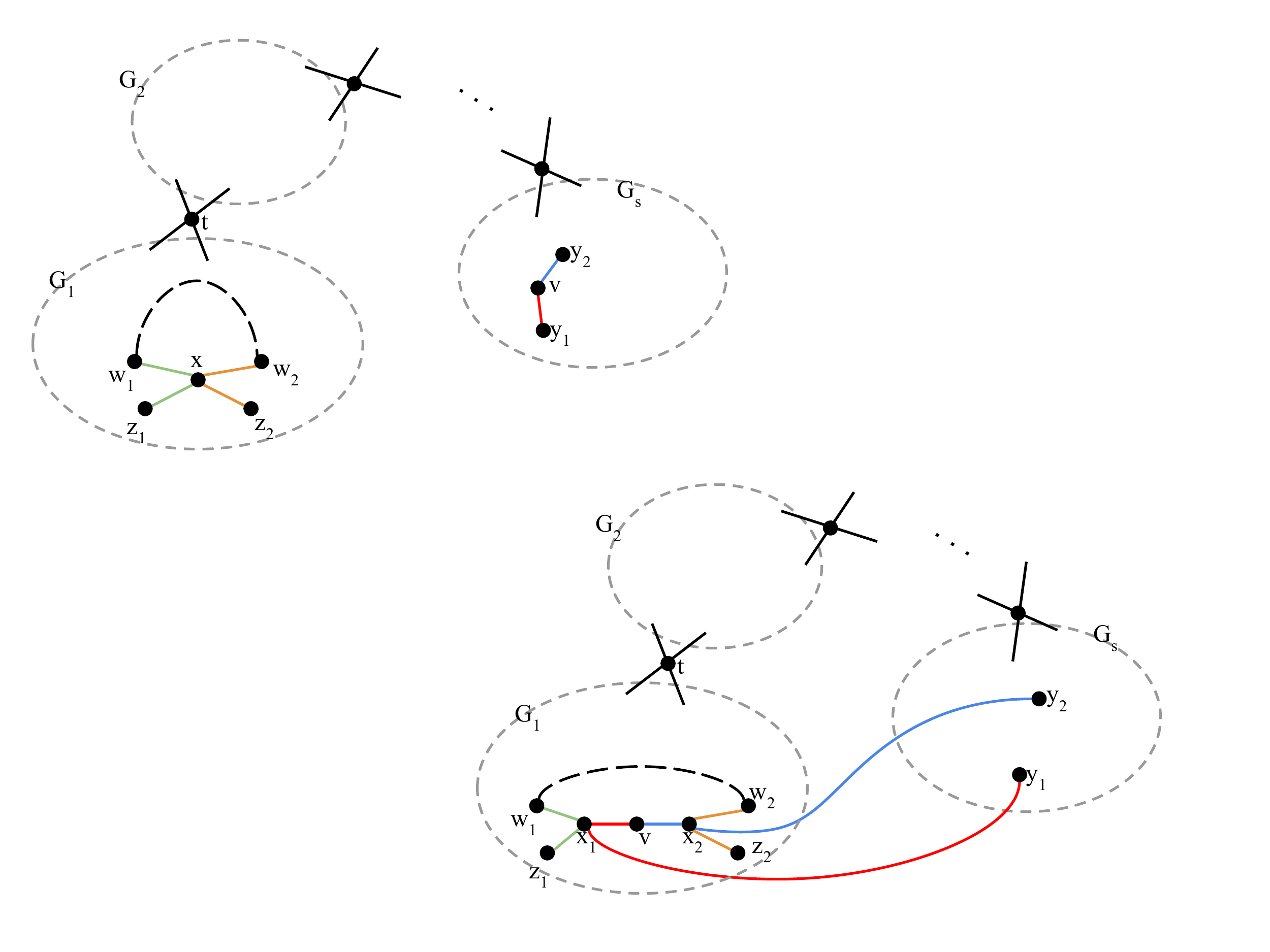}
\caption{The structure of $G$, taking into account the $X$-blocks from $G'$. As noted above, since no edges on the dashed path between $w_1$ or $w_2$ have color $\alpha, \beta, \gamma, $ or $\delta,$ the cycle $C=(v, x_1, w_1, u_2, \dots, u_\ell, w_2, x_2, v)$ is rainbow in $G$.}\label{F:2.2.1cXblockG}
\end{figure} 

To that end, let us suppose that $r$ is a cut vertex of Type $X$ in $G\ba C$. We note that as $x_1, x_2, y_1, y_2$ are all Type I vertices in $G\ba C$, that $r$ may not be one of these nodes.

Suppose that $r\in V(G_1)$, and we first suppose that $r\neq t$. First, we note that if $r$ were a cut vertex in $G_1$, it was not a cut vertex of Type $X$, that is, if $G_1'$ and $G_2'$ are pseudoblocks of $G_1$ at $r$, then $r$ must have two incident edges in $G_1'$ of different colors, and two incident edges in $G_2'$ of different colors. Moreover, by recalling that $G_1$ is connected, we must have that $G_1'$ is connected and $G_2'$ is connected. But the cycle $C$ may only intersect one of $G_1'$ and $G_2'$, say $G_1'$ wolog. But then $r$ is not a cut-vertex of Type $X$ in $G\ba C$, as removing $C$ still yields an induced connected subgraph $G_2'$ having two incident edges to $r$ with two different colors. Thus, $r$ is not a cut vertex in $G_1$.

Now, as $r$ is a cut vertex in $G\ba C$ it must be that there are a pair of vertices, $a, b\in V(G_1)$ such that the only paths from $a$ to $b$ in $G\ba C$ use the vertex $r$. Moreover, since $G_1\ba C_2$ is almost good, it must be that $r$ is not a cut vertex in $G_1\ba C_2$. Hence there is a path $P'$ in $G_1\ba C_2$ from $a$ to $b$ not using the vertex $r$ that is not present in $G\ba C$. We note that such a path must have used, as a subpath, the length two path $z_1xz_2$, as that is the only path that has been destroyed in the transformation from $G'\ba C_2$ to $G\ba C$. 

Let $P$ be a path in $G_s$ between $y_1$ and $y_2$, not using the vertex $v$. Note that such a path must exist, as $v$ is of degree 2 in $G_s$, and hence cannot be a cut vertex. Then we may form a new path in $G\ba C$ from $a$ to $b$, and not using the vertex $r$, by replacing the length two path $z_1xz_2$ in $P'$ with the path $z_1x_1y_1Py_2x_2z_2$. Hence, $r$ is not a cut vertex in $G\ba C$. This situation is illustrated in Figure \ref{F:2.2.1cG_1cut}.

\begin{figure}[htp]
\includegraphics[scale=.7]{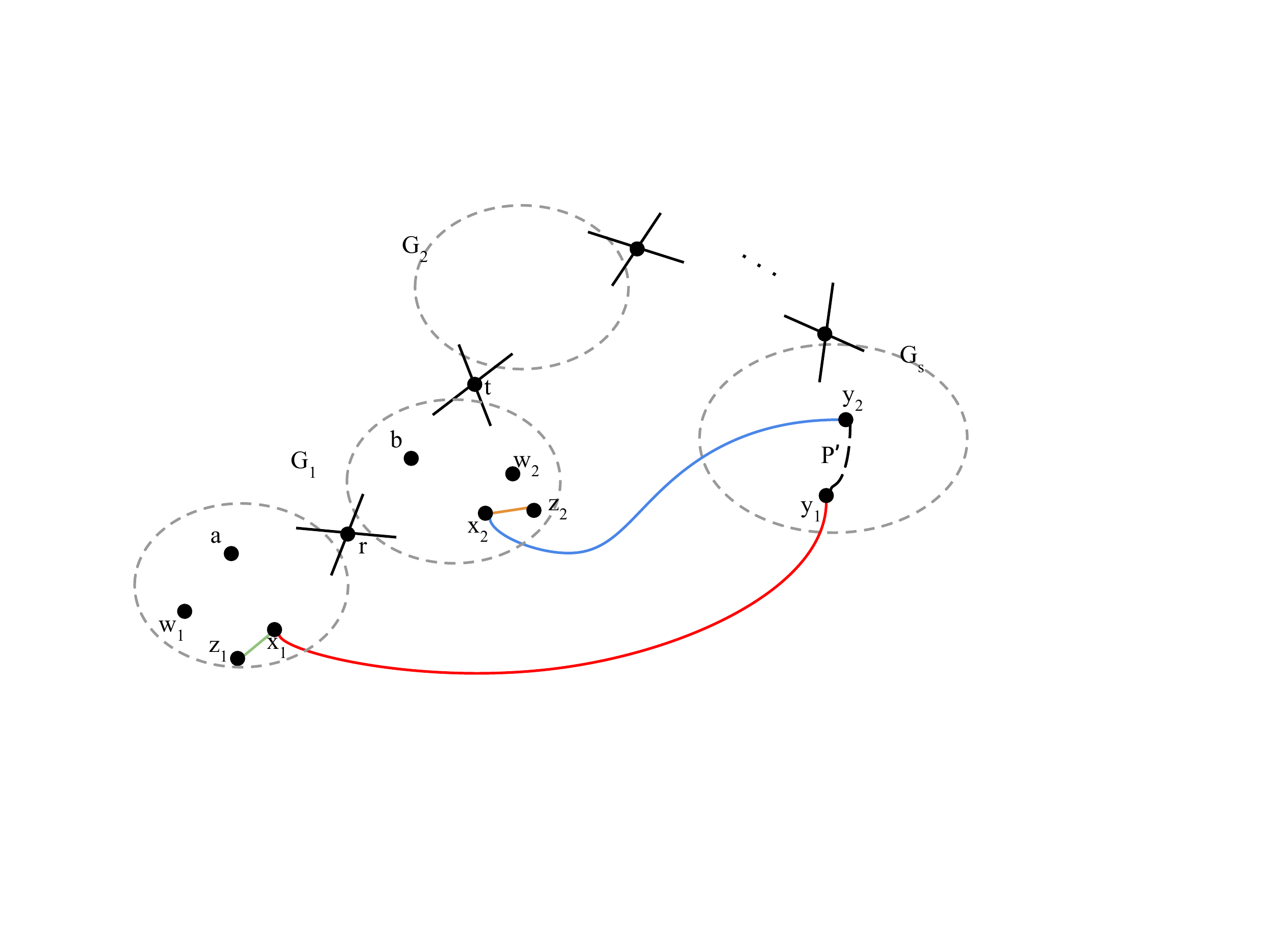}
\caption{The structure of $G\ba C$ in the case that $r\in V(G_1)$, with $r\neq t$, is a cut vertex of Type $X$ in $G\ba C$. Note that the first two ovals indicate the pseudoblocks of $G_1$ at $r$, and the remaining ovals indicate the $X$-block structure in $G'$. }\label{F:2.2.1cG_1cut}
\end{figure}

Therefore, any cut vertex of Type $X$ in $G\ba C$ must be from $G_i$, with $i>1$, or equal to $t$. Note that as we have not removed any edges or vertices from any $X$-block other than $G_1$, in fact we must have that any cut vertex of Type $X$ in $G\ba C$ must occur as the vertex $r=V(G_i)\cap V(G_{i-1})$ for some $i$. As we will add edges between $G_s$ and $G_1$ when we transform from $G'\ba C_2$ to $G\ba C$, we note that it must be the case that $t$ itself is a cut vertex of Type $X$. Moreover, as we will add the edges between $G_1$ and $G_s$ when we transform from $G'$ to $G$, the component of $G\ba C$ containing $x$ will be connected to all of $G_2, G_3, \dots, G_s$. Since $G_2$ also contains $t$, we must have that the removal of the cycle $C_2$ from $G_1$ disconnects $G_1$, in such a way that the vertices $z_1, x, z_2$ are in one component and the bad vertex $t$ is in another component (see Figure \ref{F:2.2.1ctacutsetup}).

\begin{figure}[htp]
\includegraphics[scale=.7]{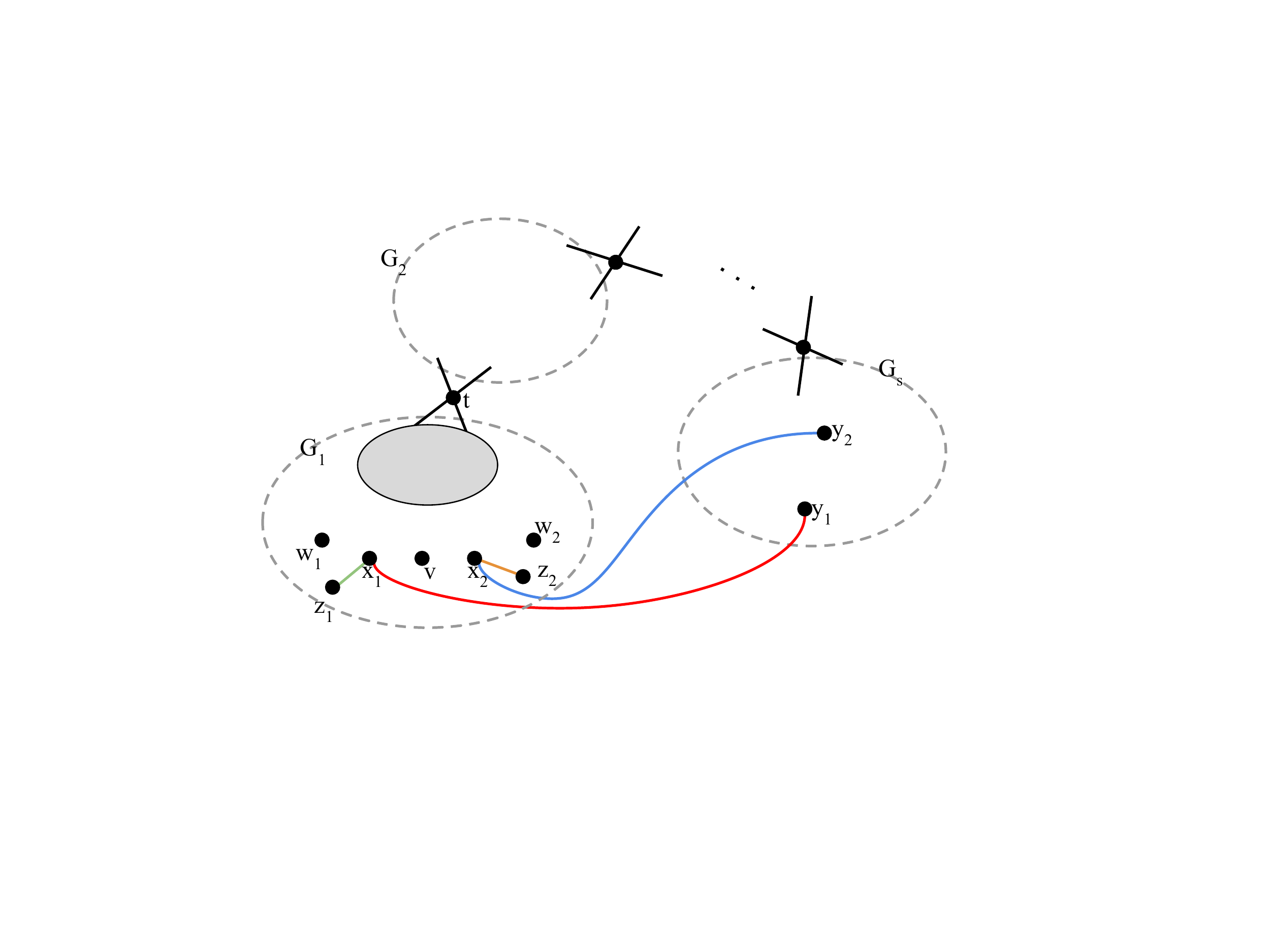}
\caption{ The structure of $G$ in the case that $t$ is a cut vertex of Type $X$ in $G\ba C$. Note that in order for the removal of $t$ to disconnect $G$, we must have that there is a collection of vertices (indicated here by the grey oval) in $G_1$ such that the removal of $t$ disconnects these vertices from the remainder of the graph. Since $x_1$ and $x_2$ are both connected via paths to $G_2$, we cannot have $x_1$ or $x_2$ in the component of $G_1\ba C$ containing $t$. Moreover, although it will not be relevant, none of $w_1, w_2, z_1,$ or $z_2$ can be connected to $t$ in $G\ba C$, as $w_i$ are both either isolated or adjacent to $z_i$, and $z_i$ are both adjacent to $x_i$.}\label{F:2.2.1ctacutsetup}
\end{figure}

Let $C_j\neq C_2$ be the other cycle in $\mathcal{C}$ using the vertex $x$. Note that since $C_j$ is edge disjoint from $C_2$, and $G_1\ba C_2$ has $t$ and $x$ in distinct components, we must have that $t$ is not used in the cycle $C_j$. 

Now, since $G_1$ is connected, there must be a path in $G_1$ from $x$ to $t$. Moreover, since removing $C_2$ disconnects $G_1$, we must have that there exists a vertex $a$ on $C_2$ such that there is a path from $x$ to $t$ that follows $C_2$ to the vertex $a$, then leaves $C_2$ and takes a path to $t$. Note that $a\neq z_1$ and $a\neq z_2$, as otherwise all edges on a path from $x$ to $t$ are present in $G_1\ba C_2$, and as noted above these vertices must be in distinct components of $G_1\ba C_2$. Note further that no edges on this path are used in $C_j$, as all edges are either members of $C_2$ or in a different component of $G_1\ba C_2$ from $C_j$. This is illustrated in Figure \ref{F:2.2.1ctacut}.

Now, notice that $C_j$ must be entirely contained in the component of $G_1\ba C_2$ containing $x$. Hence, we may replace $C_2$ with $C_j$ and repeat this argument; upon so doing, we must have that $x$ and $t$ are in the same component of $G_1\ba C_j$, since we may obtain a path between the two by following the path indicated above. But then, using $C_j$ in place of $C_2$, we have that $t$ is not a cut vertex of Type $X$ in $G\ba C$. Moreover, the previous analysis of cut vertices is unaffected, and hence with this new cycle $C$, we have that $G\ba C$ is a good colored graph.

\begin{figure}[htp]
\includegraphics[scale=.7]{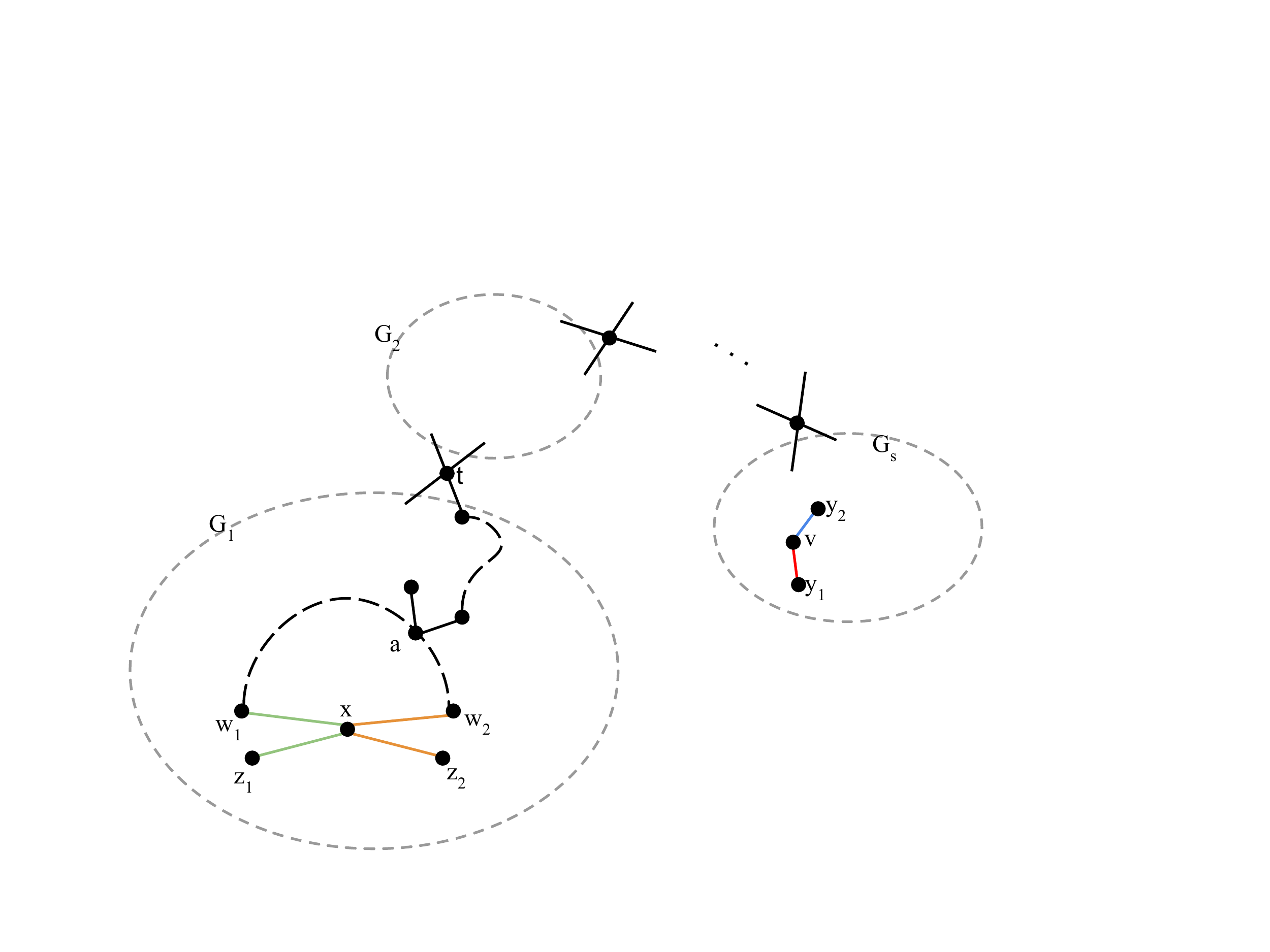}
\caption{ The structure of $G'$ in the case that $t$ is a cut vertex of Type $X$ in $G\ba C$ in Subcase 2.2.1(b).}\label{F:2.2.1ctacut}
\end{figure}

We now turn our attention to the final subcase, that in the basic structure necessary for Subcase \ref{C:nonsingular} illustrated in Figure \ref{F:2.2setup}, the neighbor sets of $x_1$ and $x_2$ are not disjoint.

\begin{sscase}The sets $\{y_1, w_1, z_1\}$ and $\{y_2, w_2, z_2\}$ are not disjoint.\label{C:overlap}
\end{sscase}
We consider here several possibilities, that cover all possible overlaps between these two sets, wolog. 

\quad

\noindent\textbf{Subcase 2.2.2(a).} $\{w_1, z_1\}$ is not disjoint from $\{w_2, z_2\}$.

\quad
Without loss of generality, let us suppose that $w_1=w_2$.
Let $C=(x_1,v,x_2,w_1,x_1)$. Note that $C$ is a rainbow cycle in $G$. Hence, we need only verify that $G\ba C$ contains no cut vertices of Type $X$. 

First, note that upon removing this cycle, we have that $x_1$ and $x_2$ are both vertices of Type I, and hence cannot be vertices of Type $X$. Suppose that one of $z_1, z_2$ is a cut vertex of Type $X$, wolog, suppose it is $z_1$. Note that this immediately implies that $z_1$ is a vertex of Type II in $G$, and hence the edge $z_1w_1$ must also be present, with color $\gamma$. But then $w_1=w_2$ is also a vertex of Type II in $G$, and therefore, the edge $z_2w_1$ must also be present, with color $\delta$ (this also implies that we cannot have $z_1=z_2$). However, this immediately implies that in any pseudoblock decomposition of $G\ba C$, the vertices $y_1, y_2, x_1, x_2, w_1, z_2$ are all in the same pseudoblock. Therefore, it must have been the case that $z_1$ was a cut vertex of Type $X$ in $G$, which is impossible (see Figure \ref{F:2.2.2az_1cut}).

\begin{figure}[htp]
\includegraphics{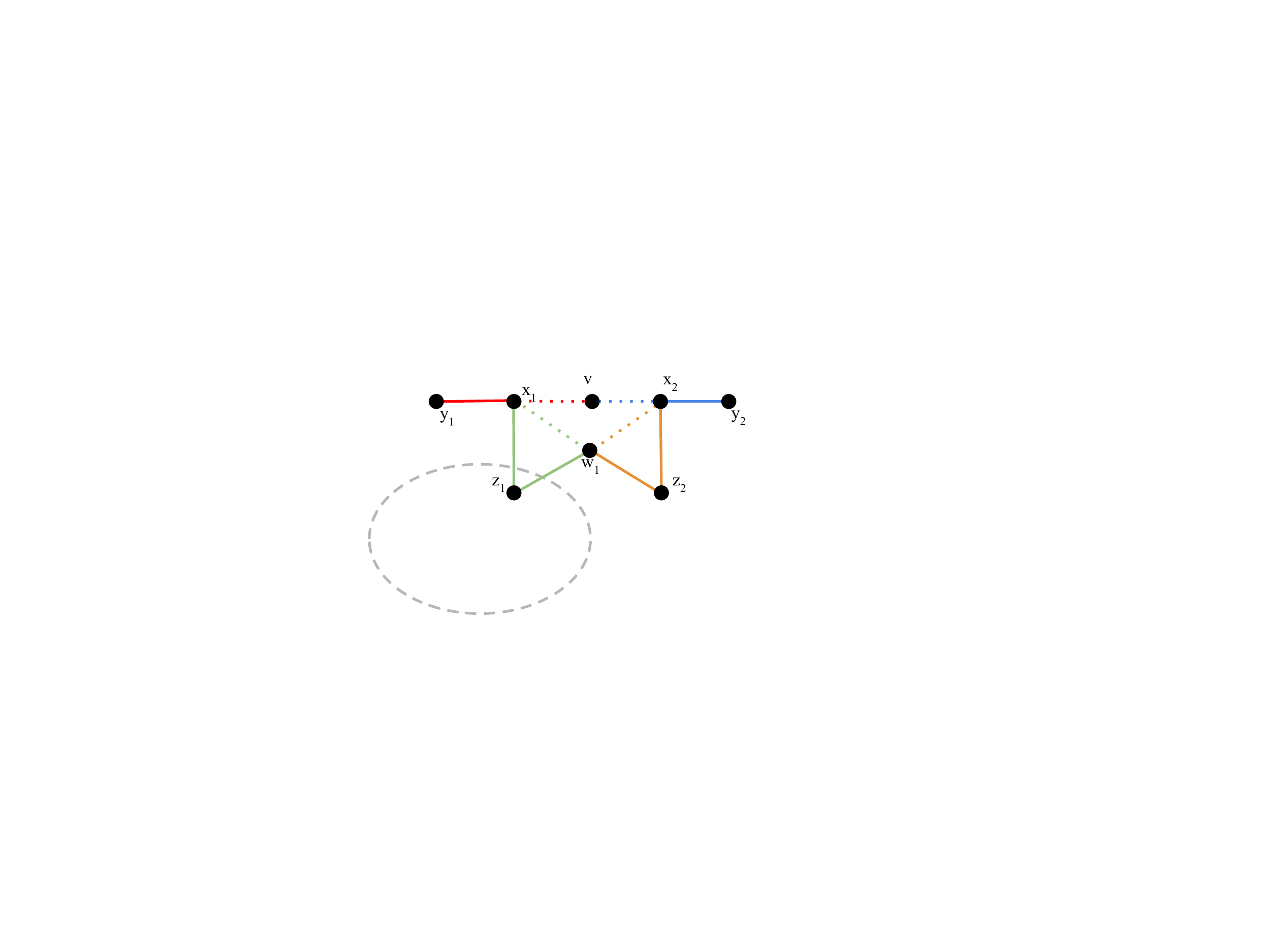}
\caption{The structure of $G$ in Subcase 2.2.2(a), when $z_1$ is a cut vertex of Type $X$. Here, the dashed gray oval represents one $X$-block of $G\ba C$, and the dotted lines indicate the rainbow cycle to be removed. Note that upon removing that cycle, we must have that the remaining labeled vertices are all in the same $X$-block of $G\ba C$. We include in this case the possibility (not pictured) that either $y_1=z_2$ or $y_2=z_1$.}\label{F:2.2.2az_1cut}
\end{figure}

Hence, there must be a cut vertex of Type $X$ that is not one of our heretofore labeled nodes. Suppose that $t$ is such a cut vertex, and let us take $G_1$ and $G_2$ to be pseudoblocks of $G$ at $t$. Note that we must have some vertices among $x_1, x_2, y_1, y_2, z_1, z_2, w_1, w_2$ in each of $G_1$ and $G_2$. Moreover, if $w_1=w_2$ is a Type II vertex in $G$, then the subgraph induced on $\{x_1, x_2, y_1, y_2, z_1, z_2, w_1, w_2\}$ is connected in $G\ba C$, as we would require $z_1w_1, z_2w_1\in E(G)$, and hence this is impossible. Therefore, it must be that $w_1$ is a Type I vertex in $G$. Similarly, if the sets $\{y_1, z_1\}$ and $\{y_2, z_2\}$ are not disjoint, we would also have this subgraph connected in $G\ba C$, and hence this is also impossible. Therefore, wolog, we have $x_1, z_1, y_1\in V(G_1)$ and $x_2, z_2, y_2\in V(G_2)$.

Note that in this situation, we must have that all of $y_1, y_2, z_1, z_2$ are vertices of Type I in $G$. Let their heretofore unlabeled neighbors be $\hat{y}_1, \hat{y}_2, \hat{z}_1, \hat{z}_2$, respectively. Note that every rainbow path between $G_1$ and $G_2$ either passes through $t$, or includes one of the following subpaths: $P_y=\hat{y}_1, y_1, x_1, w_1, x_2, y_2, \hat{y}_2$ or $P_z=\hat{z}_1, z_1, x_1, v, x_2, z_2, \hat{z}_2$. This structure is illustrated in Figure \ref{F:2.2.2aGcut}, left.

\begin{figure}[htp]
\includegraphics[scale=.8]{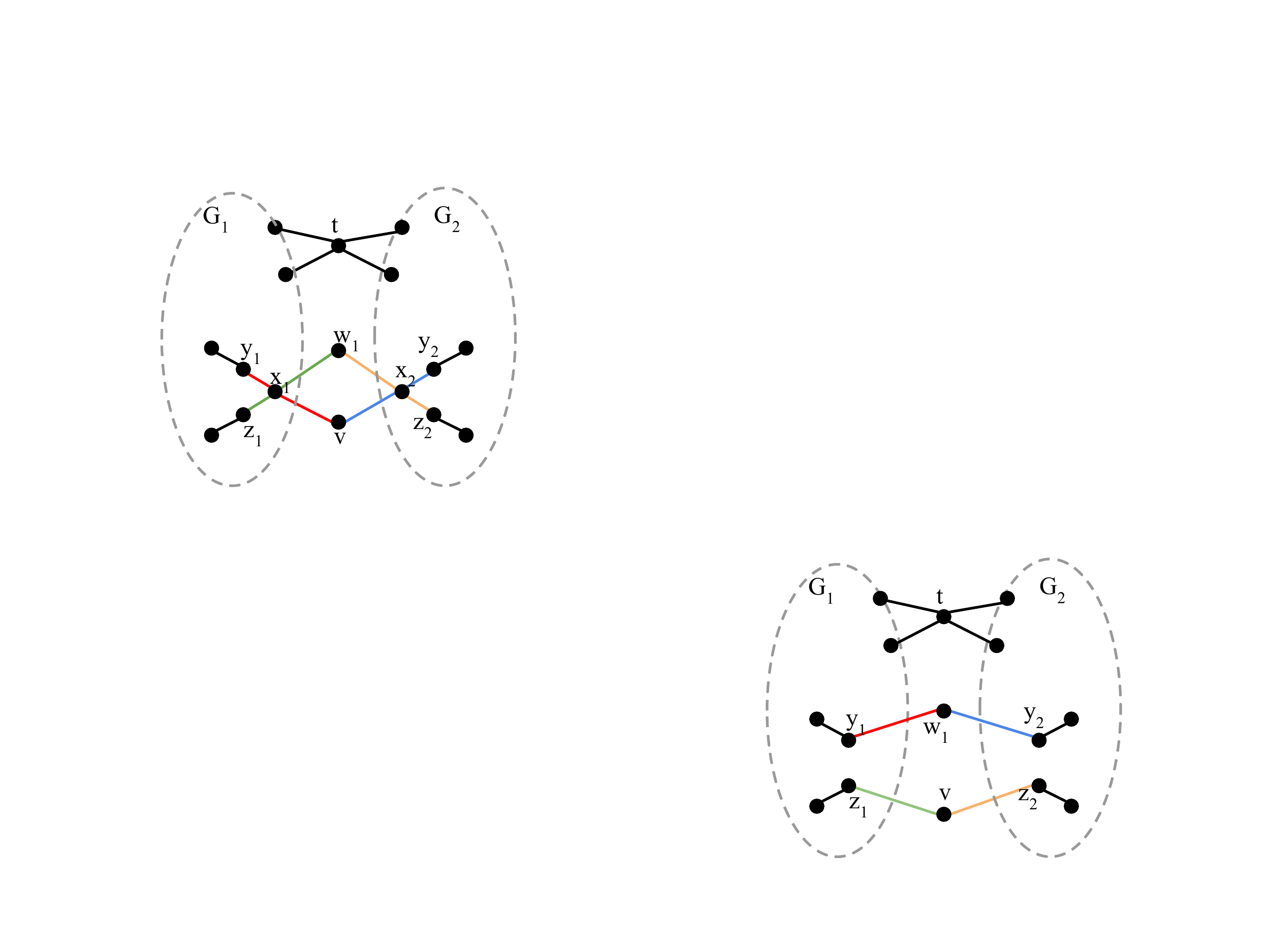}\hspace{.5in}
\includegraphics[scale=.8]{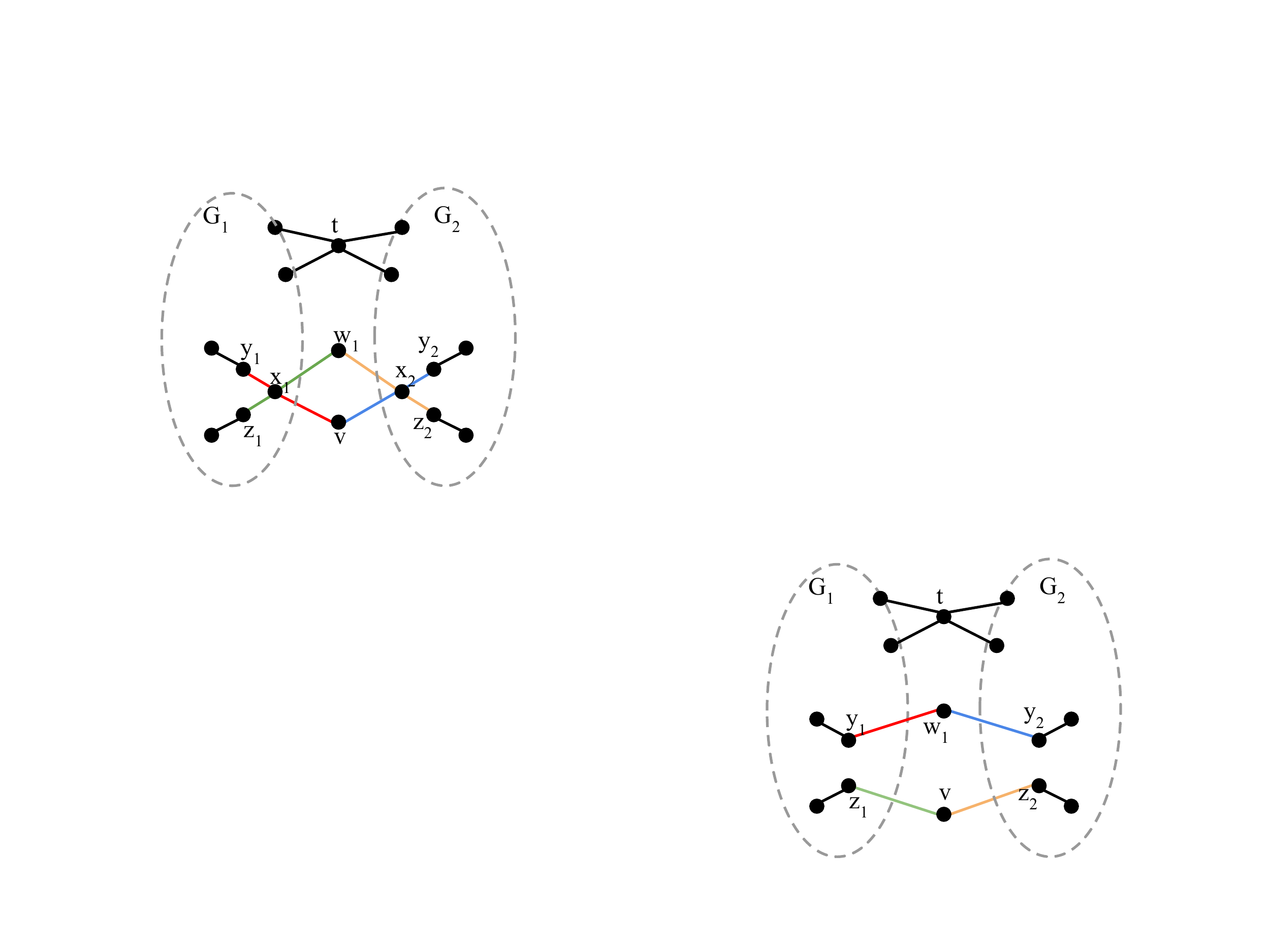}
\caption{The structure of $G$ (left) and $G'$ (right) in Subcase 2.2.2(a), in the case that there is a cut vertex in $G\ba C$ of Type $X$ other than $z_1$. Note that the unlabeled vertices incident to $y_1, y_2, z_1,z_2$ are $\hat{y}_1, \hat{y}_2, \hat{z}_1, \hat{z}_2$, respectively. We further note that these vertices may not be distinct.}\label{F:2.2.2aGcut}
\end{figure}

Create a new graph $G'$ from $G$ as follows:
\begin{itemize}
\item remove the edges $\{x_1y_1, x_1z_1, x_1v, x_1w_1, x_2y_2, x_2z_2, x_2v, x_2w_1\}$, that is, all edges colored $\alpha, \beta, \gamma, $ or $\delta$.
\item remove the vertices $x_1, x_2$.
\item add the edges $y_1w_1, y_2w_1, z_1v, z_2v$. Color these edges with $\alpha, \beta, \gamma, \delta$, respectively.
\end{itemize}

This structure $G'$ is illustrated in Figure \ref{F:2.2.2aGcut}, right. Note that clearly, we have created no triangles in $G'$, and moreover, there can be no additional cut vertices of Type $X$. As all other properties are clear from construction, we thus have that $G'$ is a good colored graph, having strictly fewer vertices than $G$. Hence, we may apply the induction hypothesis to obtain a rainbow cycle $C'$ in $G'$ that uses the vertex $y_1$; note that such a rainbow cycle must be present as we may decompose all edges of $G'$ into rainbow cycles.

As $y_1$ is of Type I in $G'$, we must have that $C'$ contains the entire path $\hat{y}_1, y_1, w_1, y_2, \hat{y}_2$, and we may thus replace this path by $P_y$ in $G$ to form a new rainbow cycle $C$. Moreover, $G\ba C$ can be obtained from $G'\ba C'$ by subdividing the edges $z_1v$ by $x_1$ and $vz_2$ by $x_2$. Hence, $G\ba C$ can contain no cut vertices of Type $X$, and therefore, $G\ba C$ is good. 

Therefore, if $w_1=w_2$, then $G$ contains a rainbow cycle $C$ such that $G\ba C$ is good.

We note that our analysis in this case, did not rely on the fact that $\{y_1, z_1\}$ and $\{y_2, z_2\}$ are disjoint, although this was a consequence of the presence of any cut vertices of Type $X$ in $G\ba C$. Hence, we may assume for all remaining cases that $\{w_1, z_1\}$ is disjoint from $\{w_2, z_2\}$.

\quad

\noindent\textbf{Subcase 2.2.2(b).} $y_1=y_2$.

\quad

Here, we have a rectangle $(x_1,y_1,x_2,v,x_1)$, and the only edges colored $\alpha$ or $\beta$ are in this rectangle. Contract the rectangle to form a new graph $G'$, having a single Type II vertex $x$, with neighbors $w_1, z_1, w_2$, and $z_2$, and recolor these edges as $\gamma,\gamma, \delta, \delta$, respectively; see Figure \ref{F:2.2.2b}.

\begin{figure}[htp]
\includegraphics[scale=.8]{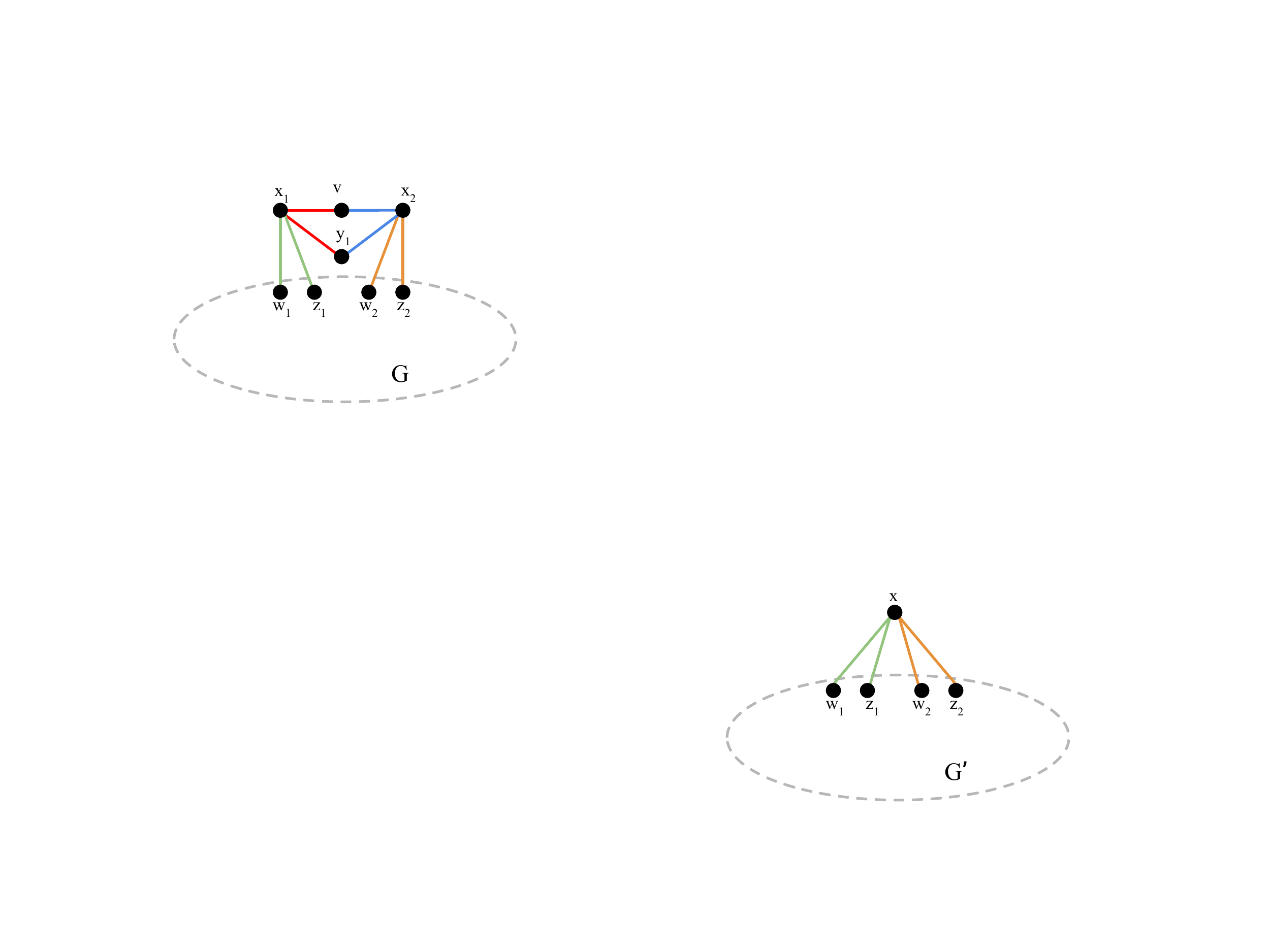}\hspace{.3in}
\includegraphics[scale=.8]{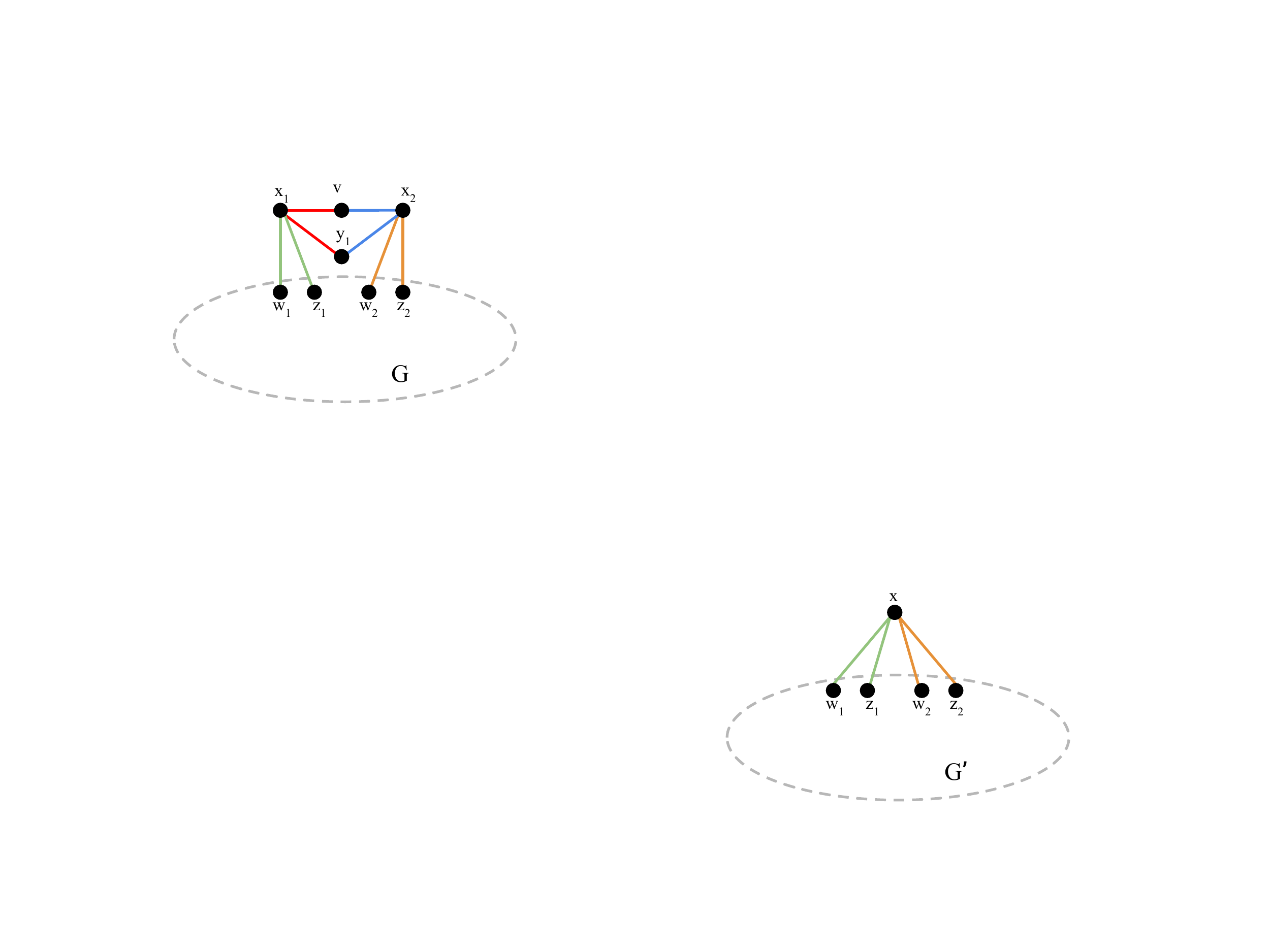}
\caption{The structure of $G$ (left) and $G'$ (right) in Subcase 2.2.2(b). }\label{F:2.2.2b}
\end{figure}

We note that $x$ cannot be a cut vertex of Type $X$ in $G'$, since if so, clearly $x_1$ and $x_2$ would also be cut vertices of Type $X$ in $G$. Moreover, if we have formed a triangle that was not present in the original graph $G$, it must be that this triangle uses the new vertex $x$, and (wolog) the vertices $w_1, w_2$. Note that in $G$, the edge $w_1w_2$ cannot use color $\delta$ or $\gamma$, as otherwise we would have more than three vertices incident to this color. Hence, the triangle created is rainbow. Therefore, $G'$ is good, so we may apply the inductive hypothesis to form a rainbow cycle $C'$ in $G'$, such that $G'\ba C'$ is good. 

If $C'$ does not use the vertex $x$, then $C'$ is a rainbow cycle in $G$, and clearly expanding $x$ back to a rectangle cannot introduce any cut vertices of Type $X$. If $C'$ does use the vertex $x$, then we may create a rainbow cycle $C$ in $G$ by replacing this vertex with the length two path $x_1y_1x_2$. As $C'$ does not use colors $\alpha$ or $\beta$, this is a rainbow cycle in $G$, and moreover, $G\ba C$ can be obtained from $G'\ba C'$ by subdividing the path through $x$ (which consists entirely of Type I vertices). Hence, $G\ba C$ is good.

\quad

\noindent\textbf{Subcase 2.2.2(c).} $y_1=w_2$, but $\{w_1, z_1\}$ is disjoint from $\{y_2, z_2\}$.

\quad

Note that $y_1=w_2$ must be a Type I vertex in $G$, and hence we have a rectangle $(x_1,y_1,x_2,v,x_1)$. Moreover, since $w_2$ is a Type I vertex, we cannot have the edge $w_2z_2$, and hence $z_2$ is also a Type I vertex, and no edges other than $w_2x_2$ and $z_2x_2$ can take color $\delta$. Form a new graph $G'$ by contracting this rectangle to a single vertex $x$, having neighbors $w_1, z_1, y_2, z_2$, and recolor these edges as $\gamma, \gamma, \beta, \beta$, respectively; see Figure \ref{F:2.2.2c}.

\begin{figure}[htp]
\includegraphics[scale=.8]{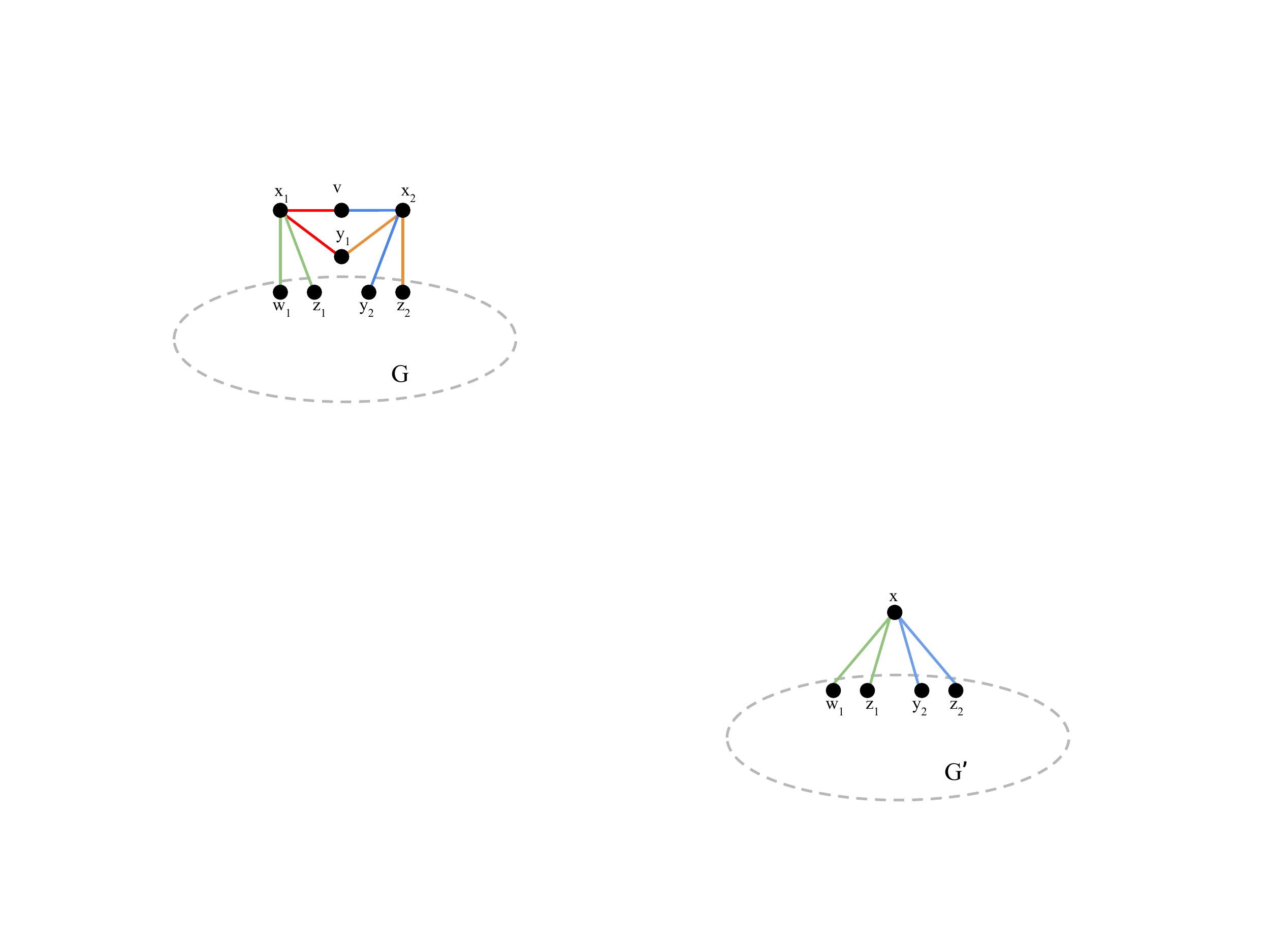}\hspace{.3in}
\includegraphics[scale=.8]{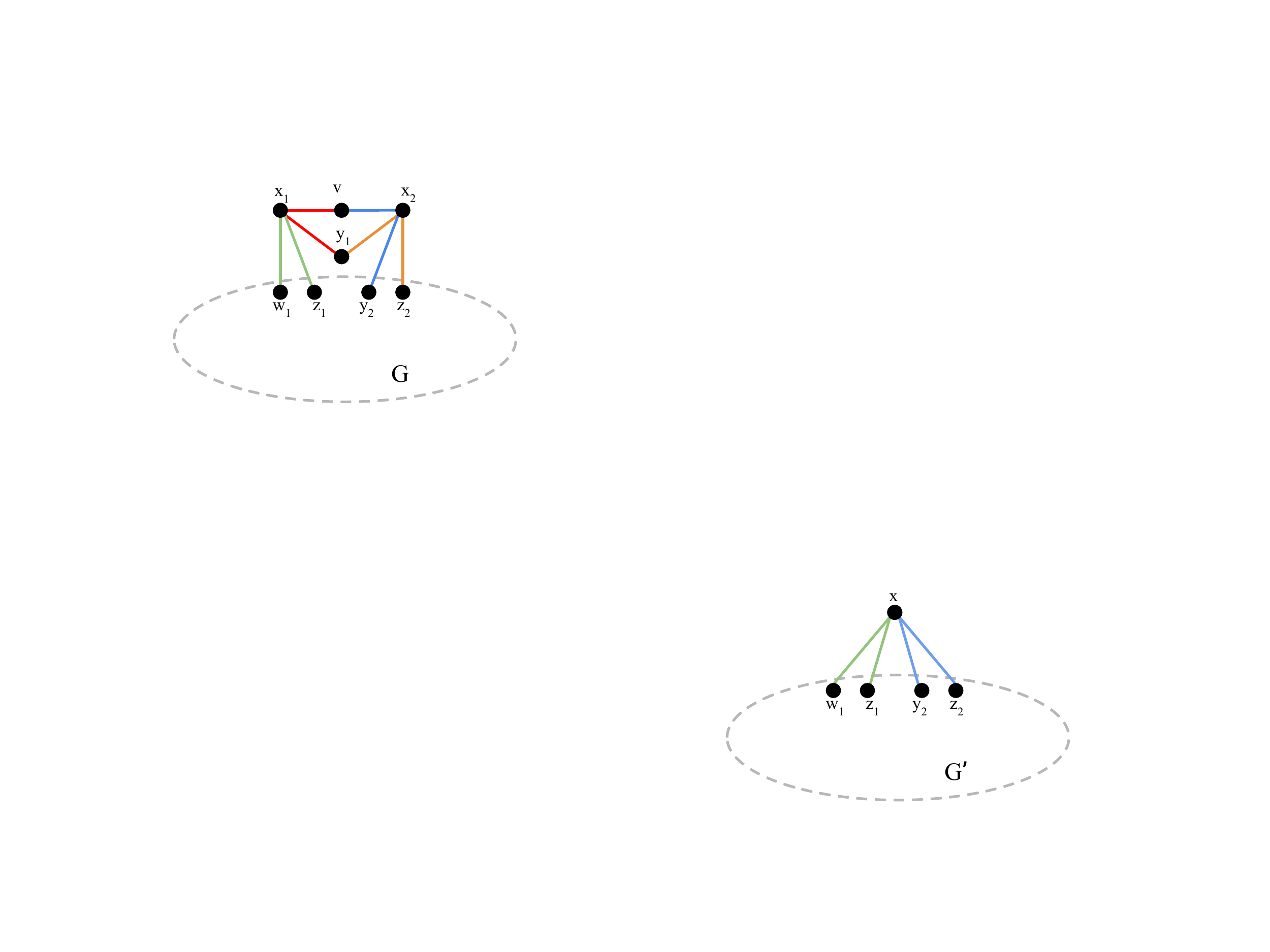}
\caption{The structure of $G$ (left) and $G'$ (right) in Subcase 2.2.2(c). }\label{F:2.2.2c}
\end{figure}

We note that $x$ cannot be a cut vertex of Type $X$ in $G'$, since if so, we have $x_1$ is a cut vertex of Type $X$ in $G$. Moreover, if we have formed a new triangle that was not present in the original graph $G$, it must be that this triangle uses the vertex $x$, and hence one of its other vertices is $w_1$ (wolog). But note that neither edge $w_1y_2$ nor $w_1z_2$ can use colors $\gamma$ or $\beta$, as we already have three vertices in $G$ incident to these color classes, and hence any new triangle formed must be rainbow.

Therefore, $G'$ is good, and by induction there exists a rainbow cycle $C'$ in $G'$ such that $G'\ba C'$ is good. If $C'$ does not use the vertex $x$, then $C'$ is a rainbow cycle in $G$, and clearly expanding $x$ back to a rectangle cannot introduce any cut vertices of Type $X$. If $C'$ does use the vertex $x$, then there are two possibilities. Either the edge $xy_2$ is used, or the edge $xz_2$ is used; note that one of these must be true as any rainbow cycle through $x$ must use color $\beta$. If the edge $xy_2$ is used, we shall form a rainbow cycle $C$ in $G$ by replacing this edge with the path $x_1y_1x_2y_2$, which replaces an edge of color $\beta$ with three edges, having colors $\alpha, \delta, \beta$, respectively. Moreover, we note that $G\ba C$ can be obtained from $G'\ba C'$ by subdividing the edge $xz_2$ with the vertex $v$, and recoloring appropriately. Hence, $G\ba C$ cannot contain a cut vertex of Type $X$, and thus $G\ba C$ is good.

On the other hand, if the edge $xz_2$ is used in $C'$, we similarly create a rainbow cycle $C$ in $G$ by replacing this edge with the path $x_1vx_2z_2$, having colors $\alpha, \beta, \delta$, respectively. As above, this will yield a good colored graph $G\ba C$.

\quad

\noindent\textbf{Subcase 2.2.2(d).} $y_1=w_2$ and $y_2=w_1$.

\quad

Note that in this case, as $y_1$ and $y_2$ are both Type I vertices in $G$, we have that $\{z_1, z_2\}$ is a cutset of $G$, and both $z_1$ and $z_2$ are vertices of Type I; see Figure \ref{F:2.2.2d} Moreover, we have a rainbow cycle $x_1y_1x_2y_2$. Clearly, we cannot create any cut vertices of Type $X$ by the removal of this rainbow cycle, as no cut vertices are introduced at all.

\begin{figure}[htp]
\includegraphics[scale=.8]{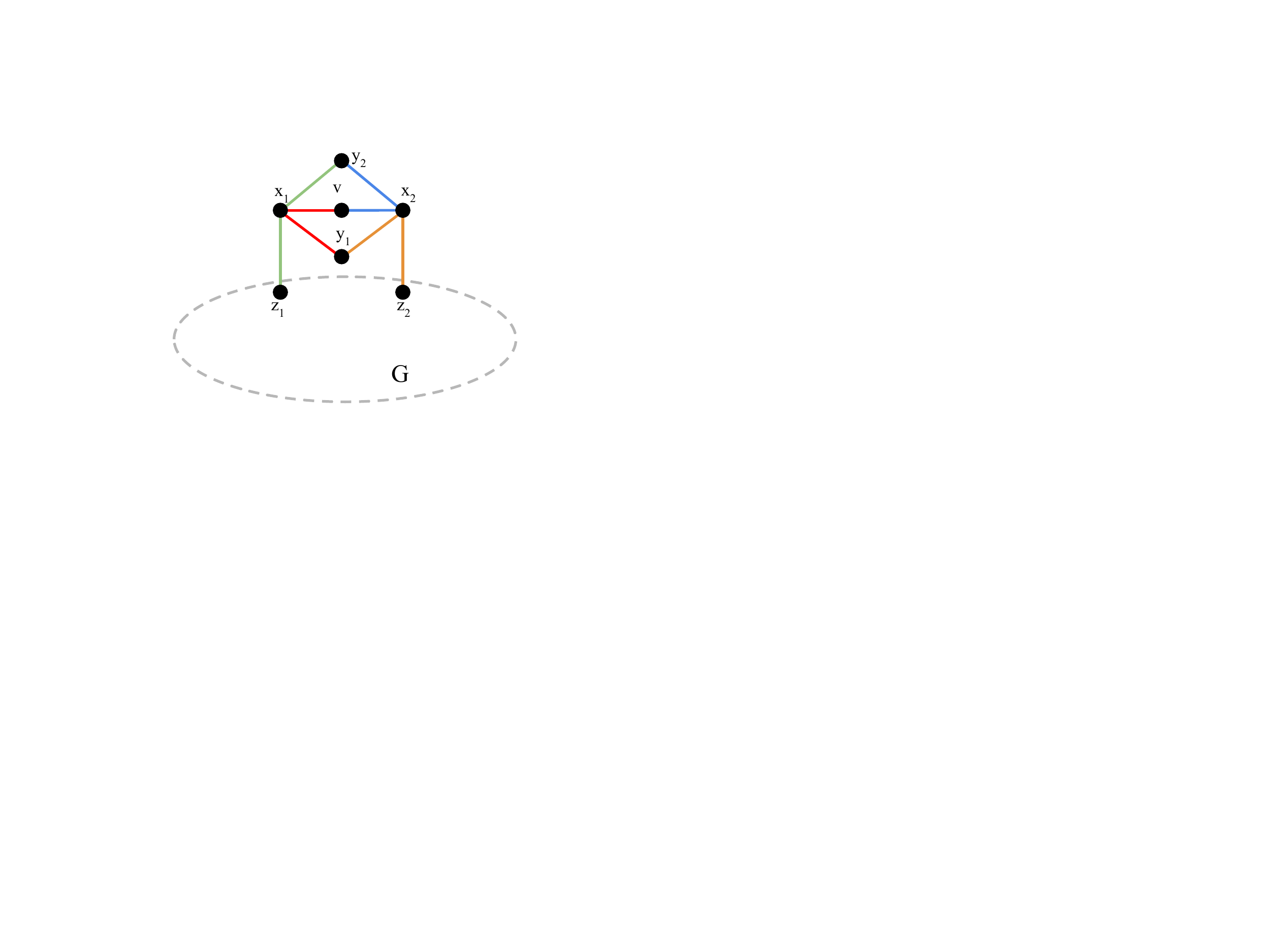}
\caption{The structure of $G$ in Subcase 2.2.2(d).}\label{F:2.2.2d}
\end{figure}

\quad

\quad

Therefore, if $G$ is a good or almost good graph having $n$ vertices and $m$ edges, we can find a rainbow cycle or almost-rainbow cycle, respectively, in $G$, such that $G\ba C$ is good or almost good, respectively. By then applying the induction hypothesis, we therefore have a decomposition of $G$ into cycles, such that at most one such cycle is almost rainbow, and the remainder are rainbow.

\end{proof}

\section{Conclusions and Conjectures}\label{S:conclusions}

We note, as mentioned in Section \ref{S:intro}, that our proof technique also resolves Conjecture \ref{C:Goddyn} in the case of 3-regular graphs. This conjecture is implied by our proof technique, as in Lemma \ref{type2}, the rainbow cycle chosen was entirely arbitrary. Hence, the first cycle we remove from $G$ is irrelevant, as removing any rainbow cycle from $L(G)$ will allow us to proceed with the remainder of the inductive proof.

However, it is unclear from this proof if Goddyn's conjecture holds in the general case. Hence, for graphs that are not 3-regular, Goddyn's conjecture remains unresolved.

In addition, we suspect that the technique used here could also be used to consider cycle $k$-covers for certain graphs, as follows. A cycle $k$-cover is a collection of cycles in $G$ for which every edge of $G$ is contained in exactly $k$ of the cycles.
\begin{conjecture}
Let $G$ be a $k$-regular graph, $(k-1)$-edge connected graph. Then there exists a list of cycles $\mathcal{C}$ in $G$ such that every edge of $G$ appears in exactly $k-1$ cycles.
\end{conjecture}

We note that in this case, the color classes in $L(G)$ are $k$-cliques, and every vertex in $L(G)$ is a member of two of these, and hence has $k-1$ neighbors in each of two incident color classes. Thus, any decomposition of the edges of $L(G)$ into rainbow cycles would produce a cycle $(k-1)$-cover, as suggested by the conjecture. The difficulty in generalizing to this case is likely to be found in how to generalize a cut vertex of Type $X$. The true condition here is that if $S$ is a cutset, and $G_1$ and $G_2$ are pseudoblocks corresponding to $S$, then no more than half the edges incident to vertices of $S$ in $G_1$ may have the same color. In the case of a 3-regular graph, this is automatically true provided that the cut set has at least two vertices, or the cutset has one vertex, but its two edges in $G_1$ are not of the same color. In the case of a higher regularity, this condition becomes more obtuse. 

There are many other standing conjectures related to cycle covers in graphs, and we do not doubt that similar techniques might be used to approach these conjectures. Many of these can be found in \cite{jaeger1985survey}.

In addition, there are many questions here relating, rather than to cycle covers, to decompositions of edge colored graphs into rainbow cycles or subgraphs. Here, we show that if $G$ is an edge colored graph, such that every color class of $G$ is a triangle, no two color classes share more than one vertex, and $G$ has no cut vertices, then the edges of $G$ may be decomposed into a set of disjoint rainbow cycles. This begs the question: under what conditions can such a decomposition be guaranteed?

\begin{question}Let $G$ be an even, edge-colored graph having no cut vertices. Under what conditions on the color classes of $G$ can it be assured that $G$ has an edge decomposition into disjoint rainbow cycles?
\end{question}

More specifically, we restricted here to the case that each color class has at most 3 vertices, and no two color classes share more than one vertex. Is it possible that under the second condition, a similar proof could be found for a graph having more vertices in each color class?

\begin{question}\label{decompd}
Let $G$ be an even, edge-colored graph having no cut vertices, and suppose that any two color classes share no more than one vertex. Moreover, suppose that the subgraph induced on each color class contains at most $d$ vertices. For what values of $d$ can we guarantee that the edges of $G$ can be decomposed into rainbow cycles?
\end{question}

This paper answers Question \ref{decompd} in the case that $d=3$ and $n\geq 6$. The case that $d=2$ is trivial; every edge of $G$ has a unique color. However, a generalization of the proof in this article to the case that $d\geq 4$ is not apparent. Further, it may be that a precise value of $d$ will depend upon $n$; in the case that $d=3$ analyzed in this proof, we have at least $\frac{2}{3}n$ distinct colors available. It may be that there is a function $d=d(n)$ such that the decomposition will be possible in the case of any even edge-colored graph on $n$ vertices, having at most $d(n)$ vertices in each color class.

\section{Acknowledgements}
The author would like to extend sincere gratitutde to Paul Horn for his thoughts in the development of this approach, and to SOMEBODY for proofreading this manuscript.

\bibliographystyle{siam}
\bibliography{bib_items}

\end{document}